\documentclass[12pt]{amsart}
\usepackage{amsmath, amsthm, amssymb}
\usepackage{fullpage}
\usepackage{xcolor}
\usepackage{hyperref}
\usepackage{soul}
\usepackage{enumitem}

\usepackage{mathtools}

\mathtoolsset{showonlyrefs}

\newtheorem{theorem}{Theorem}
\newtheorem{lemma}{Lemma}
\newtheorem{proposition}[lemma]{Proposition}
\newtheorem{corollary}[lemma]{Corollary}
\newtheorem{definition}[lemma]{Definition}
\newtheorem{remark}[lemma]{Remark}
\newtheorem{conjecture}{Conjecture}
\numberwithin{lemma}{section}

\numberwithin{equation}{section}

\newcommand{\R}{{\mathbb R}}
\newcommand{\C}{{\mathbb C}}
\newcommand{\N}{{\mathbb N}}
\newcommand{\Z}{{\mathbb Z}}

\renewcommand{\R}{\mathbb R}

\newcommand{\bM}{\mathbf M}
\newcommand{\bP}{\mathbf P}

\newcommand{\bR}{\mathbf R}

\newcommand{\bI}{\mathbf I}
\newcommand{\bJ}{\mathbf J}
\newcommand{\bK}{\mathbf K}

\newcommand{\bu}{{\bar u}}
\newcommand{\bv}{{\bar v}}
\newcommand{\bfu}{{\mathbf u}}

\newcommand{\Ns}{N^\sharp}

\newcommand{\la}{\langle}
\newcommand{\ra}{\rangle}

\newcommand{\ol}{\overline}
\newcommand{\ms}{M^\sharp}
\newcommand{\ps}{P^\sharp}

\newcommand{\calR}{\mathcal{R}}
\newcommand{\UH}{{U\!H}}

\begin{document}

\title{Global solutions for cubic quasilinear ultrahyperbolic Schr\"odinger flows}

\author{Mihaela Ifrim}
\address{Department of Mathematics, University of Wisconsin, Madison}
\email{ifrim@wisc.edu}

\author{Ben Pineau}
\address{Courant Institute for Mathematical Sciences\\
New York University
} \email{brp305@nyu.edu}

\author{ Daniel Tataru}
\address{Department of Mathematics, University of California at Berkeley}
\email{tataru@math.berkeley.edu}

\begin{abstract}

In recent work, two of the authors proposed a 
broad global well-posedness conjecture
for cubic quasilinear dispersive equations in two space dimensions, which asserts that global well-posedness and scattering holds for small initial data in Sobolev spaces. As a first validation they  proved the conjecture for  quasilinear Schr\"odinger flows.

 In the present article we expand the reach of these 
ideas and methods to the case of quasilinear ultrahyperbolic Schr\"odinger flows,  which is the first example  with a nonconvex dispersion relation.

The study of local well-posedness for this class of problems, in all dimensions, was initiated  in pioneering work of Kenig-Ponce-Vega for localized initial data, and then continued by Marzuola-Metcalfe-Tataru (MMT) and Pineau-Taylor (PT) for initial data in Sobolev spaces in the elliptic and non-elliptic cases, respectively. 

Our results here mirror the earlier results in the elliptic case: (i) a new, potentially sharp  local well-posedness result
in low regularity Sobolev spaces, one 
derivative below MMT and just one-half derivative above scaling, (ii) a small data global well-posedness and scattering result at the same regularity level. One key novelty in this setting is the introduction of a new family of interaction Morawetz functionals which are suitable for obtaining bilinear estimates in the ultrahyperbolic setting. We remark that this method appears to be robust enough to potentially be of use in a large data regime when the metric is not a small perturbation of a Euclidean one. 

\bigskip

\end{abstract}


\maketitle

\setcounter{tocdepth}{1}
\tableofcontents


\section{Introduction}
In this article, we consider the problem of local and global well-posedness in the small data, low regularity regime for general quasilinear Schr\"odinger equations of the form
\begin{equation}
\label{qnlsint}
\left\{ \begin{array}{l}
i u_t + g^{jk}(u,\nabla u) \partial_j \partial_k  u = 
N(u,\nabla u) , \quad u:
\R \times \R^n \to \C \\ \\
u(0,x) = u_0 (x).
\end{array} \right. 
\end{equation}
Here, the metric $g$ (assumed to be real-valued and symmetric) and nonlinearity $N$ are assumed to be smooth functions of their arguments, which are at least of cubic order. This means that
\begin{equation}\label{cubiccondition}
g(y,z)=g(0,0)+\mathcal{O}(|y|^2+|z|^2),\hspace{5mm} N(y,z)=\mathcal{O}(|y|^3+|z|^3),\hspace{2mm}\text{near}\hspace{2mm} (y,z)=(0,0).
\end{equation}
The present work serves as a natural continuation of the prior works \cite{IT-qnls} and \cite{IT-qnls2} of the first and third authors, which establish the global well-posedness conjecture (first formulated in \cite{IT-global} in the context of semilinear Schr\"odinger flows) for a large class of quasilinear Schr\"odinger equations. These results consider the one-dimensional case, and higher dimensional cases, respectively. Importantly, these works were devoted exclusively to the case when the principal operator $g^{jk}\partial_j\partial_k$ is elliptic, meaning that  
\begin{equation}\label{elliptic}
c^{-1}|\xi|^2\leq g^{ij}(x)\xi_i\xi_j\leq c|\xi|^2   
\end{equation}
for some uniform (in $x$ and $\xi$) constant $c>0$. In the present work we consider the 
ultra-hyperbolic case in two and higher dimensions, where the metric is not assumed to be positive-definite; more precisely, we have the following definition.
\begin{definition}
We call the problems \eqref{qnlsint}
ultrahyperbolic if $g$ is a symmetric, uniformly non-degenerate matrix of indefinite signature. 
\end{definition}
By uniformly non-degenerate, we mean that there exists a constant $c>0$ such that for all $x,\xi$, there holds
\begin{equation}\label{nondegen}
c^{-1}|\xi|\leq |g^{ij}(x)\xi_j|\leq c|\xi|.
\end{equation}
There are many models of nonlinear Schr\"odinger equations of ultrahyperbolic type that are of physical interest. Some important examples arise in water waves (see, for instance \cite{davey1974three}).
Others arise as certain completely integrable models \cite{ishimori1984multi, zakharov1980degenerative}. We also refer to the discussion in \cite{ghidaglia1993nonelliptic} which surveys several other interesting equations. A recent example coming from fluid mechanics stems from the ideal Hall and
electron MHD equations. Near nondegenerate perturbations of uniform magnetic fields, the equation (for the magnetic field $B$) behaves at leading order like an ultrahyperbolic Schr\"odinger type equation with principal operator of the
form $B\cdot\nabla |\nabla|$. See \cite{jeong2022cauchy, 2024arXiv240206278J} for results pertaining to both well-posedness and ill-posedness for this system (depending on the regime considered).

For the general class of equations we consider here, both the ultrahyperbolic and elliptic variants of this problem have received considerable attention over the past several decades. For the local theory, the pioneering works in this direction are due to Kenig, Ponce, Rolvung and Vega \cite{kenig2005variable, kenig2006general, kenig1998smoothing, kenig2004cauchy}, which (taken together) establish local well-posedness for localized data in high regularity Sobolev spaces (also in the large data regime with a necessary supplementary nontrapping condition on the metric). In the case of positive-definite metrics, these results were extended by Marzuola, Metcalfe and Tataru to allow for data in low regularity, translation invariant Sobolev-type spaces (including the vanilla $H^s$ scale for cubic problems) in \cite{MMT1, MMT2} for small data, and later in \cite{MMT3} in the large data regime. These results were extended to encompass the full class of ultrahyperbolic Schr\"odinger equations in the recent work of Pineau and Taylor in \cite{pineau2024low}. A common theme of these previous results is the reliance on the local smoothing properties of the equation (although with a weaker functional framework in the low regularity regime).
\medskip

In sharp contrast, the approach used in \cite{IT-qnls} and \cite{IT-qnls2} is significantly different from the above mentioned works, and instead establishes bilinear estimates via a robust interaction Morawetz analysis, coupled with suitable normal form corrections, implemented via modified energies. The main aim of this article is to extend the reach of these ideas to the ultrahyperbolic setting, in turn dramatically improving both the known (small data) local and global results for this problem to potentially optimal Sobolev thresholds. More precisely, our goals in this article are to address the following three points:
\begin{enumerate}[label=\roman*)]
    \item To establish the higher dimensional version of the global well-posedness conjecture (proposed in \cite{IT-qnls2}) in the context of cubic, ultrahyperbolic Schr\"odinger flows. 
    \item To sharply improve the local well-posedness theory for cubic quasilinear Schr\"odinger flows
    in higher dimension, one full
derivative below the earlier results of \cite{MMT2} and \cite{pineau2024low} and just one half derivative above scaling, in the small data regime. Depending on the dimension, we conjecture these results are generically optimal, in terms of Sobolev regularity.
    \item To introduce a novel interaction Morawetz analysis which is robust enough to establish bilinear estimates for variable coefficient, non-elliptic Schr\"odinger equations.  
    
\end{enumerate}

\subsection{An overview of the global well-posedness conjectures}
To better motivate the present work, we now give a brief overview of the global well-posedness conjectures formulated by the first and third authors in the recent work \cite{IT-global}. In the one-dimensional setting (for both quasilinear and semilinear equations), the conjecture can be roughly phrased as follows
\begin{conjecture}[Non-localized data defocusing GWP conjecture]\label{c:nld}
One dimensional dispersive flows on the real line with cubic defocusing nonlinearities and small initial data have global in time, scattering solutions.
\end{conjecture}
The main results of \cite{IT-global} and \cite{IT-qnls} provided the first validation for this conjecture in the semilinear, and quasilinear settings, respectively. These results represented the first global in time well-posedness results of this type at this level of generality. More precisely, in these articles,  it was shown that if the equation has \emph{phase rotation symmetry} and is \emph{conservative}  and \emph{defocusing}, then small data in Sobolev spaces give rise to global, scattering solutions.   Here we remark that scattering should be interpreted in a suitable weak sense, meaning that the solution satisfies a global $L^6$ Strichartz bound as well as bilinear $L^2$ bounds. Classical scattering is precluded due to the strong nonlinear effects for this problem.

We crucially note that the defocusing condition is imperative for the global result in one dimension. In the focusing case, the existence of small solitons serves as an obstruction to obtaining global, scattering solutions. Nevertheless, an analogue of the above conjecture which was proposed by two of the authors suggests a potentially optimal lifespan bound for solutions.

\begin{conjecture}[Non-localized data long time well-posedness conjecture]\label{c:nld-focusing}
One dimensional \\ dispersive flows on the real line with cubic conservative nonlinearities and  initial data of size $\epsilon \ll 1$ have  
long time solutions on the $\epsilon^{-8}$ time scale.
\end{conjecture}

The conservative condition (which will be relevant also in the two-dimensional setting) is a natural assumption. Roughly, one can interpret this as preventing nonlinear ODE blow-up of solutions with wave packet localization.
\medskip

In contrast to the one-dimensional setting, in two and higher dimensions, solutions exhibit stronger dispersive decay. In particular, the $L^4$ space-time norm of the solutions will play the leading role, as opposed to $L^6$ in one-dimension. We remark that in the semilinear case, 
the small data problem may be treated in a relatively standard fashion (even in the non-elliptic case) by Strichartz estimates, which yields global well-posedness and scattering. The focusing/defocusing assumption is not required in this setting, and only plays a role in the large data case, see e.g. Dodson's results in \cite{D,D-focus}.

Therefore, the most interesting case to consider is that of quasilinear flows, where Strichartz estimates are not typically directly applicable. With this in mind, the recent work \cite{IT-qnls2} posed the following higher dimensional analogue of the above conjectures.

\begin{conjecture}[2D non-localized data  GWP conjecture]\label{c:nld2}
Cubic quasilinear dispersive flows with small data give rise to global in time, scattering solutions in two and higher dimension.
\end{conjecture}

In contrast to 1D, here we interpret scattering in the classical fashion, but with the caveat that, due to the nature of quasilinear  
well-posedness, we only expect convergence to a linear evolution in a weaker topology. In the prior article \cite{IT-qnls2}, this was established for all cubic quasilinear Schr\"odinger flows with positive-definite metrics. In the present paper, we will establish this conjecture for merely non-degenerate metrics, thus establishing the global well-posedness conjecture for the full class of (non-degenerate) quasilinear, cubic Schr\"odinger flows. When precisely outlining our results below, we will impose no additional structural assumptions in dimension three and higher, nor in dimension two at sufficiently high Sobolev regularity. However, in order to obtain the lowest possible global well-posedness Sobolev threshold in dimension two, we will additionally require the flow to be conservative (a condition which we will describe more precisely below).
\\

Although we will study this conjecture in all dimensions, the most interesting case by far is two dimensions, where the dispersive character is weakest. Prior to the work of the first and third author in \cite{IT-qnls2}, no small-data global well-posedness results existed in this setting. The reason for this is morally due to the fact that in two space dimensions $L^4_{t,x}$ is a sharp Strichartz norm, whereas in higher 
dimensions there is room to interpolate. In fact, in dimension three and higher, global well-posedness was known in some specific model cases, though only at higher regulatity; see for instance \cite{HLT}. 

\subsection{ Quasilinear Schr\"odinger flows: local solutions}

The general form for a  quasilinear Schr\"odinger flow  is
\begin{equation}\tag{UDQNLS}
\label{dqnls}
\left\{ \begin{array}{l}
i u_t + g^{jk}(u,\nabla u ) \partial_j \partial_k u = 
N(u,\nabla u) , \quad u:
\R \times \R^n \to \C \\ \\
u(0,x) = u_0 (x),
\end{array} \right. 
\end{equation}
with a metric $g$ which is a real valued, symmetric function and a source term $N$ which is a complex valued
smooth function of its arguments satisfying the vanishing condition \eqref{cubiccondition}. 
Here smoothness is interpreted in the real sense, but if $g$ 
and $N$ are analytic, then they can also be thought of as separately (complex) analytic functions of $u$ and $\bu$, which is the interpretation we use in the present paper.

In a similar vein, we also consider the  problem (written in divergence form for simplicity)
\begin{equation}\tag{UQNLS}
\label{qnls}
\left\{ \begin{array}{l}
i u_t + \partial_jg^{jk}(u) \partial_k  u = 
N(u,\nabla u) , \quad u:
\R \times \R^n \to \C \\ \\
u(0,x) = u_0 (x),
\end{array} \right. 
\end{equation}
where $N$ is  at most quadratic in $\nabla u$. This can be viewed as a differentiated form of \eqref{dqnls}. Therefore, while we will state our results for both
of these flows, we will only give complete proofs for \eqref{qnls} to isolate the key points.

The first objective of this article is to study the question of local well-posedness in Sobolev spaces. Even in the small data regime, this comes with a multitude of difficulties due to the infinite speed of propagation (in sharp contrast to corresponding questions for nonlinear wave equations). For instance, in problems with only quadratic nonlinearity $N$, a well-known generic necessary condition for $L^2$ well-posedness of the corresponding linearized equation (around a solution $u$) is essentially that $u$ be integrable along the Hamilton flow corresponding to the principal differential operator $\partial_jg^{jk}\partial_k$. Even for flat metrics (when the Hamilton curves are straight lines), this condition fails for initial data in $H^s$. See for instance, the classical works of Takeuchi~\cite{MR1191488}, Mizohata~\cite{MR0860041} and Ichinose~\cite{MR0948533}.

Problems of this type were first studied in the semilinear setting in the work of Kenig-Ponce-Vega~\cite{MR1230709} for  small data and later Hayashi-Ozawa~\cite{MR1255899} and Chihara~\cite{MR1344627} for large data, making use of ideas of Doi~\cite{MR1284428}.
Finally, Kenig-Ponce-Vega~\cite{MR1660933} also considered the (semilinear) ultrahyperbolic case. In all of these results, the data is localized. In contrast, the first result in Sobolev-type spaces with a weaker translation invariant decay assumption (imposed in view of the necessary Mizohata-type integrability condition mentioned above) was due to Bejenaru-Tataru~\cite{MR2443925}.

The quasilinear case also has a rich history. The first problems that were considered in this case were either one-dimensional or tailored to specific model problems. See for instance, de Bouard-Hayashi-Naumkin-Saut~\cite{BNPS}, Colin~\cite{Co}, Poppenberg~\cite{Pop} and Lin-Ponce~\cite{Lin-Ponce}.

The pioneering local well-posedness results 
for general quasilinear problems are due to Kenig-Ponce-Vega and Kenig-Ponce-Rolvung-Vega in the series \cite{kenig2005variable, kenig2006general, kenig1998smoothing, kenig2004cauchy} which collectively considered both the elliptic and ultrahyperbolic cases. These results are posed in high regularity, weighted Soboelv spaces (i.e. for smooth, localized data). A natural next step was 
to study the local well-posedness question in low regularity Sobolev spaces with the weaker translation-invariant decay condition mentioned above. This was accomplished 
by Marzuola-Metcalfe-Tataru, first for small data  in \cite{MMT1,MMT2}, and then for large data in \cite{MMT3}; the latter 
article considers only the elliptic case,
while the ultra-hyperbolic case was resolved (in the same scale of spaces) in the recent work of Pineau-Taylor~\cite{pineau2024low}. See also the related to the work of Jeong-Oh~\cite{2024arXiv240206278J}, which considered the well-posedness problem for the electron MHD equations (which can be interpreted as a (possibly degenerate)  nonlocal ultrahyperbolic Schr\"odinger type equation). We also mention another recent result of
Shao-Zhou~\cite{2023arXiv231102556S} which (among other things) recovers the results of \cite{KPV,MMT2} by an alternative approach.
\medskip 

A very important distinction made in these last papers was between quadratic and cubic (or higher order) nonlinearities. Due to the Mizohata-type necessary condition mentioned above, it is only in the cubic case that local well-posedness 
is proved in classical $H^s$ Sobolev spaces. Another important distinction is between small and large data; in the latter case, the Hamilton curves of $\Delta_g$ can be confined (indefinitely) to a compact set, which causes large energy growth to occur on arbitrarily short time scales. To overcome this, one has to impose a nontrapping condition on the (initial data) metric. In the small data regime, the Hamilton trajectories (for sufficiently regular data) are close to straight lines, so that non-trapping is automatic.

In addition to the general setting mentioned above, many other related models have also been studied. For a non-exhaustive sample, we refer the reader to Colin~\cite{MR1886962}, Chemin-Salort~\cite{MR3336355} and Huang-Tataru~\cite{HT1,HT2} and also the 
small data global result of Huang-Li-Tataru~\cite{HLT}; the latter three articles consider the skew mean curvature flow, which can be viewed as a quasilinear Schr\"odinger evolution with respect to a suitable gauge.

In order to work in vanilla $H^s$ Sobolev spaces, we will restrict our attention to the cubic case, meaning that the metric $g$ and nonlinearity $N$ are assumed to satisfy the vanishing condition \eqref{cubiccondition}. For ease of comparison with the results in the present article, we briefly summarize the results of MMT and PT below. In the case of positive-definite metrics we have the following
\begin{theorem}[cubic nonlinearities, positive-definite metric \cite{MMT2,MMT3}]\label{t:regular}
For a positive-definite metric $g$, the nD cubic problem \eqref{qnls} is locally well-posed 
for nontrapping data in $H^s$ for $s > \frac{n+3}2$, and the cubic problem \eqref{dqnls} is locally well-posed 
in $H^s$ for $s > \frac{n+5}2$.
\end{theorem}
We also have the following generalization to the case of an indefinite metric, 
\begin{theorem}[cubic nonlinearities, indefinite metrics \cite{pineau2024low}]\label{t:regular1}
 For a non-degenerate metric $g$, the nD cubic problem \eqref{qnls} is locally well-posed 
for nontrapping data in $H^s$ for $s > \frac{n+3}2$, and the cubic problem \eqref{dqnls} is locally well-posed 
in $H^s$ for $s > \frac{n+5}2$.
\end{theorem}
As mentioned earlier, the nontrapping assumption is only needed for large data, and is redundant for small data. In the case of a positive-definite metric, these results were improved in \cite{IT-qnls} (in 1D) and \cite{IT-qnls2} (in higher dimensions) up to the natural, (conjectured to be optimal) Sobolev threshold in the small data regime:
\begin{theorem}[cubic nonlinearities, positive-definite metrics \cite{IT-qnls}, \cite{IT-qnls2}]\label{t:regularnew}
For a positive-definite metric, the nD cubic problem \eqref{qnls} is locally well-posed 
for small data in $H^s$ for $s > \frac{n+1}2$, and the cubic problem \eqref{dqnls} is locally well-posed 
in $H^s$ for $s > \frac{n+3}2$.
\end{theorem}
One of the primary objectives of this paper is to remove the remove the positive-definite requirement on the metric. Our main local result is the following: 
\begin{theorem}\label{t:local3+}
Let $n \geq 2$. For a non-degenerate metric, the nD cubic problem \eqref{qnls} is locally well-posed for small data in $H^s$ for $s >  \frac{n+1}2$, and the nD cubic problem \eqref{dqnls} is locally well-posed for small data in $H^s$ for $s >  \frac{n+3}2$.
\end{theorem}
This result represents a significant generalization of the corresponding result in \cite{IT-qnls2} for positive-definite metrics. Moreover, this also yields a major improvement in regularity over the result of \cite{pineau2024low} in the small data regime. 
\medskip 

As in the previous article \cite{IT-qnls2}, we interpret well-posedness in an enhanced Hadamard sense, which in this article includes the following suite of important properties 
:
\begin{itemize}
    \item existence of solutions in  $C([0,T]; H^s)$.
    \item uniqueness of regular solutions.
    \item uniqueness of rough solutions as 
    uniform limits of regular solutions.
    \item continuous dependence on the initial data.
    \item Lipschitz dependence of the solutions on the initial data in a weaker topology ($L^2$)
    \item higher regularity: more regular data yields more regular solutions.
\end{itemize}

We importantly remark that the above results do not require any additional structure for the nonlinearity, other than that it is cubic (in contrast to some of the 2D global results below). 
\medskip

To better understand the numerology of the above result, it is instructive to compare the regularity restriction in Theorem~\ref{t:local3+} with the analogous restriction imposed by the leading order scaling symmetry for these problems. This latter threshold corresponds to the exponents $s_c = \frac{n}2$ 
for \eqref{qnls}, respectively $s_c = \frac{n+2}2$ for \eqref{dqnls}.
In other words, our result is only $1/2$ derivative above scaling, which is similar to the earlier one dimensional result. One might therefore be tempted to suspect that our result could still be improved by $\frac{1}{2}$ derivatives. Nevertheless, we conjecture that our result is sharp (except possibly at the endpoint).
\begin{conjecture}
The range of $s$ in Theorem~\ref{t:local3+}    is
sharp for generic problems, except possibly for the endpoint.
\end{conjecture}
This obstruction seems to be related to non-trapping, which should be highly unstable for $s$ below our thresholds, even for small data. Correspondingly, we also conjecture that the result of Theorem~\ref{t:local3+} holds for large data under a suitable non-trapping assumption on the initial data.
 \begin{conjecture}
The result in Theorem~\ref{t:local3+} holds for all non-trapping initial data.
\end{conjecture}
Thanks to the well-posedness result in Theorem~\ref{t:regular1}, we will be able to show that for any regular data (say, $u_0\in H^{\sigma}$ for some $\sigma> \frac{n+5}{2}$) which satisfies the smallness condition in the weaker norm (here $s$ satisfies the conditions of Theorem~\ref{t:regularnew}).
\begin{equation*}
\|u_0\|_{H^s}\lesssim \epsilon
\end{equation*}
generates a solution $u$ in the more regular class $C([0,T];H^{\sigma})$. Here, the lifespan of $u$ apriori depends on the size of the $H^{\sigma}$ data of $u_0$, which could be large. Therefore, the main task in proving our results is to improve the a-priori bounds
on these solutions, so that the $H^{\sigma}$ lifespan depends only on the $H^{s}$ size of the data. We will also aim to carefully propagate the smallness of the $H^s$ data in order to ensure the metric $g$ remains uniformly non-trapping. Once we have constructed $H^{\sigma}$ solutions with uniform (in the size of the $H^s$ data) lifespan, we will obtain $H^{s}$ solutions as 
unique uniform limits of regular solutions. To this end, we will also need a suitable suite of estimates for the linearized 
flow as well as the associated paradifferential linear flow at the level of $L^2$.
\subsection{ Ultrahyperbolic Schr\"odinger flows:  global solutions}

Now, we turn our attention to the question of global well-posedness for the for small initial data in $H^s$, in the spirit of Conjecture~\ref{c:nld2}. We begin with a discussion of the 1D counterpart of this conjecture which was proved in \cite{IT-qnls}. In this setting, we require some natural assumptions on the nonlinearity in order to guarantee long time and  global in time results. We recall some of these conditions below which will be relevant (but in much weaker forms) for the higher dimensional results (specifically in 2D).
\begin{definition}
We say that the equation \eqref{qnls}/\eqref{dqnls}
has  \emph{phase rotation} symmetry if it is 
invariant with respect to the transformation 
$u \to u e^{i\theta}$ for $\theta \in \R$.
\end{definition}

This assumption ensures that the nonlinearity (after Taylor expanding) only has 
odd terms which contain multilinear expressions 
in $u$ and $\bar u$ (and also their derivatives) where there is exactly one more factor of $u$ than $\bar u$. In the 1D result, it suffices to impose this condition only for the cubic part of the nonlinearity. 

To heuristically describe the next structural assumptions, it is instructive to consider the following semilinear model with phase rotation symmetry. 
\begin{equation}\label{eq-cubic}
i u_t + \partial^2_x u = C(u,\bar u, u) + \text{higher order terms}.
\end{equation}
Here $C$ is a translation invariant trilinear form with symbol $c(\xi_1,\xi_2,\xi_3)$. See Section~\ref{s:notations} for a description of this notation.

\begin{definition}
We say that the 1D equation \eqref{qnls}/\eqref{dqnls}
is \emph{conservative} if  the cubic component of the nonlinearity satisfies 
\[
c(\xi,\xi,\xi), \ \partial_{\xi_j}
c(\xi,\xi,\xi) \in \R .
\]
\end{definition}
It is useful to note that such a condition is satisfied if the equation has a coercive conservation law, but the converse is not necessarily true. The final structural assumption used in one dimension (which will not be required in any form for our higher dimensional results) is that the problem is defocusing
\begin{definition}
We say that the equation \eqref{qnls}/\eqref{dqnls}
is \emph{defocusing} if  the cubic component of the nonlinearity satisfies the following bound on the diagonal
\[
c(\xi,\xi,\xi) \gtrsim 1+ \xi^2.
\]
\end{definition}
With the above, we can now recall the one-dimensional result

\begin{theorem}\label{t:global}
 Consider a 1D \eqref{qnls}/\eqref{dqnls} problem, which is cubic, phase rotation invariant and defocusing.
For initial data $u_0$ which is small in $H^s$,
(where $s > 1$ for \eqref{qnls} and $s > 2$ for \eqref{dqnls}), the corresponding solutions are global in time, and satisfy the global energy bound
\begin{equation}
\| u\|_{L^\infty_t H^s_x} \lesssim \| u_0\|_{H^s_x}.
\end{equation}
\end{theorem}
Now we turn our focus toward the higher dimensional problem. These differ from the 1D case in three key ways:

\begin{enumerate}
    \item The dispersive effects are stronger in higher dimensions which is one sense in which the problem becomes simpler. In particular, the defocusing assumption is no longer necessary. 

    \item There are far more cubic resonant interactions in nD; This is a sense in which the problem becomes harder, as it obstructs to an extent the use of normal form/modified energy type arguments.

    \item Another problem (which is at the heart of this paper) is due to the indefinite nature of the metric. This causes two additional difficulties. First, there are more resonant interactions at high frequency. This is due to the fact that the null set $\{\xi: g^{ij}\xi_i\xi_j=0\}$ can be unbounded in the indefinite case. See Section 2 for a discussion about this. The second difficulty is in establishing a strong enough $L^4_{t,x}$ Strichartz bound to control solutions globally in time. The interaction Morawetz functional used in the definite case \cite{IT-qnls2} does not yield a sufficiently coercive bound for this purpose. See Section 5 for a discussion in the linear setting.
\end{enumerate}

In light of the first feature above, we are able to remove both the phase rotation invariance, the conservative assumption and the defocusing condition. Without phase rotation the expansion in 
\eqref{eq-cubic} acquires extra terms,
\begin{equation}\label{eq-cubic2}
i u_t + \partial_jg^{jk}\partial_k u = C(u,\bar u, u)  
+ C^{np} + 
\text{higher order terms},
\end{equation}
where $C^{np}$ generically denotes the cubic part of the nonlinearity which is not phase-rotation invariant. The particular structure of the above terms does not play a significant role in dimension $n \geq 3$, but it does become important when $n = 2$. In the low regularity regime, we will need to introduce the following weak conservative condition to handle the worst case resonant interactions arising from $C(u,\bar u, u)$. As in \cite{IT-qnls2}, we have
\begin{definition}\label{d:conservative2}
We say that the 2D equation \eqref{qnls}/\eqref{dqnls}
is \emph{conservative} if  the phase-rotation invariant cubic component of the nonlinearity satisfies 
\[
c(\xi,\xi,\xi) \in \R, \qquad \xi \in \R^2.
\]
\end{definition}
Compared to the definite case in \cite{IT-qnls2}, we emphasize that the non-rotation invariant part $C^{np}$ has additional resonant interactions at high frequency due to the fact that the null set corresponding to $g$ can be unbounded (see Section 2). This adds an additional challenge in our setting. Nevertheless, remarkably, we will not need to impose any additional structural assumptions compared to the definite case to deal with these interactions (rather, we will be able to treat them with the novel family of Morawetz estimates that we prove later).
\medskip 

With the preliminaries out of the way, we are now ready to formulate our global results.
We will distinguish between the two and higher-dimensional cases. First, we state the higher dimensional result where we work at the same regularity level as in the local well-posedness result in Theorem~\ref{t:local3+}: 

\begin{theorem}\label{t:global3}
Let $n \geq 3$.
 a) Assume that the equation \eqref{qnls} is cubic, and that the initial data $u_0$ is small in $H^s$, with $s > \frac{n+1}2$,
 \begin{equation}
 \|u_0\|_{H^s} \leq \epsilon \ll 1.   
 \end{equation}
Then the solutions are global in time,  satisfy the uniform bound 
\begin{equation}
\| u\|_{L^\infty_t H^s_x} \lesssim \epsilon,
\end{equation}
and scatter at infinity, in the sense that
\begin{equation}
u_\infty = \lim_{t \to \infty} e^{-it \Delta_{g(0)}} u(t)     \qquad \text{in } H^{s}
\end{equation}
exists and has a continuous dependence on the initial data in the same topology.

b) The same result holds for  the equation \eqref{dqnls} for $s > \frac{n+3}2$. 
\end{theorem}

Next, we state the analogous result in two space dimensions (with the same structural assumptions on the nonlinearity), but for a more restrictive range of Sobolev exponents, namely $1/2$ derivative above the local well-posedness threshold in Theorem~\ref{t:local3+}:

\begin{theorem}\label{t:global2}
Let $n = 2$.
 a) Assume that the equation \eqref{qnls} is cubic, and that the initial data $u_0$ is small in $H^s$, with $s \geq \frac{3}2+\frac12=2$,
 \begin{equation}
 \|u_0\|_{H^s} \leq \epsilon \ll 1.   
 \end{equation}
Then the solutions are global in time, and satisfy the uniform bound
\begin{equation}
\| u\|_{L^\infty_t H^s_x} \lesssim \epsilon,
\end{equation}
and scatter at infinity, in the sense that
\begin{equation}
u_\infty = \lim_{t \to \infty} e^{-it \Delta_{g(0)}} u(t)     \qquad \text{in } H^{s}
\end{equation}
exists and has a continuous dependence on the initial data in the same topology.

b) The same result holds for  the equation \eqref{dqnls} for $s \geq \frac{5}2 +\frac12=3$. 
\end{theorem}
\begin{remark}
In the larger range $s>\frac{3}{2}$, our local well-posedness result will show that solutions have a lifespan of at least $T\lesssim \epsilon^{-6}$. See Theorem~\ref{t:local-fe2}.
\end{remark}
In the low regularity setting in dimension $n=2$, we are able to remove the additional $\frac{1}{2}$ derivative restriction under the additional hypothesis that the nonlinearity is conservative. Precisely, we have the following:

\begin{theorem}\label{t:global2c}
Let $n = 2$.
 a) Assume that the equation \eqref{qnls} is cubic and conservative, and that the initial data $u_0$ is small in $H^s$, with $s > \frac{3}2$,
 \begin{equation}
 \|u_0\|_{H^s} \leq \epsilon \ll 1.   
 \end{equation}
Then the solutions are global in time, satisfy the uniform bound
\begin{equation}
\| u\|_{L^\infty_t H^s_x} \lesssim \epsilon,
\end{equation}
and scatter at infinity, in the sense that
\begin{equation}
u_\infty = \lim_{t\to \infty} e^{-it \Delta_{g(0)}} u(t)     \qquad \text{in } H^{s}
\end{equation}
exists and has a continuous dependence on the initial data in the same topology.

b) The same result holds for  the equation \eqref{dqnls} for $s > \frac{5}2 $. 
\end{theorem}

These results represent the first global well-posedness results  for general cubic quasilinear ultrahyperbolic Schr\"odinger equations with small initial data in Sobolev spaces in dimension $n \geq 2$.  We make the following two important remarks about the scattering result:

\begin{remark}
It is useful to compare the results on scattering with 
the 1D results of the first and third authors in \cite{IT-qnls}. Here, the nonlinear effects are too strong to allow for standard scattering. This is well known even for localized data, where the best one can obtain is a form of modified 
scattering. In the higher dimensional setting, we do however have classical scattering.
\end{remark}

\begin{remark}
One may also be inclined to study the regularity of the 
wave operator $u_0 \to u_\infty$. Our analysis will imply that it has continuity properties similar to what one gets in a Hadamard local well-posedness theory. Namely,
\begin{itemize}
    \item it is continuous in $H^s$
    \item it is Lipschitz continuous (when restricted to small $H^s$ data) in certain weaker topologies
    (for instance, $L^2$), and is essentially close to the identity in the Lipschitz topology.
\end{itemize}
    
\end{remark}

While the above results are stated in a more compact form, the proofs will imply the following collection of auxiliary bounds:

\begin{enumerate}[label=(\roman*)]
\item{Bilinear $L^2$ bounds}, which can be stated in a balanced form 
\begin{equation} \label{intro-bi-bal}
   \| |D|^{\frac{3-n}2+\delta} |\la D\ra^{s-\frac14-c\delta} u|^2\|_{L^2_{t,x}} \lesssim_{\delta} \epsilon^2 , 
\end{equation}  
for some fixed $c>1$ and any $0<\delta\ll 1$, and in an imbalanced form (using standard paraproduct notation),
\begin{equation}\label{intro-bi-unbal}
   \| T_{\partial u} \bu\|_{L^2_t H^{s+\frac12}_x} \lesssim \epsilon^2  .
  \end{equation} 
 \item A lossless $L^{4}$ Strichartz bounds for $n\geq 4$:
  \begin{equation}\label{intro-Str4}
   \| \la D\ra^{s-\frac{2-n}4} u\|_{L^4_{t,x}} \lesssim\epsilon  . 
  \end{equation}  
    \item A nearly lossless $L^{4}$ Strichartz bounds for $n = 3$:
  \begin{equation}\label{intro-Str3}
   \| \la D\ra^{s-\frac{2-n}4-\delta} u\|_{L^4_{t,x}} \lesssim_{\delta} \epsilon  . 
  \end{equation}  
    \item A nearly lossless $L^{4} L^8$ Strichartz bounds for $n = 2$
  \begin{equation}\label{intro-Str2}
   \| \la D\ra^{s-\frac14-\delta} u\|_{L^4_t L^8_x} \lesssim_{\delta} \epsilon .  
  \end{equation}  
   \item Full Strichartz bounds with derivative loss:
  \begin{equation}\label{intro-Str-full}
   \| \la D\ra^{s-\frac{2}{p}} u\|_{L^{p}_t L^{q}_x} \lesssim \epsilon
   \end{equation}  
for any pair $(p,q)$ of sharp Strichartz exponents.
\end{enumerate}

\medskip
The reader is referred to Theorem~\ref{t:local-fe3} and Theorem~\ref{t:local-fe2} for a more precise, frequency envelope-based description of these results, where the exponent $\delta$ above is also  replaced by a scale invariant Besov space structure.
One may contrast the $L^4_{t,x}$  bounds above with the $L^6_{t,x}$ bounds which play the primary role in the 1D case, where $6$ is the sharp Strichartz exponent.

\bigskip
It is also interesting to discuss these results  in relation to the more general question of obtaining long time solutions for one or two dimensional dispersive flows with quadratic/cubic nonlinearities, which has gathered much interest lately. One can distinguish two different but closely related types of results that have been obtained, as well as several successful approaches.

On one hand, \emph{normal form methods} have been exploited 
in order to extend the lifespan of solutions, beginning with 
\cite{Shatah} in the late '80's. On the other hand, the \emph{ I-method}, introduced in \cite{I-method} introduced the idea of constructing  better almost conserved quantities. These two methods have been successfully used in the study of semilinear flows, where it was later observed that they are somewhat connected ~\cite{Bourgain-nf}.

However, when it comes to quasilinear problems, neither of these approaches 
works well; in particular, standard normal form transformations are almost
always unbounded. In order to remedy this, it was discovered in the one dimensional work of the authors and collaborators \cite{BH}, \cite{IT-g} that the normal form method can be adapted to quasilinear problems by constructing energies which simultaneously capture both the quasilinear and the normal form structures. This idea was called the \emph{modified energy method}, and can also be viewed as a quasilinear adaptation of the I-method.  Other alternate quasilinear approaches are provided by the \emph{ flow method } of  Hunter-Ifrim~\cite{hi}, where a better, bounded  nonlinear normal form transformation is constructed using a well chosen auxiliary flow, and by the paradiagonalization method of Alazard and-Delort \cite{AD}, where a paradifferential symmetrization is applied instead, in combination with a partial normal form transformation. Overall, these
ideas lead to improved lifespan bounds for solutions with initial data in Sobolev spaces.

Turning our attention to the question of obtaining scattering, global in time solutions for one or two dimensional dispersive flows with quadratic/cubic nonlinearities, this has been almost exclusively studied  under the assumption that the initial data is both \emph{small} and \emph{localized};  for a few examples out of many, see for instance \cite{HN,HN1,LS,KP,IT-NLS} in one dimension, and \cite{GMS} in two dimensions. The localization 
plays a fundamental role here, as it enforces a pattern of waves emerging from a center and with uniform dispersive decay, with either classical 
or modified scattering at infinity. The nonlinearities in the one dimensional models are primarily cubic, though the analysis has also been extended via normal form and modified energy methods to problems which also have nonresonant quadratic interactions; several such examples are  \cite{AD,IT-g,D,IT-c,LLS}, see also further references therein, as well as the authors' expository paper \cite{IT-packet}.
On the other hand some quadratic nonlinearities can also be handled, see for instance \cite{GMS} and references therein, based on the idea of \emph{space-time} resonances.

Comparing the above class of results where the initial data
is both small and localized with the present results, which do not require any localization assumption, it is clear that in the latter case the problem becomes much more difficult, because the absence of localization allows for far stronger nonlinear interactions over long time-scales
and in particular prevents any uniform decay and any scattering, be it classical or modified.  

Another key point in the one dimensional setting was that, one also needs to distinguish between the focusing and defocusing case, as was made clear in the earlier  semilinear work in \cite{IT-global} and \cite{IT-focusing}, and then quasilinear work in \cite{IT-qnls}.  By contrast, in two dimensions our conjecture suggests that in the small data case 
this distinction is no longer necessary.

\subsection{Outline of the paper} \
Now, we aim to give a brief roadmap of the proof. The results described in the introduction
present a coarse picture of some of the required bounds, however to establish these results and to obtain a robust description of the global in time dispersive properties of the solutions it is of critical importance to have a suitable family of quantitative estimates. As in the positive-definite case in \cite{IT-qnls2}, these come in essentially two flavors:

\begin{description}
    \item[Strichartz estimates] These are entirely classical in the constant coefficient case, though generally highly nontrivial in the quasilinear setting.
    \item[Bilinear $L^2$ estimates] These will play the leading role in our analysis here (and in particular serve as a substitute for the local smoothing estimates used in previous works). These should be interpreted as transversality bounds, which capture the bilinear interaction of two transversal waves.  These are classical in the linear, constant coefficient setting, but again, highly nontrivial in the quasilinear context. One of the key novelties in this paper is in establishing these estimates when the metric is indefinite.
\end{description}

It is a common theme in studying quasilinear problems, that proving such estimates requires a considerable amount of apriori information on the solutions to begin with. For this reason, the entire proof of our results will be tied together by  well-designed bootstrap. This will be done at the level of linear and bilinear estimates, where in essence, one assumes that a weaker version of such estimates already holds.

This bootstrap will end up being carried through each section of the paper, and only finally closed at the very end. Nevertheless, we organize the proofs in such a way that each section is more or less modular. We briefly describe each of the requisite components below.

\medskip

\textbf{I. Littlewood-Paley theory and frequency envelopes.} 
The estimates \eqref{intro-bi-bal}-\eqref{intro-Str-full} give coarse description of the overall setup, however, for our purposes, it is important to have a finer accounting of these estimates relative to each frequency scale. This is naturally phrased using a combination of a Littlewood-Paley decomposition of the solutions and the language of frequency envelopes. This will allow us to precisely track the size of the dyadic pieces in both linear and bilinear estimates. We give a precise description of this in  Section~\ref{s:boot}, where we provide a more precise frequency localized formulation of the bounds \eqref{intro-bi-bal}-\eqref{intro-Str-full}, as well as the set-up of the global bootstrap argument. 

\medskip

\textbf{II. The paradifferential reduction.} Obviously, it is of critical importance to obtain good apriori estimates for the nonlinear equation, but it is similarly crucial in quasilinear problems to obtain estimates for the corresponding linearized evolution. An elegant way to carry out these steps (more or less) simultaneously is to instead prove estimates for the associated linear paradifferential flow. We will introduce this formalism precisely in  Section~\ref{s:lin} and demonstrate how to reformulate both the nonlinear and linearized equations as paradifferential flows with suitable source terms. In a sufficiently broad sense, these source terms will play a perturbative role. However, in the two-dimensional setting, which is the most difficult case, we will have to carefully analyze the structure of the cubic part of the nonlinearity. Of particular note will be the so-called non-transversal resonant interactions. See Section 2 for details.

\medskip

\textbf{III. Conservation laws in density-flux form.}
Efficient energy estimates naturally play an important role in our analysis, but for the proof of the bilinear estimates we need to consider these estimates in a frequency-localized setting,
and also as local conservation laws in density-flux form. We carry out this discussion at a broad level in Section~\ref{s:df}, but refine it in the last section for 2D case.

\medskip

\textbf{IV. Interaction Morawetz bounds.} Establishing a suitable family of interaction Morawetz estimates is of critical importance in the bilinear analysis mentioned above. This idea originated in work of the I-team in \cite{MR2415387}. In the first and third author's previous work \cite{IT-qnls2}, the starting point 
 was provided by Planchon-Vega~\cite{PV}, which considered the constant coefficient case. We remark however, that even in the case of small data, positive-definite metrics, obtaining suitable analogues in the quasilinear setting is not at all straightforward. This issue is further compounded in our setting by the fact that the metric is possibly indefinite. This issue manifests most critically in the balanced bilinear estimate, where carrying out the previous approach does not come close to yielding a coercive bound (see the discussion in Section 5). To remedy this, we introduce a novel family of interaction Morawetz functionals, which recover the bound one obtains in the elliptic case, up to an arbitrarily small dyadic loss. This new approach is quite simple, and also appears to be very robust, in the sense that it could apply in large data scenarios (where the metric is not a small perturbation of a Euclidean one). Therefore, this could serve to potentially improve the corresponding results even in the definite case, where the Morawetz estimate required small data.

\medskip

\textbf{V. Strichartz estimates for paradifferential flows.}
It is possible to establish Strichartz estimates for the paradifferential problem in 1D, and this was carried out in \cite{IT-qnls}, locally in time. It becomes somewhat intractable to carry this out higher dimension. We partially ameliorate this issue with the stronger bilinear bounds in the balanced case. In dimensions three and higher, this yields an almost scale-invariant (almost) $L^4$ 
bound. However, as in the positive-definite setting, this is no longer the case in 2D. Instead, in Section~\ref{s:Str} we employ a similar strategy to \cite{IT-qnls2} to establish Strichartz estimates for a full range of exponents, globally in time, but with a loss of derivatives.

\medskip

\textbf{VI. Rough solutions as limits of smooth solutions.}
In Section~\ref{s:rough}, we carry out the (by now) standard procedure for obtaining rough solutions for quasilinear flows. The analysis here is very similar to that carried out in the previous paper \cite{IT-qnls2}. See also the expository paper \cite{IT-primer} for an outline of the overarching scheme.

\medskip

\textbf{VII. The 2D global result at low regularity.}
It will be possible to view the cubic balanced terms (with and without phase rotation symmetry) as perturbative in dimension three and higher. However this is not the case in two space dimensions at low regularity. To deal with the cubic interactions with phase-rotation symmetry, we will need to make use of the ``conservative" assumption. This will be addressed in Section 10. We remark importantly that in the ultrahyperbolic setting, we also have to contend with a large set of balanced, high-frequency resonant interactions which appear in the cubic terms without phase rotation symmetry. Remarkably, we will not need to impose any further structural restrictions (in addition to the definite case) to deal with these terms. Roughly, this is because either 
the three input frequencies are not all equal, or 
the output frequency of these cubic interactions will be transverse (when suitably interpreted) to the three (equal) input frequencies.
 This is in sharp contrast to the cubic term with phase rotation symmetry. We will exploit this mild frequency separation between the output and input frequencies by establishing a new, refined bilinear Strichartz estimate which holds even for waves localized at the same dyadic scale (but supported at different frequency angles). This issue is described in Section 2 and Section 7.

\subsection{Acknowledgements} 
 The first author was supported by the NSF grant DMS-2348908, by a Miller Visiting Professorship at UC Berkeley for the Fall semester 2023, and by the Simons Foundation as a Simons Fellow. The second author was supported
by a fellowship in the Simons Society of Fellows. The third author was supported by the NSF grant DMS-2054975 as well as by a Simons Fellowship from the Simons Foundation.

 \section{Notations and preliminaries}
\label{s:notations}

 Here we start with the usual Littlewood-Paley decomposition. Then, for orientation purposes, we review some standard facts for constant coefficient linear Schr\"odinger flows, including the Strichartz and bilinear $L^2$ estimates.  At the nonlinear level, we recall  the notion of cubic resonant interactions, and provide a classification of such interactions adapted to the ultrahyperbolic case. Finally, we set up  our notations for multilinear forms and their symbols.

\subsection{ The Littlewood-Paley decomposition}
This is done in a standard fashion. Given  a bump function $\psi$ adapted to $[-2,2]$ and equal to $1$ on $[-1,1]$, we define the Littlewood-Paley operators $P_{k}$, as well as
and
$P_{\leq k} = P_{<k+1}$ for  $k \in \Z$ by 
\[
\widehat{P_{\leq k} f}(\xi) := \psi(2^{-k} \xi) \hat f(\xi)
\] 
 and $P_k := P_{\leq k} - P_{\leq k-1}$.  All the operators $P_k$,
$P_{\leq k}$ are bounded on all translation-invariant Banach spaces, thanks to Minkowski's inequality.  We also define $P_{>k} := P_{\geq k-1}
:= 1 - P_{\leq k}$.
Then the inhomogeneous Littlewood-Paley decomposition reads
\begin{equation*}
1=P_{\leq 1}+\sum_{k=2}^{\infty}P_{k},
\end{equation*}
where the multipliers $P_k$ have smooth symbols localized at frequency $2^k$. Correspondingly, our solution $u$ will be decomposed as
\[
u = \sum_{k \in  \N} u_k, 
\]
where we adopt the convention that $u_1:=P_{\leq 1}$ and $u_k=P_ku$ for $k\geq 2$. 
\begin{remark}
Most of the time, it will entirely suffice to work with projectors $P_k$ for $k\geq 0$, but some of our estimates will be naturally phrased in terms of suitable homogeneous Besov-type spaces, which require considering $P_k$ for negative $k$. See for instance \eqref{uab-bi-bal}.
\end{remark}
The main estimates we will establish for our solution $u$
will be linear and bilinear estimates for the functions $u_k$.

Alternatively, instead of the above discrete Littlewood-Paley decomposition one may also consider its continuous version, where,
with $k$ a real nonnegative parameter, we define 
\[
P_k u := \frac{d}{dk} P_{< k} u.
\]
Then our Littlewood-Paley decomposition for $u$
becomes
\[
u = u_0 +\int_0^\infty u_k \, dk.
\]
This will be useful in the proof of the local well-posedness result in Section~\ref{s:rough}, where rough solutions are constructed as a continuous limit of smooth solutions.

To shorten some calculations we will also use Greek indices
for Littlewood-Paley projectors, e.g., $P_\lambda$, $P_\mu$
where we take $\lambda,\mu \in 2^{\N}$. Correspondingly, our Littlewood-Paley decomposition of a function $u$ will read
\[
u = \sum_{\lambda \in  2^\N} u_\lambda, \qquad u_\lambda = P_\lambda u. 
\]

\subsection{Strichartz and bilinear $L^2$ bounds} 
Here we collect the requisite the  Strichartz inequalities, 
which apply to solutions to the inhomogeneous linear ultrahyperbolic Schr\"odinger equation:
\begin{equation}\label{bo-lin-inhom}
(\partial_t + \Delta_g)u = f, \qquad u(0) = u_0.
\end{equation}
Here $g$ is a constant, symmetric, non-degenerate metric (not necessarily positive-definite). These estimates are the $L^p_t L^q_x$ spacetime bounds which measure the dispersive effect of the linear Schr\"odinger flow. The admissible Strichartz exponents are exactly the same as in the elliptic case.

\begin{definition}
The pair $(p,q)$ is an admissible Strichartz exponent in $n\geq 1$ space dimensions if
\begin{equation}
\frac{2}{p}+\frac{n}q = \frac{2}{n}, \qquad 2 \leq p,q \leq \infty    \end{equation}
with the only exception of the forbidden endpoint $(2,\infty)$ 
in dimension $n=2$.
\end{definition}
We summarize the Strichartz estimates in the $L^2$ setting as follows.

\begin{lemma}
Assume that $u$ solves \eqref{bo-lin-inhom} in $[0,T] \times \R$. Then the following estimate holds for all sharp Strichartz pairs $(p,q), (p_1,q_1)$:
\begin{equation}
\label{strichartz}
\| u\|_{L^p_t L^q_x} \lesssim \|u_0 \|_{L^2} + \|f\|_{L^{p'_1}_t L^{q'_1}_x}.
\end{equation}
\end{lemma}
These estimates follow in an entirely analogous fashion to the case $g=Id$. See for instance, the discussion in \cite{ghidaglia1993nonelliptic}. We also direct the interested reader to \cite{Ke-Ta}, where the final endpoint $p=2$ was studied in the elliptic case. 
 \medskip
 
It is efficient to place these estimates under a single umbrella. In dimension three and higher we have access to both end-points, and it is most practical 
 to  define the Strichartz space $S$ associated to the $L^2$ flow by
\[
S = L^\infty_t L^2_x \cap L^2_t L^{\frac{2n}{n-2}}_x,
\]
as well as its dual 
\[
S' = L^1_t L^2_x + L^{2} _t L^{\frac{2n}{n+2}}_x .
\]

Then the Strichartz estimates can be summarized all as 
\begin{equation}
\label{strichartz3}
\| u\|_{S} \lesssim \|u_0 \|_{L^2} + \|f\|_{S'}.
\end{equation}

In two space dimensions we lose access to the $(2,\infty)$ endpoint
so the above definition of the Strichartz space $S$ no longer applies. Instead, it is convenient to use the $U^p_{\UH}$ and $V^p_{\UH}$ spaces associated to the linear ultrahyperbolic Schr\"odinger evolution
\begin{equation} \label{flat-uh-full}
(i \partial_t + \Delta_{g}) u = f, \qquad u(0) = u_0,
\end{equation}
where the operator $\Delta_{g}$ associated to the constant metric $g$ is nondegenerate but not elliptic.
These were introduced in unpublished work of the third author~\cite{T-unpublished}, as an improvement over the earlier 
Bourgain spaces which can be used in scale invariant contexts.
See also \cite{KT}, \cite{HTT} for some of the first applications of these spaces. In this setting the Strichartz estimates can be all placed under a single umbrella,
\begin{equation}
\label{strichartz2}
\| u\|_{V^2_\UH} \lesssim \|u_0 \|_{L^2} + \|f\|_{DV^2_\UH},
\end{equation}
where the transition to the Strichartz bounds is provided by the embeddings
\begin{equation}\label{UV-embed}
V^2_\UH \subset U^p_\UH \subset L^p_t L^q_x,
\qquad L^{p'}_t L^{q'}_x \subset DV^{p'}_\UH \subset DV^2_\UH.
\end{equation}

\medskip
The last property of the linear Schr\"odinger equation we will describe here is the bilinear $L^2$ estimate, which is as follows:

\begin{lemma}
\label{l:bi}
Let $u^1$, $u^2$ be two solutions to the inhomogeneous ultrahyperbolic Schr\"odinger equation \eqref{flat-uh-full} with data 
$u^1_0$, $u^2_0$ and inhomogeneous terms $f^1$ and $f^2$. Assume 
that $u^1$ and $u^2$ are frequency localized in balls $B_1$, $B_2$
with radius $\lambda_1 \lesssim  \lambda_2$ so that  
\[
d(B_1,B_2) \gtrsim \lambda_2.
\]
 Then we have 
\begin{equation}
\label{bi-di}
\| u^1 u^2\|_{L^2_{t,x}} \lesssim \lambda_1^{\frac{n-1}{2}} \lambda_2^{-\frac12} 
( \|u_0^1 \|_{L^2_x} + \|f^1\|_{S'}) ( \|u_0^2 \|_{L^2_x} + \|f^2\|_{S'}).
\end{equation}
\end{lemma}
This is not directly used in the paper; instead it is meant to serve as a guide for similar estimates we prove later for the paradifferential flow, where the frequency localizations are assumed to be dyadic. The above statement is in effect equivalent to the corresponding dyadic setting via Galilean transformations.  
This is a standard result, which in particular can be seen as a corollary of more general estimates in \cite{Tao-bi}, at least in the homogeneous case. We also note that, as a corollary, a similar bound can be stated at the level of the $U^2$ spaces:
\begin{equation}
\label{bi-di-UV}
\| u^1 u^2\|_{L^2_{t,x}} \lesssim \lambda_1^{\frac{n-1}{2}} \lambda_2^{-\frac12} 
 \|u^1 \|_{U^2_{UH}} \|u^2 \|_{U^2_{UH}} .
\end{equation}

\subsection{ Resonant analysis}
\label{s:resonances}
In this subsection we aim to classify the trilinear interactions associated to the linear 
constant coefficient flow; for simplicity, the reader may think 
of the context of an equation of the form
\[
(i\partial_t + \Delta_g) u = C(\bfu,\bfu,\bfu),
\]
where $C$ is a translation invariant trilinear form and we have used the notation $\bfu = \{ u, \bu\}$ in order to allow for both $u$ and $\bu$ in each of the arguments of $C$. For clarity, we note that this includes expressions of the form
\[
C(u,\bu,u), \quad C(u,u,u), \quad C(u,\bu,\bu), \quad C(\bu,\bu,\bu).
\]
Within these choices we distinguish the first one,  namely $C(u,\bu,u)$, 
as the only one with \emph{phase rotation symmetry}.

Given three input frequencies $\xi^1, \xi^2,\xi^3$ for 
our cubic nonlinearity, the output will be at frequency 
\begin{equation}\label{freq}
\xi^4 = \pm \xi^1 \pm \xi^2 \pm \xi^3,
\end{equation}
where the signs are chosen depending on whether we use the input $u$ or $\bu$. We call such an interaction {balanced} if all $\xi^j$'s are of comparable size, and \emph{unbalanced} otherwise.
This is a \emph{resonant interaction} if and only if we have a similar relation for the associated time frequencies, namely 
\[
|\xi^4|_g^2 = \pm |\xi^1|_g^2 \pm |\xi^2|_g^2 \pm |\xi^3|_g^2.
\]
where we write (by slight abuse of notation)
\begin{equation*}
|v|_g^2:=v^jv_j=g^{ij}v_iv_j.
\end{equation*}
On the other hand from the perspective of bilinear estimates we 
seek to distinguish the \emph{transversal interactions}, which are 
defined as those for which no more than two of the frequencies $\xi^j$ coincide. We seek to classify interactions from this perspective:

\begin{description}
\item[(i) With phase rotation symmetry]
Here the relation \eqref{freq} can be expressed in a more symmetric fashion as 
\[
\Delta^4 \xi = 0, \qquad \Delta^4 \xi = \xi^1-\xi^2+\xi^3-\xi^4 .
\]
This is a resonant interaction if and only if we have a similar relation for the associated time frequencies, namely 
\[
\Delta^4 |\xi|_g^2 = 0, \qquad \Delta^4 |\xi|_g^2 = |\xi^1|_g^2-|\xi^2|_g^2+|\xi^3|_g^2-|\xi^4|_g^2 .
\]
Hence, we define the resonant set in a symmetric fashion as 
\[
\calR := \{ \Delta^4 \xi = 0, \ \Delta^4 |\xi|_g^2 = 0\}.
\]
In 1D it is easily seen that this set may be characterized as
\[
\calR = \{ \{\xi^1,\xi^3\} = \{\xi^2,\xi^4\}\}.
\]
However, in higher dimensions the situation is more complicated, and the resonant set 
$\calR$ consists of all quadruples $\xi^1$, $\xi^2$, $\xi^3$
and $\xi^4$ where $\xi^4=\xi^1-\xi^2+\xi^3$ and where $\xi^1-\xi^2$ is orthogonal to $\xi^2-\xi^3$, with respect to the metric $g$. 

Most resonant quadruples have the property that they can be split
into two pairs of separated frequencies; we call these
\emph{transversal resonant} quadruples. The set of non-transversal resonant interactions has a simple description and is given by the diagonal set 
\[
\calR_2 = \{\xi^1=\xi^2=\xi^3 = \xi^4\}.
\]
We will refer to these interactions as the \emph{doubly resonant interactions}.

\item[(ii) Without phase rotation symmetry] Here we concentrate  our attention
on describing the worst case scenario, which corresponds to interactions which are both resonant and non-transversal.
The first requirement corresponds to frequencies where
\[
 \xi^1 \pm \xi^2 \pm \xi^3 \pm \xi^4 = 0,
\]
and 
\[
|\xi^1|_g^2-|\xi^2|_g^2+|\xi^3|_g^2-|\xi^4|_g^2=0.
\]
The second requirement is that three of the frequencies are equal.
Without any loss in generality assume that $\xi^1 = \xi^2 = \xi^3:=\xi$.
Depending on the choice of signs and excluding the case of phase rotation symmetry, we obtain one of the following systems:
\[
\xi^4 = \pm 3 \xi, \qquad |\xi^4|_g^2 = \pm 3 |\xi|_g^2,
\]
with matched signs, or
\[
\xi^4 = - \xi, \qquad |\xi^4|_g^2 = - |\xi|_g^2.
\]
If $g$ is positive definite, then $\xi = 0$ is the only solution. Unfortunately, in the indefinite case, this is no longer true. However, when $\xi\neq 0$, we make the crucial observation that $\xi^4$ and $\xi$ are frequency separated. Put another way, this case corresponds to cubic interactions in which there is a gap of size $\approx \xi$ between the output frequency and the three (equal) input frequencies. This information turns out to be sufficient for our global result, even at low regularity (albeit, barely), as we are still partially able to take advantage of the new unbalanced bilinear estimates in our analysis that we will establish later. In the sequel, we will call such interactions \emph{weakly transversal}. We summarize our observations in the following lemma 
\begin{lemma}\label{l:C4-dec} Except for the case $(0,0,0,0)$, all cubic interactions 
without phase rotation symmetry are either nonresonant, transversal or weakly transversal.
\end{lemma}

\end{description}
We will use the above classification to separate the cubic interactions using smooth frequency cut-offs as 
\[
C(\bfu,\bfu,\bfu) = C^{res}(u,\bu,u) + C^{0}(\bfu,\bfu,\bfu) + 
C^{nr}(\bfu,\bfu,\bfu) + C^{tr}(\bfu,\bfu,\bfu)+C^{wt}(\bfu,\bfu,\bfu),
\]
where
\begin{itemize}
    \item $C^{res}(u,\bu,u)$ contains balanced interactions with phase rotation symmetry
    \item $C^{0}(\bfu,\bfu,\bfu)$ contains only low frequency interactions
    \item $C^{nr}(\bfu,\bfu,\bfu)$ contains only balanced nonresonant interactions
  \item $C^{tr}(\bfu,\bfu,\bfu)$ contains only transversal interactions, balanced or unbalanced
  \item $C^{wt}(\bfu,\bfu,\bfu)$ contains balanced, weakly transversal interactions.
\end{itemize}

\subsection{ Translation invariant multilinear forms} 
Since our problem is invariant with respect to translations, translation invariant multilinear forms play an important role in the analysis. Below we review some of our notations in this regard, following 
\cite{IT-global}, \cite{IT-qnls} and \cite{IT-qnls2}.

We begin with multipliers $m(D)$, which correspond to multiplication in the Fourier space by their symbol $m(\xi)$, and which in the physical space  are interpreted as convolution operators,
\[
m(D) u (x) = \int K(y) u(x-y) \,dy, \quad \hat K = m.
\]
More generally, we will denote by $L$  any convolution operator 
\[
Lu (x) = \int K(y) u^y(x)\, dy,
\]
where $K$ is an integrable kernel, or 
 a bounded measure, with a universal bound and where for translations we use the notation
\[
u^y(x) := u(x-y).
\]

\medskip

 Moving on to multilinear operators,  we denote by $L$ any multilinear form 
\[
L(u_1,\cdots, u_k) (x) = 
\int K(y_1, \cdots,y_k) u_1^{y_1}(x) \cdots u_k^{y_k}(x) \, dy,
\]
where, again, $K$ is assumed to have a kernel which is integrable, or more generally,  a bounded measure (in order to be able to include products here). Such multilinear forms satisfy the same bounds as 
corresponding products do, as long as we work in translation invariant Sobolev spaces. Such form may also be described by their symbol $a(\xi^1, \cdots \xi^k)$, as
\[
\widehat{L(u_1,\cdots, u_k)}(\xi) = (2\pi)^{-\frac{n(k-1)}2}\int_{\xi^1+ \cdots \xi^k = \xi} 
a(\xi^1, \cdots \xi^k) \hat u_1(\xi^1) \cdots \hat u_k(\xi^k) d \xi^1 \cdots d \xi^{k-1},
\]
where $a$ is nothing but the  inverse Fourier transform of the kernel $K$.

\medskip

In the study of the two dimensional problems we also need to separately consider cubic forms which exhibit phase rotation symmetry, where the arguments $u$ and $\bar u$ are alternating. For such forms we borrow the notations from our earlier paper \cite{IT-global}.

Precisely, for an integer $k \geq 2$, we will use translation invariant $k$-linear  forms 
\[
(\mathcal D(\R^n))^{k} \ni (u_1, \cdots, u_{k}) \to     Q(u_1,\bu_2,\cdots) \in \mathcal D'(\R^n),
\]
where the nonconjugated and conjugated entries are alternating.

Such a form is uniquely described by its symbol $q(\xi_1,\xi_2, \cdots,\xi_{k})$
via
\[
\begin{aligned}
Q(u_1,\bu_2,\cdots)(x) = (2\pi)^{-nk} & 
\int e^{i(x-x_1)\xi^1} e^{-i(x-x_2)\xi^2}
\cdots 
q(\xi^1,\cdots,\xi^{k})
\\ & \qquad 
u_1(x_1) \bu_2(x_2) \cdots  
dx_1 \cdots dx_{k}\, d\xi^1\cdots d\xi^k,
\end{aligned}
\]
or equivalently on the Fourier side
\[
\mathcal F Q(u_1,\bu_2,\cdots)(\xi)
= (2\pi)^{-\frac{n(k-1)}2} \int_{D}
q(\xi^1,\cdots,\xi^{k})
\hat u_1(\xi^1) \bar{\hat u}_2(\xi^2) \cdots  \,
d\xi^1 \cdots d\xi^{k-1},
\]
where, with alternating signs, 
\[
D := \{ \xi = \xi^1-\xi^2 + \cdots \}.
\]

They can also be described via their kernel
\[
Q(u_1,\bu_2,\cdots)(x) =  
\int K(x-x_1,\cdots,x-x_{k})
u_1(x_1) \bu_2(x_2) \cdots  
dx_1 \cdots dx_{k},
\]
where $K$ is defined in terms of the  
Fourier transform  of $q$
\[
K(x_1,x_2,\cdots,x_{k}) = 
(2\pi)^{-\frac{kn}2} \hat q(-x_1,x_2,\cdots,(-1)^k x_{k}).
\]

These notations are convenient but somewhat nonstandard because of the alternation of complex conjugates; our usage will be clear from the context.
We note that the cases of odd $k$, respectively even $k$ play different roles here, as follows:

\medskip

i) The $2k+1$ multilinear forms will be thought of as functions, e.g. those which appear in some of our evolution equations.

\medskip

ii) The $2k$ multilinear forms will be thought of as densities, e.g. which appear  in some of our density-flux pairs.

\medskip
In particular, 
to each $2k$-linear form $Q$ we will associate
a corresponding $2k$-linear functional $\mathbf{Q}$ defined by 
\[
\mathbf{Q}(u_1,\cdots,u_{2k}) := \int_\R Q(u_1,\cdots,\bu_{2k})(x)\, dx,
\]
which takes real or complex values.
This may be alternatively expressed 
on the Fourier side as 
\[
\mathbf{Q}(u_1,\cdots,u_{2k}) = (2\pi)^{n(1-k)} \int_{D}
q(\xi^1,\cdots,\xi^{2k})
\hat u_1(\xi^1) \bar{\hat u}_2(\xi^2) \cdots  
\bar{\hat u}_{2k}(\xi^{2k})\,d\xi^1 \cdots d\xi^{2k-1},
\]
where, with alternating signs, the diagonal $D_0$ is given by
\[
D_0 = \{ 0 = \xi^1-\xi^2 + \cdots \}.
\]
In order to define the multilinear functional $\mathbf{Q}$ it suffices to know the symbol $q$ on $D_0$.

\section{A frequency envelope formulation of the results}
\label{s:boot}
For ease of exposition, we stated Theorem~\ref{t:local3+}, Theorem~\ref{t:global3}, Theorem~\ref{t:global2} and Theorem~\ref{t:global2c}
in a simplified form in the introduction. However, in our analysis, these results will be naturally complemented with a suitable family of linear and bilinear $L^2$ bounds for the solutions. The types of bilinear $L^2$ bounds that we have will depend on the relative frequencies of the corresponding interacting waves. We will track these bounds using a somewhat complex bootstrap loop. To keep things efficient, it is convenient to phrase this bootstrap using the language of frequency envelopes. As in the definite case in \cite{IT-qnls2}, the primary goals of this section are  
\begin{enumerate}[label=(\roman*)]
\item to recall the standard frequency envelope formalism, 
\item to phrase our main estimates in terms of these frequency envelopes, 
\item  to set up the overarching bootstrap assumptions in the proofs of each of the theorems, 
\item to outline the continuity argument which will allow us to eventually close the bootstrap.
\end{enumerate}

\subsection{Frequency envelopes}
Here we recall the standard \emph{frequency envelope} formalism. This is an idea introduced by Tao in \cite{Tao-BO}. It will serve as an efficient bookkeeping device for tracking the evolution of the energy of the solutions between dyadic energy shells. We say that a sequence $c_{k}\in l^2$
is an $L^2$ frequency envelope for $\phi \in L^2$ if
\begin{itemize}
\item[i)] $\sum_{k=0}^{\infty}c_k^2 < \infty$;\\
\item[ii)] it is slowly varying, 
\[
c_j /c_k \leq 2^{\delta \vert j-k\vert}, \qquad j,k \in \N
\]
with a small universal constant $\delta$;
\\
\item[iii)] it bounds the dyadic norms of $\phi$, namely $\Vert P_{k}\phi \Vert_{L^2} \leq c_k$. 
\end{itemize}
Given a frequency envelope $c_k$ we define 
\[
 c_{\leq k} := (\sum_{j \leq k} c_j^2)^\frac12, \qquad  c_{\geq k} := (\sum_{j \geq k} c_j^2)^\frac12.
\]
 In  practice we may choose our envelopes in a minimal fashion so that we have the equivalence
\[
\sum_{k=0}^{\infty}c_k^2  \approx \|u\|_{L^2}^2
\]
and so that we also have 
\[
c_0 \approx \|u\|_{L^2}.
\]
The above definition applies also for defining $H^s$ frequency envelopes.

\begin{remark}\label{r:unbal-fe}
It is also useful to weaken the slowly varying assumption to
\[
2^{- \delta \vert j-k\vert} \leq    c_j /c_k \leq 2^{C \vert j-k\vert}, \qquad j < k,
\]
where $C$ is a fixed (possibly large) constant. We note that all the results in this paper are compatible with this choice.
This will allow for extra flexibility when proving higher regularity results by the same arguments.
\end{remark}

\subsection{The frequency envelope formulation of the results}

The goals of this subsection are twofold: (i) restate the bounds needed for our main results in Theorem~\ref{t:local3+}, Theorem~\ref{t:global3}, Theorem~\ref{t:global2} and Theorem~\ref{t:global2c} in the frequency envelope setting,  and (ii) to set up the bootstrap argument for the proof of these bounds. 

To do this, we begin by assuming that the initial data is small in $H^s$,
\[
\| u_0\|_{H^s} \lesssim \epsilon.
\]
We then let $\{c_{\lambda}\}$ be an admissible $H^s$ frequency envelope for $u_0$ with $\|c\|_{l^2}\approx 1$. This gives the following frequency envelope bound for each dyadic piece of $u_0$,
\[
\|u_{0,\lambda}\|_{H^s} \leq \epsilon c_\lambda, \qquad \lambda\in 2^{\mathbb{N}},
\]

In our analysis, our goal will be to show that the solution $u$ satisfies a similar frequency envelope bound at later times. 

We are now almost ready to state the frequency envelope bounds that will be the subject of the next theorems. Before doing this, we will need the following definition to treat the full range of interactions we will encounter in our analysis of the resonant cubic nonlinearity later on. This is a new feature in the ultrahyperbolic problem, where one has to contend with the resonant weakly transversal interactions outlined in the previous section.
\begin{definition}\label{transversalpair}
Let $u_{\mu}$ and $v_{\lambda}$ be functions localized at dyadic frequencies $\mu$ and $\lambda$ with $\mu\leq\lambda$. We say that $u_{\mu}$ and $v_{\lambda}$ are unbalanced if one of the following two properties hold:
\begin{itemize}
\item $u_{\mu}$ and $v_{\lambda}$ are frequency separated: $\mu\ll \lambda$,
\item $u_{\mu}$ and $v_{\lambda}$ are semi-balanced: $\mu\approx\lambda$ but $u_{\mu}$ and $v_{\lambda}$ are localized in disjoint balls $B_\mu$ and $B_\lambda$ with radii $\leq \delta\lambda$ and $\leq \delta\mu$, respectively, and
\begin{equation*}
d(B_\lambda,B_\mu)\gtrsim\lambda.
\end{equation*}
\end{itemize}
\end{definition}
It is crucial to incorporate the latter property in the above definition into our unbalanced bilinear estimate below. This is due to the weakly transversal resonant interactions (as defined in the previous section) which arise in the ultrahyperbolic problem, where all of the input frequencies (of size $\lambda$) are comparable in magnitude, but they have $O(\lambda)$ separation with the output frequency. Now, we are ready to state the frequency envelope bounds for a solution $u$. These are summarized as follows and will be the main subject of our next theorem:
\begin{enumerate}[label=(\roman*)]
\item Uniform frequency envelope bound:
\begin{equation}\label{uk-ee}
\| u_\lambda \|_{L^\infty_t L^2_x} \lesssim \epsilon c_\lambda \lambda^{-s}
\end{equation}
\item Unbalanced bilinear $L^2$ bounds: Let $u_{\mu}$ and $u_{\lambda}$ be unbalanced in the sense of Definition~\ref{transversalpair}. There holds
\begin{equation} \label{uab-bi-unbal}
\| u_\lambda \bu_\mu^{x_0}  \|_{L^2_{t,x}} \lesssim \epsilon^2 c_\lambda c_\mu \lambda^{-s-\frac12} \mu^{-s+\frac{n-1}2}
\end{equation}
\item Balanced bilinear $L^2$ bound: 
\begin{equation} \label{uab-bi-bal}
\sup_{j\in\mathbb{Z}}\|P_{j}|D|^{\frac{3-n}2} (u_\lambda \bu_\mu^{x_0})  \|_{L^2_{t,x}} \lesssim \epsilon^2 c_\lambda c_\mu \lambda^{-2s+\frac12}   (1+ \lambda |x_0|)
, \qquad \lambda \approx \mu 
\end{equation}
\item Strichartz bounds with a loss, for any sharp Strichartz exponents $(p,q)$:
\begin{equation}\label{uk-Str}
\| u_\lambda\|_{L^p_t L^q_x} \lesssim \epsilon c_\lambda \lambda^{-s+ \frac{2}p}  . 
\end{equation}
\end{enumerate}
We remark that as a simple consequence of the balanced 
bilinear bound \eqref{uab-bi-bal}, we also obtain
the linear bound  
\begin{equation}\label{uk-l4-4d}
\| u_\lambda\|_{L^4_{t,x}} \lesssim \epsilon c_\lambda \lambda^{-s+ \frac{n-2}4}, \qquad n \geq 4. 
\end{equation}
The corresponding $L^4$ bound when $n=3$ is borderline, and instead \eqref{uab-bi-bal} ensures the slightly weaker bound
\begin{equation}\label{3dlinbound}
\|u_{\lambda}\|_{L_t^4L_x^{4+\delta}}\lesssim_{\delta} \epsilon c_{\lambda}\lambda^{-s+\frac{1}{4}+c\delta},\qquad n=3,\qquad 0<\delta\ll 1
\end{equation}
for some universal constant $c>0$. However, by interpolating with the $L_t^4L_x^3$ Strichartz bound in \eqref{uk-Str}, we can recover a nearly lossless $L^4$ estimate,
\begin{equation}\label{uk-l4-3d}
\| u_\lambda\|_{L^4_{t,x}} \lesssim _{\delta}\epsilon c_\lambda \lambda^{-s+ \frac{n-2}4+\delta}, \qquad n=3  ,  
\end{equation}
which will be more than enough for our purposes. When $n=2$, we also observe that interpolating \eqref{uab-bi-bal} with the energy bound \eqref{uk-ee} yields the lossless $L_t^6L_x^{4}$ bound 
\begin{equation}\label{uk-l4-2d}
\| u_\lambda\|_{L^6_t L^{4}_x} \lesssim \epsilon c_\lambda \lambda^{-s+ \frac16}, \qquad n = 2  . 
\end{equation}
This estimate will be important for establishing the $\epsilon^{-6}$ lifespan bound for our solutions in the low regularity regime. We also remark on a special case of the Strichartz bound
\eqref{uk-Str} which has a $1/2$ derivative loss and will play a leading role in the sequel: 
\begin{equation}\label{uk-Str2}
\| u_\lambda\|_{L^4_{t,x}} \lesssim \epsilon c_\lambda \lambda^{-s+ \frac12}, \qquad n = 2.  
\end{equation}

Here we distinguish between the  2D  case and higher dimensions. The 2D bounds bear a closer resemblance to the 1D case studied in \cite{IT-qnls}; this reflects the fact that there is less dispersion in 2D, and in particular the $L^4$ Strichartz bound lies on the sharp Strichartz line and cannot be obtained  directly 
using interaction Morawetz based tools. By contrast, in three and higher dimension we are able to prove the $L^4$  bound directly from interaction Morawetz (and also a very mild application of $L_t^4L_x^3$ Strichartz estimate when $n=3$).

We also remark on the need to add translations to the bilinear 
$L^2$ estimates. This is because, unlike the linear bounds
\eqref{uk-ee} and \eqref{uk-Str} which are inherently invariant with respect to translations, bilinear estimates are not invariant 
with respect to separate translations for the two factors.
Hence, in particular, they cannot be directly transferred to 
nonlocal bilinear forms. However, if we allow translations, then
one immediate corollary of \eqref{uab-bi-bal} is that for any bilinear form $L$ with smooth and bounded symbol we have 
\begin{equation} \label{uab-bi-boot-trans}
\sup_{j\in\mathbb{Z}}\|P_j|D|^{\frac{3-n}2} L(u_{\lambda} \bu_{\lambda})  \|_{L^2_{t,x}} \lesssim  \epsilon^2 c_{\lambda}^2 \lambda^{-2s+\frac12}.
\end{equation}
This is essentially the only way we will use this translation invariance in our proofs. We also remark on the $(1+\lambda|x_0|)$ 
factor, which appears in the balanced bilinear bounds but not in the 
unbalanced ones. Heuristically this is because in the unbalanced
case transversality remains valid even when one of the metrics gets translated; on the other hand, in the balanced case we also estimate interactions of near parallel waves, and there transversality is very sensitive to changes of the metric, so there is a price to pay for translations. 
\medskip

We are now ready to state our local results, beginning  with the case of three and higher space dimensions:

\begin{theorem}\label{t:local-fe3}
a)  Let $n \geq 3$, $s > \frac{n+1}2$,  $\epsilon \ll 1$ and $T > 0$. Consider the equation \eqref{qnls} with cubic nonlinearity and let $u \in C[0,T;H^s]$ be a smooth solution with initial data $u_0$  which has $H^s$ size at most $\epsilon$.
 Let $\{ \epsilon  c_\lambda\}$ be a frequency envelope for the initial data  in $H^s$. Then the solution $u$ satisfies  the bounds \eqref{uk-ee}, \eqref{uab-bi-unbal}, \eqref{uab-bi-bal} and  \eqref{uk-Str}   uniformly with respect to $x_0 \in \R$.

b) The same result holds for \eqref{dqnls} but with $s > \frac{n+3}2$.
\end{theorem}

We continue with the case of two space dimensions, where we provide several variants of the result:

\begin{theorem}\label{t:local-fe2}
a) Let $n = 2$, $s > \frac32$, $\epsilon \ll 1$ and $T > 0$. Consider the equation \eqref{qnls} with cubic nonlinearity, and for which  one of the following three conditions holds:
\begin{itemize}
\item[(i)] either $s \geq 2$ and $T$ is arbitrary  large,
\item[(ii)] or $s >\frac32$ and $T \lesssim \epsilon^{-6}$,
\item[(iii)] or $s > \frac32$, the problem is conservative and $T$ is arbitrarily large.
\end{itemize}
Let $u \in C[0,T; H^s]$ be a smooth solution with initial data $u_0$  which has $H^s$ size at most $\epsilon$.
 Let $ \{\epsilon c_\lambda\}$ be a frequency envelope for the initial data  in $H^s$. Then the solution $u$ satisfies  the bounds \eqref{uk-ee}, \eqref{uab-bi-unbal},  \eqref{uab-bi-bal} and \eqref{uk-Str}  uniformly with respect to $x_0 \in \R$.

b) The same result holds for \eqref{dqnls} but with $s$ increased by one.

\end{theorem}

\medskip

The results in Theorem~\ref{t:local-fe3} and Theorem~\ref{t:local-fe2} [(i), (iii)] apply independently of the size of $T$, so they yield the global in time solutions in Theorems~\ref{t:global3}, \ref{t:global2},  \ref{t:global2c}, given an appropriate local well-posedness result.

 We continue with a theorem that applies to the linearized  equation, which will be essential in the proof of all of our  well-posedness results:
\begin{theorem}\label{t:linearize-fe}
Consider the cubic equation \eqref{qnls} for $n \geq 2$. 
Let $u$ be a solution as in Theorem~\ref{t:local-fe2} or Theorem~\ref{t:local-fe3}. Let $v$ be a solution 
to the linearized equation around $u$ with $L^2$ initial data 
$v_0$ and with frequency envelope $d_\lambda$. Then $v$ satisfies 
the following bounds:
\begin{enumerate}[label=(\roman*)]
\item Uniform frequency envelope bound:
\begin{equation}\label{uk-ee-lin}
\| v_\lambda\|_{L^\infty_t L^2_x} \lesssim d_\lambda 
\end{equation}
\item Balanced bilinear $(v,v)$-$L^2$ bound:
\begin{equation} \label{vvab-bi-bal-lin}
\sup_{j\in\mathbb{Z}}\|P_j|D|^{\frac{3-n}2} (v_\lambda \bv_\mu^{x_0})  \|_{L^2_{t,x}} \lesssim  d_{\lambda} d_\mu \lambda^{\frac12} (1+ \lambda |x_0|)
, \qquad \lambda\approx \mu 
\end{equation}
\item Unbalanced bilinear $(v,v)$-$L^2$ bound: If $v_{\lambda}$ and $v_{\mu}$ are unbalanced, then
\begin{equation} \label{vvab-bi-unbal-lin}
\| v_\lambda \bv_\mu^{x_0}  \|_{L^2_{t,x}} \lesssim  d_{\lambda} d_\mu\frac{\min\{\lambda,\mu\}^\frac{n-1}2}{(\lambda+\mu)^{\frac12}}.
\end{equation}
\item Balanced bilinear $(u,v)$-$L^2$ bound:
\begin{equation} \label{uvab-bi-bal-lin}
\sup_{j\in\mathbb{Z}}\|P_j|D|^{\frac{3-n}2} (v_\lambda \bu_\mu^{x_0})  \|_{L^2_{t,x}} \lesssim \epsilon d_{\lambda} c_\mu \lambda^{-s+\frac12} (1+ \lambda |x_0|),
\qquad \mu \approx \lambda 
\end{equation}
\item Unbalanced bilinear $(u,v)$-$L^2$ bound: If $v_{\lambda}$ and $u_{\mu}$ are unbalanced, then
\begin{equation} \label{uvab-bi-unbal-lin}
\| v_\lambda \bu_\mu^{x_0}  \|_{L^2_{t,x}} \lesssim \epsilon d_{\lambda} c_\mu \mu^{-s} \frac{\min\{\lambda,\mu\}^\frac{n-1}2}{(\lambda+\mu)^{\frac12}}  
\end{equation}
\end{enumerate}
\end{theorem}

This theorem in particular yields $L^2$ well-posedness for the linearized equation.

\subsection{The bootstrap hypotheses}

To prove the above theorems, we make a bootstrap assumption where we assume the same bounds but with a worse constant $C$, as follows:

\begin{enumerate}[label=(\roman*)]
\item Uniform frequency envelope bound:
\begin{equation}\label{uk-ee-boot}
\| u_\lambda \|_{L^\infty_t L^2_x} \leq C\epsilon c_\lambda \lambda^{-s}
\end{equation}
\item Unbalanced bilinear $L^2$ bounds: If $u_{\lambda}$ and $u_{\mu}$ are unbalanced and $\mu\leq\lambda$, then
\begin{equation} \label{uab-bi-unbal-boot}
\| u_\lambda \bu_\mu^{x_0}  \|_{L^2_{t,x}} \leq C^2 \epsilon^2 c_\lambda c_\mu \lambda^{-s-\frac12} \mu^{-s+\frac{n-1}2},
\end{equation}
\item Balanced bilinear $L^2$ bound: 
\begin{equation} \label{uab-bi-bal-boot}
\sup_{j\in\mathbb{Z}}\|P_j|D|^{\frac{3-n}2} (u_\lambda \bu_\mu^{x_0})  \|_{L^2_{t,x}} \leq C^2\epsilon^2 c_\lambda c_\mu \lambda^{-2s+\frac{1}{2}}(1+ \lambda |x_0|)
, \qquad \lambda \approx \mu .
\end{equation}
\end{enumerate}
As noted earlier, we observe that in dimension $n\geq 4$ the estimate \eqref{uab-bi-bal-boot} implies an $L^4$ bound for $u_\lambda$ (with an arbitrarily small dyadic loss). The same $L^4_{t,x}$ estimate is also true when $n=3$ by interpolating \eqref{uab-bi-bal-boot} with the $L_t^4L_x^3$ Strichartz bound. When $n = 2$, one obtains a $L_{t,x}^4$ bound with a $\frac{1}{2}$ derivative loss. To incorporate these low dimensional cases into our analysis, we will add to the 
list above, the following $L^4$ bootstrap bounds for $u_\lambda$, which corresponds to the estimate \eqref{uk-Str}:
\begin{enumerate}[resume]
    \item $L^4$ Strichartz bound for $n = 2$:
    \begin{equation}\label{uk-Str2-boot}
\| u_\lambda\|_{L^4_{t,x}} \lesssim C\epsilon c_\lambda \lambda^{-s+ \frac12}, \qquad n = 2. 
\end{equation}
\item $L^4$ bound for $n=3$
\begin{equation}\label{uk-Str3-boot}
\| u_\lambda\|_{L^4_{t,x}} \lesssim_{\delta} C\epsilon c_\lambda \lambda^{-s+ \frac14+\delta}, \qquad n = 3. 
\end{equation}
\end{enumerate}

Then we seek to improve the constant in these bounds. The gain will come from the fact that the $C$'s will always come paired with extra $\epsilon$'s. Precisely, a continuity argument shows that 

\begin{proposition}\label{p:boot}
It suffices to prove Theorems~\ref{t:local-fe3}, \ref{t:local-fe2} under the bootstrap assumptions \eqref{uk-ee-boot}, \eqref{uab-bi-unbal-boot} and
\eqref{uab-bi-bal-boot}, together with \eqref{uk-Str2-boot} in two dimensions, respectively \eqref{uk-Str3-boot} in three dimensions.
\end{proposition}

We remark that the Strichartz bounds \eqref{uk-Str}
are in general not part of this bootstrap loop, with the exception
of \eqref{uk-Str2-boot} and \eqref{uk-Str3-boot} as noted above. Similar bootstrap assumptions are made for $v$ in the proof of Theorem~\ref{t:linearize-fe} in Section~\ref{s:lin}.

\section{A paradifferential and resonant  expansion of the equation}
\label{s:lin}

 In this section, we consider the evolution of each dyadic piece $v_{\lambda}$ 
of a solution to either the linearized or nonlinear equation. Here, we will isolate three key components of the resulting equation,

\begin{enumerate}
\item
The paradifferential part, which extracts the leading quasilinear portion of the equation. This is an essential component in establishing both the local and global results in this paper.

\item The doubly resonant, semilinear part of the nonlinearity. This is perturbative for our local results, but cannot be treated perturbatively for the global result in Theorem~\ref{t:global2c} in two dimensions. 

\item The weakly transversal part of the nonlinearity. This term is new in the ultrahyperbolic setting. We will just barely be able to treat this term perturbatively in the 2D global result. However, doing so is somewhat delicate and will rely on the new, unbalanced bilinear estimate which we prove in Section 7.
\end{enumerate}

A first approximation of the paradifferential equation can be obtained by simply truncating the coefficients in the principal part
to low frequencies,
\begin{equation}\label{paraT}
i \partial_t v + \partial_j T_{g^{jk}(u)} \partial_k v 
= f.
\end{equation}
Unfortunately, this form is not sufficiently precise for our long-time results. This is because the low-frequency part of the metric $g(u)$ can also contribute doubly resonant quadratic contributions. This will motivate us to make a better choice for the inhomogeneous term $f$ in the equation.  

Instead of working directly with the full paradifferential equation, it is easier to work with the evolution of each individual dyadic piece $v_{\lambda}$. The resulting equation takes the form
\begin{equation}\label{para}
i \partial_t v_\lambda + \partial_j g^{jk}_{[<\lambda]} \partial_k v_\lambda 
= f_\lambda, \qquad v_\lambda(0) = v_{0,\lambda},
\end{equation}
Here, we define the frequency truncated metric by the relation 
\begin{equation}
g_{[<\lambda]} := P_{<\lambda} g (u_{<\lambda}).
\end{equation}
Here we note that we have localized $u$ to lower frequencies, 
rather than directly localizing the metric $g(u)$; this will be important for later excluding the doubly resonant interactions from the leading paradifferential part of the equation. 
\bigskip

Using the above conventions, we can rephrase both the nonlinear and linearized evolutions as paradifferential equations, in the form
\begin{equation}\label{para-full}
i \partial_t u_\lambda + \partial_j g^{jk}_{[<\lambda]} \partial_ku_\lambda 
= N_\lambda(u), \qquad u_\lambda(0) = u_{0,\lambda},
\end{equation}
respectively
\begin{equation}\label{para-lin}
i \partial_t v_\lambda + \partial_j g^{jk}_{[<\lambda]} \partial_kv_\lambda 
= N^{lin}_\lambda v, \qquad v_\lambda(0) = v_{0,\lambda},
\end{equation}

The corresponding nonlinear and linearized source terms $N_\lambda(u)$ and $N^{lin}_\lambda$ will described more explicitly later. They will play a perturbative role in our local results, but we will need a more precise resonant expansion for the global 2D results, which we will describe next.

\bigskip

As alluded to above, for the 2D global results, there are two portions of the nonlinearity which we will need to analyze more carefully. The first is the doubly resonant cubic part, which plays a nonperturbative role. The second term contains the resonant, weakly transversal interactions, which will ultimately be perturbative but will require the use of the new unbalanced bilinear estimates that we establish in Section 7. We separate these terms from $N_\lambda$ by first rewriting
 the frequency localized evolution \eqref{para-full} in the form
\begin{equation}\label{full-lpara0}
i \partial_t u_\lambda + \partial_j g^{jk}_{[<\lambda]} \partial_ku_{\lambda} = C_\lambda(\bfu,\bfu, \bfu) + F_\lambda(u) ,
\end{equation}
where we are adopting the notation $\mathbf{u} =\left\{ u, \bar{u}\right\}$ from the discussion in Section~\ref{s:resonances}. Here, we write $F_{\lambda}$ to generically denote parts of the nonlinearity that we will be able to treat perturbatively (for both our local and global results). For high frequencies $\lambda \gg 1$, we can further decompose the cubic part into four sub-components
\begin{equation}\label{C-split}
C_\lambda(\bfu,\bfu, \bfu) = C_\lambda^{res}(u,\bu,u) + 
C_\lambda^{nr}(\bfu,\bfu,\bfu) + C_\lambda^{tr}(\bfu,\bfu,\bfu)+C_\lambda^{wt}(\bfu,\bfu,\bfu).
\end{equation}
Here, the last two terms correspond to the transversal and weakly transversal interactions. As mentioned above, we will ultimately be able to absorb these into the term $F_\lambda$ (although with some more care needed to address the weakly transversal part).
Correspondingly, we arrive at the final expansion 
\begin{equation}\label{full-lpara}
i \partial_t u_\lambda + \partial_x g_{[<\lambda]} \partial_x u_{\lambda} = C_\lambda^{res}(u,\bu,u) +  C^{nr}_\lambda(\bfu,\bfu, \bfu) + F_\lambda(u), 
\end{equation}
where the first two components on the right represent terms as follows:
\begin{itemize}
    \item $C_\lambda^{res}(u,\bu,u)$ contains balanced cubic terms
    with phase rotation symmetry, including the doubly resonant interactions
    \item  $C^{nr}_\lambda(\bfu,\bfu, \bfu)$ contains balanced nonresonant terms.
\end{itemize}
As in \cite{IT-qnls2}, to obtain good quantitative bounds, we make an explicit choice for $C_\lambda^{res}$ by introducing the symmetric symbol $c_{diag}$, which is smooth on the corresponding dyadic scale
so that
\[
c_{diag}(\xi_1,\xi_2,\xi_3) := \left\{
\begin{aligned}
1 & \qquad \sum |\xi_i - \xi_j| \ll \la \xi_1\ra + \la \xi_2 \ra 
+ \la \xi_3 \ra 
\\
0 & \qquad  \sum |\xi_i - \xi_j| \gtrsim \la \xi_1\ra + \la \xi_2 \ra 
+ \la \xi_3 \ra \ .
\end{aligned}
\right.
\]
Then, the cubic doubly resonant part of the nonlinearity can be described in terms of the symbol $c(\cdot \, , \, \cdot \, , \, \cdot)$ of the form in \eqref{eq-cubic} as 
\begin{equation}\label{defcl}
c^{res}_\lambda(\xi_1,\xi_2,\xi_3) = \widehat{P}_\lambda(\xi_1-\xi_2+\xi_3) \, c_{diag}(\xi_1,\xi_2,\xi_3) c(\xi_1,\xi_2,\xi_3).
\end{equation}
If the symbol $c$ is conservative, we further note that $c_{\lambda}^{res}$ is real-valued on the diagonal,
\[
\Im c_\lambda^{res}(\xi,\xi,\xi) = 0 , \qquad \xi \in \R^2.
\]

We end this section by briefly discussing the term $F_\lambda(u)$, which contains the remaining terms in the paradifferential equation for $u_\lambda$. This term contains interactions which can be essentially classified as follows: 

\begin{enumerate}[label=(\roman*)]
\item cubic $ high \times high$ terms of the form 
\[
\mu^2 L(\bfu_{\leq \mu}, \bu_\mu,u_\mu) \qquad \mu \gtrsim \lambda,
\]

\item cubic $low \times high$ commutator terms of the form
\[
\lambda L(\partial_x \bfu_{\ll \lambda}, \bfu_{\ll\lambda}, u_\lambda),
\]

\item cubic terms with only transversal or weakly transversal interactions of the form
\[
\lambda^2 L(u_\lambda,u_\lambda,u_\lambda),
\qquad 
\lambda^2 L(u_\lambda,\bu_\lambda,\bu_\lambda),
\qquad 
\lambda^2 L(\bu_\lambda,\bu_\lambda,\bu_\lambda),
\qquad 
\]

 \item all cubic terms at frequency $\lambda \leq 1$,
 
\item quartic and higher terms of the form 
\[
\mu^2 L(\bfu_{\leq \mu}, \bfu_{\leq \mu}, \bu_\mu,  u_\mu) \qquad \mu \gtrsim \lambda .
\]

\end{enumerate}
In our analysis, the $F_\lambda$ terms will play a perturbative role, though establishing that is not at all immediate, and is instead one of the main difficulties.

\section{Balanced bilinear estimates: the constant coefficient case}\label{s:im-cc}

This section is devoted to understanding the balanced bilinear estimates in the simpler case of the homogeneous, linear ultrahyperbolic Schr\"odinger flow
\begin{equation}\label{eq:flat1}
i u_t + \Delta_g u = 0, \qquad u(0) = u_0,
\end{equation}
with a constant non-degenerate metric $g$. This will provide the basic building blocks for establishing the analogous estimates for both the linear paradifferential and nonlinear equations in later sections.
\medskip 

The key step in obtaining our bilinear bounds will be to establish a suitable family of interaction Morawetz-type estimates. The choice of the associated interaction functional is largely motivated by the density-flux identities for the mass and momentum. In the elliptic case where $g=Id$, the approach used in the article \cite{IT-qnls2} is somewhat inspired by the choices in \cite{PV} (although with different choices for the fluxes which are better tailored to the nonlinear problem). However, when the metric $g$ is merely non-degenerate, the approach used in the above papers is not sufficient for our purposes. Indeed, using their approach, the corresponding estimate gives control of the quantity 
\begin{equation}\label{noncoercive}
\||D|^{-\frac{n+1}{2}}\Delta_g|u|^2\|_{L^{2}_{t,x}}
\end{equation}
which is not suitably coercive. This is in sharp contrast to when $g=Id$. Below, we propose a new method that will allow us to estimate the slightly weaker homogeneous Besov-type norm
\begin{equation}\label{Besovnorm}
\sup_{j\in\mathbb{Z}}\|P_{j}|D|^{-\frac{n-3}{2}}|u|^2\|_{L^2_{t,x}}.
\end{equation}
\begin{remark}
We note importantly that the supremum is taken over all $j\in\mathbb{Z}$ and not just $j\geq 0$
\end{remark}
Up to an arbitrarily small dyadic loss, this bound is nearly as strong as the one which one obtains in the elliptic case, and will entirely suffice for our purposes. We also remark that the approach used below appears to be robust enough to apply to variable metrics that arise as possibly large perturbations of a constant metric. Therefore, we suspect that this approach may be useful in large data problems.

\subsection{Conservation laws in the flat case}

For the linear ultrahyperbolic Schr\"odinger equation
\begin{equation}\label{eq:flat}
i u_t + \Delta_g u = 0, \qquad u(0) = u_0,
\end{equation}
one has the following formally conserved quantities:
\[
\bM(u) = \int |u|^2 \,dx,
\]
\[
\bP^j(u) = 2 \int \Im (\bar u \partial^j u) \,dx.
\]
which are the mass and momentum, respectively. Above, and in the sequel, indices are raised with respect to the metric $g$.

These conservation laws are naturally associated to the densities
\[
M(u) = |u|^2, \qquad P_j(u) = i ( \bar u \partial_j u - u \partial_j \bar u).
\]
which gives rise to the density-flux identities
\begin{align}
    &\label{df-lin}
\partial_t M(u) = \partial_j P^j(u),\\
\
\
& \partial_t P^j (u) = \partial_m E^{jm}(u),
\end{align}
where we choose
\[
E^{jm}(u) :=  \partial^m u \partial^j \bar u +  \partial^j u \partial^m \bar u
- v \partial^j  \partial^m \bar v - \bar v 
\partial^j \partial^m v.
\]
The symbols of these densities viewed as bilinear forms  are
\[
m(\xi,\eta) = 1, \qquad p^j(\xi,\eta) = -g^{jk}(\xi_k+\eta_k),
\]
respectively
\[
\qquad e^{jm}(\xi,\eta) = g^{jk}g^{lm}(\xi_k+\eta_k)(\xi_l +\eta_l).
\]


\subsection{Interaction Morawetz identities in the flat case}
Our starting point is to consider the general interaction Morawetz functional,
\begin{equation}\label{Ia-sharp-def-vv-flat}
\bI(u,v) :=   \iint a_j(x-y) (M(u)(x) P_j(v) (y) -  
P_j(u)(x) M(v) (y)) \, dx dy,
\end{equation}
where $a_j$ is a suitable weight to be chosen.
Assuming that $u,v$ solve the equation \eqref{eq:flat}, the time derivative of $\bI(u,v)$ is 
\begin{equation}\label{I-vv-flat}
\frac{d}{dt} \bI(u,v) =  \bJ^4(u,v) ,
\end{equation}
where
\[
\bJ^4(u,v) = \int  \partial_m a_j(x-y) \left(M(u) E^{jm}(v) - P^j(u) P^m(v) + M(v) E^{jm}(u)
-  P^j(v) P^m(u)\right)\, dxdy.
\]
Following the analogous computation in \cite{IT-qnls2}, one obtains the formula (noting carefully the raised indices),
\begin{equation}\label{J4-flat}
\bJ^4(u,v) = 2 \int a_{jm}(x-y) F^j \bar F^m\,  dx dy,
\end{equation}
where
\begin{equation}\label{Fj-flat}
F_{j}(x,y) = u(x) \partial_j \bar v(y) + 
\partial_j u(x) \bar v(y).
\end{equation}
Now we consider the choice for $a_j$. In order to simultaneously ensure that $\mathbf{J}^4(u,v)$ is sufficiently coercive and that $\mathbf{I}(u,v)$ is bounded uniformly in $t$, it is natural to take
\[
a_j (x) = \partial_j a, \qquad \partial_m a_j = a_{mj},
\]
where $a$ is a well-chosen bounded, (possibly degenerate) convex function. To motivate our forthcoming calculations, it is instructive first to outline what goes wrong when using the standard weight $a(x)=|x|$, which was the choice made in \cite{IT-qnls2}. Indeed, by diagonalizing and using the identity
\begin{equation}\label{algebraicid}
| u(x) \partial_j \bar v(y) + 
\partial_j u(x) \bar v(y)|^2 
= | u(x) \partial_j  v(y) -
\partial_j u(x) v(y)|^2 + 
\partial_j |u(x)|^2
 \partial_j |v(y)|^2,
\end{equation}
we would obtain in particular, the bound
\begin{equation}\label{J4-flat-u=v}
\bJ^4(u,u) \geq \int a_{jm}(x-y) \partial^j |u(x)|^2
 \partial^m |u(y)|^2 \, dx dy
\end{equation}
when $u=v$. The choice $a(x)=|x|$ then leads to the identity
\begin{equation}\label{IM-clean}
    \int a_{jm}(x-y) \partial^j |u(x)|^2
 \partial^m |u(y)|^2 \, dx dy = c_n \| |D|^{-\frac{n+1}2} \Delta_g |u|^2 \|_{L^2}^2,
\end{equation}
which can be verified by interpreting $a_{jm}$ in \eqref{J4-flat-u=v} as kernels for suitable multipliers as well as Plancherel's theorem. This finally  leads to 
\begin{equation}\label{J4-flat-u=v+}
\bJ^4(u,u) \gtrsim  \| |D|^{-\frac{n+1}2} \Delta_g |u|^2 \|_{L^2}^2.
\end{equation}
As alluded to earlier, this estimate is useful in the case where $g=Id$ because it immediately yields control of the quantity
\begin{equation*}
\||D|^{\frac{3-n}{2}}|u|^2\|_{L^2}^2.
\end{equation*}
Unfortunately, in the non-elliptic case, \eqref{J4-flat-u=v+} is not sufficient. 
\medskip

To rectify this issue, our strategy will be to correct the weight $a(x)$ so that it gains a certain degree of uniform convexity in the radial direction (at least in a fixed dyadic region). The benefit of this is that it will (in principle) allow us to directly exploit the convexity of the weight in the expression \eqref{J4-flat}. More precisely, for each dyadic number $r>0$, we will define a ($r$-dependent) weight $a_r$ such that in the dyadic spatial region
$|x| \approx r$ we have $D^2 a_r \gtrsim r^{-1}$. An elegant choice for this purpose is the family of weights
\begin{equation*}
a_{r}(x)=|x|+\frac{r}{2}e^{-r^{-1}|x|}
\end{equation*}
Clearly, we have the uniform in $r$ bound for the derivative of $a_r$,
\begin{equation*}
|Da_r|\lesssim 1.
\end{equation*}
Moreover, it is easy to compute the eigenvalues of the Hessian $D^2a_r$,
\begin{equation*}
\lambda_1=\cdot\cdot\cdot=\lambda_{n-1}=\frac{1}{|x|}-\frac{1}{2|x|}e^{-r^{-1}|x|},\hspace{5mm}\lambda_n=\frac{1}{2r}e^{-r^{-1}|x|}.
\end{equation*}
which by straightforward calculation yields the Gaussian lower bound
\begin{equation*}
\min_i\lambda_i\gtrsim \frac{1}{r}e^{-r^{-2}|x|^2}:=b_r(x).
\end{equation*}
Therefore, if we write $G_j:=F^j:=g^{ij}F_j$ (i.e. $G$ is the vector $F$ with raised indices), we obtain the powerful estimate
\begin{equation}\label{lowerboundultra}
\mathbf{J}^4(u,v)\gtrsim r^{-1}\int e^{-r^{-2}|x-y|^2}|G|^2\gtrsim r^{-1}\int e^{-r^{-2}|x-y|^2}|F|^2
\end{equation}
where the second estimate follows from (merely) the non-degeneracy of the (flat) metric $g$. We have the following simple lower bound for $\widehat{b}$
\begin{equation*}
\widehat{b}_r(\xi)=c_nr^{n-1}e^{-c_nr^2|\xi|^2}\gtrsim \widehat{P}_{r^{-1}}(\xi)|\xi|^{1-n}
\end{equation*}
Therefore, at this stage, we can invoke \eqref{algebraicid} to obtain the Besov type bound,
\begin{equation}\label{balancedflatbound}
\|P_{r^{-1}}|D|^{\frac{3-n}{2}}|u|^2\|_{L_{t,x}^2}^2\lesssim \int_{0}^{t}J(u,u)ds \lesssim \|u\|_{L_t^{\infty}L_x^2}^3\|u\|_{L_t^{\infty}H_x^1}.
\end{equation}
which is uniform in $r$.
\begin{remark}
If $u=u_{\lambda}$ is localized at frequency $\lambda\geq 1$, then by dyadic summation (separately over frequencies $\lesssim 1$ and $\gtrsim 1$) we obtain the following bound with small summability loss,
\begin{equation*}
\||D|^{\frac{3-n}{2}+\delta}|u_{\lambda}|^2\|_{L_{t,x}^2}\lesssim_{\delta} \lambda^{\frac{1}{2}+\delta}\|u_{\lambda}\|_{L_t^{\infty}L_x^2}^2
\end{equation*}
for every $0<\delta\ll 1$. This is only slightly weaker than the bound one obtains in the elliptic case, but applies in far greater generality.
\end{remark}
\begin{remark}\label{uniformnondegremark}
We note the remarkable property that the bound \eqref{lowerboundultra} actually holds for any uniformly non-degenerate metric $g$ (with implicit constant depending on $g$), which makes the above approach very robust. The price one pays is an arbitrarily small (logarithmic) dyadic loss, which will be harmless for our purposes.
\end{remark}
In addition to the above, we will also be interested in frequency localized versions of the mass and momentum densities. Given a dyadic frequency $\lambda$, we start with a symbol $a_\lambda(\xi,\eta)$ which is symmetric, in the sense that
\[
a_\lambda(\eta,\xi) = \ol{a_\lambda(\xi,\eta)},
\]
and localized at frequency $\lambda$, and then define an associated weighted mass density by
\[
M_\lambda(u) = A_\lambda(u,\bar u).
\]
We also define corresponding momentum symbols $p^j_{\lambda}$  by 
\[
p_\lambda^j(\xi,\eta) = -g^{jk}(\xi_k+\eta_k)\, a(\xi,\eta).
\]
Then a direct computation yields the density flux relations
\[
\frac{d}{dt} M_\lambda(u,\bar u) = \partial_j P_{\lambda}^j(u,\bar u), \qquad \frac{d}{dt} P_{\lambda}^j(u,\bar u) = \partial_m E^{mj}_{\lambda}(u,\bar u).
\]

\section{Density-flux relations for mass and momentum}
\label{s:df}

In this section, we establish the relevant density-flux identities 
for the paradifferential and the nonlinear flows. These will give us the starting point for 
interaction Morawetz analysis in the next section. The corresponding identities derived in the definite case in \cite{IT-qnls2} do not rely on the ellipticity of $\Delta_g$. Therefore, we will just state the requisite identities and refer the reader to \cite{IT-qnls2} for more detailed computations.

\subsection{ Density-flux identities for the paradifferential problem}
\
Using the flat case as a model, we now turn our attention to considering solutions $v_{\lambda}$ to the paradifferential equation \eqref{para}. Again, we will only recall the relevant identities below and refer to \cite{IT-qnls2} for the requisite calculations. First, we have the mass-momentum identity
\[
\begin{aligned}
\frac{d}{dt} M(v_\lambda) &= -2 \partial_j [g^{jl}_{[<\lambda]} \Im (\partial_l v_\lambda   \cdot \bar v_\lambda)] 
+ 2 \Im ( f_\lambda \bar v_\lambda).
\end{aligned}
\]
To write this in a cleaner form, we define the covariant momenta, 
\[
P^j (v_\lambda) = g^{jl}_{[<\lambda]} P_l(v_\lambda).
\]
This allows us to rewrite the mass-momentum identity in the form
\begin{equation}\label{dens-flux-param}
\frac{d}{dt} M(v_\lambda) =  \partial_j [ P^j(v_\lambda)] + F^{para}_{\lambda,m}, \qquad F^{para}_{\lambda,m} = 2 \Im ( f_\lambda \bar v_\lambda).
\end{equation}

Similarly, with $g:= g_{[<\lambda]}$ and $v:= v_\lambda$, one computes the momentum-energy identity,
\begin{equation*}
\begin{aligned}
\frac{d}{dt} P_j(v_\lambda) =  \partial_k [ E^k_j(v_\lambda)] + F^{para}_{j,\lambda,p},
\end{aligned}
\end{equation*}
where
\begin{equation}
 E^k_j(v_\lambda) :=    2  \Re (  g^{kl} \partial_l v \partial_j \bv -
v \partial_j  g^{kl} \partial_l   \bv),
\end{equation}
and
\begin{equation}
    F^{para}_{j,\lambda,p} :=  2\Re ( \partial_k  v (\partial_j g^{kl}) \partial_l \bar v)
- 2\Re (f \partial_j \bv)
+ 2 \Re ( v \partial_j \bar f),
\end{equation}
which we can write in the covariant form
\begin{equation}\label{dens-flux-parap}
\begin{aligned}
\frac{d}{dt} P^j(v_\lambda)
= & \ \partial_k [ E^{kj}(v_\lambda)] + F^{para,j}_{\lambda,p},
\end{aligned}
\end{equation}
where for $E^{kj}$ we have raised indices with respect to $g$,
\begin{equation} \label{Ekj-para}
 E^{kj}(v_\lambda) =    2  \Re (   \partial^k v_\lambda \partial^j \bv_\lambda -
v_\lambda \partial^j  \partial^k   \bv_\lambda),
\end{equation}
but for the source term we have additional commutator terms,
\begin{equation}
    F^{j,para}_{\lambda,p} =   G^{j,para}_{\lambda,p}
- 2\Re (f_\lambda \partial^j \bv_\lambda)
+ 2 \Re ( v_\lambda \partial^j \bar f_\lambda),
\end{equation}
where the quadratic term $G^{j,para}_{\lambda,p}$
is given by 
\begin{equation}\label{G4}
   G^{j,para}_{\lambda,p} =  \partial_k v_\lambda   (\partial^j g^{kl}_{[<\lambda]}) \partial_l \bar v_{\lambda} + (\partial_t g^{jk}_{[<\lambda]}) P_k(v_\lambda) - (\partial_k g^{kl}_{[<\lambda]})E^j_l(v_\lambda).
\end{equation}
The exact form of this term is not particularly important, as it will be perturbatively estimated in $L^1$ in the next section. Here we note that this term can be schematically interpreted in the form
\begin{equation}
   G^{j,para}_{\lambda,p}  = h(\bfu_{<\lambda}) \bfu_{<\lambda} \partial \bfu_{<\lambda} \partial v_{\lambda}    \partial \bv_{\lambda},
\end{equation}
which is all that will be needed for our analysis.
\bigskip

\subsection{Density flux identities for the frequency localized mass and momentum for the nonlinear problem} 

The density-flux identities for the paradifferential problem will be sufficient for our results in dimension $n \geq 3$. However, a finer analysis is needed in the most interesting case when $n=2$, especially for the proof of the global result in Theorem~\ref{t:global2c}. This will be the purpose of this subsection.
We remark that the analysis below is only important for treating large frequencies, where
the $L^4$ Strichartz estimate contributes a loss that is too big to allow for handling the resonant cubic terms perturbatively. As in the previous subsection, the algebraic identities we derive here do not rely on the ellipticity of the underlying metric. Therefore, we will only recall the necessary identities and refer to \cite{IT-qnls2} for more details.

 We begin by observing that for a localized solution $u_\lambda$ to the nonlinear problem, we can use the equation \eqref{full-lpara} to compute
\[
\begin{aligned}
\frac{d}{dt} M_\lambda(u) =  2 \partial_j [g_{[<\lambda]}^{jk} \Im (\partial_k u_\lambda   \cdot \bar u_\lambda)] 
+ 2 \Im ( (C_\lambda^{res}(u,\bu, u)+  C_\lambda^{nr}(\bfu,\bfu, \bfu) + F_\lambda(\bfu)) \bar u_\lambda),
\end{aligned}
\]
which can be rewritten as
\begin{equation}\label{dens-flux-m0}
\frac{d}{dt} M_\lambda(u) =  \partial_j [ P_\lambda^j(u)] 
+ C^{4,res}_{\lambda,m}(u,\bu,u,\bu) + C^{4,nr}_{\lambda,m}(\bfu,\bfu,\bfu,\bfu) + 2 F^4_{\lambda,m}(\bfu),
\end{equation}
where 
\[
C^{4,res}_{\lambda,m} (u,\bu,u,\bu) := 2 \Im ( (C_\lambda^{res}(u,\bu, u)  \bu_\lambda), 
\qquad 
C^{4,nr}_{\lambda,m}(\bfu,\bfu,\bfu,\bfu) := 2 \Im ( C_\lambda^{nr}(\bfu,\bfu, \bfu)\bu_\lambda)),
\]
\[
 F^4_{\lambda,m} (\bfu) := 2 \Im ( (F_\lambda(\bfu)  \bu_\lambda).
\]
The three source terms in \eqref{dens-flux-m0} will need to be treated more carefully in the last section.

The term $C^{4,res}_{\lambda,m}$ contains only interactions which are localized at frequency $\lambda$, but also includes the non-transversal interactions where all frequencies are equal. We remark importantly that the conservative assumption in Definition~\ref{d:conservative2}, yields the key symbol identity on the diagonal
\begin{equation}\label{c4-cancel}
   c^4_{\lambda,m}(\xi,\xi,\xi,\xi) = 0,\hspace{5mm}\xi\in\mathbb{R}^2 .
\end{equation}
This was also assumed in the definite case in \cite{IT-qnls2}. One can also compare this with the one-dimensional results in \cite{IT-global},
\cite{IT-qnls}, where a stronger vanishing condition on the diagonal was imposed. In the 1D case, this enabled the removal of this term via a quartic correction to the mass density and a quartic correction to the mass flux (up to perturbative errors). In two dimensions, this strategy is not tractable since there are many more resonant interactions to deal with. Nevertheless, as in the definite case, the above vanishing condition will roughly allow us to obtain a representation
\[
C^{4,res}_{\lambda,m} (u,\bu,u,\bu) \approx \lambda L(\partial |u|^2, u, \bu)
+ \lambda \partial R^4_{\lambda,m}(u,\bu,\bu,\bu),
\]
which in turn will be used in order to treat it as a perturbative term, by separating it into an $L^1_{t,x}$ component and a flux correction.

The term $C^{4,nr}_{\lambda,m}$ contains a larger range of possible interactions  which are also localized at frequency $\lambda$.
However, these interactions are nonresonant, which will enable us to remove them as in \cite{IT-qnls2} using both density and flux corrections
\[
C^{4,nr}_{\lambda,m} (\bfu,\bfu,\bfu,\bfu) \approx 
\partial_t L(\bfu,\bfu,\bfu,\bfu)
+ \lambda \partial_x L(\bfu,\bfu,\bfu,\bfu) + F^{4,nr}(\bfu)
\]
with a perturbative source term $F^{4,nr}(\bfu)$.

Finally, $F^4_{\lambda,m}$ is at least quartic and only involves either weakly transversal, double transversal interactions or one pair of transversal interactions and one balanced, All of these interactions will be directly estimated as perturbations in $L^1$. As mentioned earlier, weakly transversal interactions are a new feature for the ultrahyperbolic problem. Treating these perturbatively will be somewhat delicate (as each factor is roughly localized in the same dyadic annulus) and will require us to use the more refined unbalanced bilinear estimate that we will establish in the next section.

Similarly to the above, we also have a
 density-flux relation for the momentum
\begin{equation}\label{dens-flux-p}
\frac{d}{dt} P^j_{\lambda}(u) =  \partial_k  E^{jk}_{\lambda}(u) 
+ C^{j,4,res}_{\lambda,p}(u,\bu,u,\bu)  +  C^{4,nr}_{\lambda,p}(\bfu,\bfu,\bfu,\bfu)+
 F^{j,4}_{\lambda,p}(\bfu)
\end{equation}
with a cancellation relation similar to \eqref{c4-cancel} and an analogous set of density and flux corrections for the second and third terms on the right.


\section{Interaction Morawetz bounds for the linear paradifferential flow}
\label{s:para}

In this section and the next, our aim is to study the linear paradifferential equation
corresponding to a solution $u$ to the nonlinear equation \eqref{qnls}. The primary focus of this section will be to establish the requisite $L^2$ and bilinear $L^2$ bounds at the paradifferential level.
On one hand, this is a key step in the proof of both the local well-posedness results in Theorem~\ref{t:local3+} and also provides a key building block for a large portion of the analysis for the global results in Theorem~\ref{t:global2}, Theorem~\ref{t:global3}. We will use an enhanced version of these arguments later 
to establish estimates for the full nonlinear equation and to prove the more delicate, low regularity global result in Theorem~\ref{t:global2c}. The estimates established here will be supplemented by the next section where we provide a suitable set of Strichartz estimates for the paradifferential equation with derivative loss. These latter estimates will be essential in the two-dimensional problem to establish a global $L_{t,x}^4$ Strichartz bound with a $\frac{1}{2}$ derivative loss, but also will play a small role in the three-dimensional problem when establishing an almost lossless $L_{t,x}^4$ bound. This is because the analysis here will only give us access to a a slightly worse $L_t^4L_x^{4+\delta}$ bound in the 3-dimensional setting. By analogous reasoning, we also have a near lossless $L_t^4L_x^8$ estimate in two dimensions, which will be relevant for establishing the $\epsilon^{-6}$ lifespan bound in our local result.

To set the stage, let us consider a one parameter family of functions $v_\lambda$ which solve the equations \eqref{para}. At this point, we assume no connection between these functions. Later we will apply these bounds in the case when  $v_\lambda=P_\lambda u$ or $v_\lambda = P_\lambda v$ where $v$ solves the linearized equation.

For these functions we have the conservation laws
\begin{equation}\label{df-v-m}
\partial_t M(v_\lambda) = \partial_j  P^j(v_\lambda)+ F^{para}_{\lambda,m},
\end{equation}
respectively
\begin{equation}\label{df-v-p}
\partial_t P^j(v_\lambda) = \partial_k  E^{kj} (v_\lambda) + F^{j,para}_{p,\lambda},
\end{equation}
where the source terms are 
\begin{equation}\label{f-v-m}
F^{para}_{\lambda,m}
:=
2\Im ( f_\lambda \bv_\lambda) ,  
\end{equation}
respectively 
\begin{equation}\label{f-v-p}
F^{j,para}_{p,\lambda}:=  G^{j,para}_{p,\lambda}
- 2\Re (f_\lambda \partial^j \bv_\lambda)
+ 2 \Re ( v_\lambda \partial^j \bar f_\lambda)
\end{equation}
with $G^{j,para}_{p,\lambda}$ given by \eqref{G4}.

We will measure each $v_\lambda$ based on the size of the initial data
and of the source term $f_\lambda$. However, following the previous results \cite{IT-qnls}, \cite{IT-qnls2}, rather than measure $f_\lambda$ 
directly, it is useful to instead measure its interaction with $v_\lambda$ and as well as its translates. For this, we define
\begin{equation}\label{dl}
d_\lambda^2 := \sup_{\mu,\nu \approx \lambda} \| v_\mu(0)\|_{L^2}^2 +  \sup_{x_0 \in \R} \| v_\mu f_\nu^{x_0} \|_{L^1_{t,x}}.
\end{equation}
Using this slightly "weaker" parameter $d_\lambda$ is important for allowing us to more easily study the bilinear interactions between a solution and the source terms in the equation.

The main result in this section asserts that we can obtain energy and bilinear $L^2$ bounds for $v_\lambda$ (with a small loss in the balanced estimate). Precisely, we have the following:
\begin{theorem}\label{t:para}
Assume that $u$ solves \eqref{qnls} in a time interval $[0,T]$,
and satisfies the bounds \eqref{uk-ee},\eqref{uab-bi-unbal}, \eqref{uab-bi-bal} and \eqref{uk-Str} in the same time interval for some $s > \frac{n+1}2$. Moreover, if $n=2$, we also assume that $T \lesssim \epsilon^{-6}$. Assume $v_\lambda$ solve \eqref{para}, and let $d_\lambda$ be as in \eqref{dl}. Then the following bounds hold for the functions $v_\lambda$ uniformly in $x_0 \in \R$:

\begin{enumerate}
    \item Uniform energy bounds:
\begin{equation}\label{v-ee}
\| v_\lambda \|_{L^\infty_t L^2_x} \lesssim d_\lambda 
\end{equation}

\item Balanced bilinear $(u,v)$-$L^2$ bound: There holds
\begin{equation} \label{uv-bi-bal}
\sup_{j\in\mathbb{Z}}\|P_j|D|^{\frac{3-n}2} (v_\lambda \bu_\mu^{x_0})  \|_{L^2_{t,x}} \lesssim \epsilon d_{\lambda} c_\mu \lambda^{-s+\frac12} (1+ \lambda |x_0|),
\qquad \mu \approx \lambda
\end{equation}
 \item Unbalanced bilinear $(u,v)$-$L^2$ bounds: Let $v_{\lambda}$ and $u^{x_0}_{\mu}$ be unbalanced in the sense of Definition~\ref{transversalpair}. There holds
\begin{equation} \label{uv-bi-unbal}
\| v_\lambda \bu_\mu^{x_0}  \|_{L^2_{t,x}} \lesssim \epsilon d_{\lambda} c_\mu \mu^{-s} \frac{\min\{\lambda,\mu\}^\frac{n-1}2}{(\lambda+\mu)^{\frac12}} 
\end{equation}

\item Balanced bilinear $(v,v)$-$L^2$ bound: There holds
\begin{equation} \label{vv-bi-bal}
\sup_{j\in\mathbb{Z}}\|P_j|D|^{\frac{3-n}2} (v_\lambda \bv_\mu^{x_0})  \|_{L^2_{t,x}} \lesssim  d_{\lambda} d_\mu \lambda^{\frac12} (1+ \lambda |x_0|),
\qquad \mu \approx \lambda
\end{equation}

 \item Unbalanced bilinear $(v,v)$-$L^2$ bound: Let $v_{\lambda}$ and $v_{\mu}^{x_0}$ be unbalanced. There holds
\begin{equation} \label{vv-bi-unbal}
\| v_\lambda \bv_\mu^{x_0}  \|_{L^2_{t,x}} \lesssim  d_{\lambda} d_\mu\frac{\min\{\lambda,\mu\}^\frac{n-1}2}{(\lambda+\mu)^{\frac12}}.
\end{equation}
\end{enumerate}
\end{theorem}
We note crucially that the unbalanced estimates above allow for semi-balanced pairs of functions where the two solutions can lie in the same frequency annulus (see Definition~\ref{transversalpair}). This generalizes the corresponding unbalanced bound (even in the definite case) from the prior article \cite{IT-qnls2}, where the proof required the solutions to be localized in different annuli. The new proof we present below is more robust and also utilizes a simpler family of interaction Morawetz functionals. 
\\

We also remark that the unbalanced bounds
are more robust than the balanced bounds above, in the sense that they do not require any relation between the metric in the $v_\lambda$ equation and the one in the $v_\mu$ equation:

\begin{corollary}\label{c:para}
The unbalanced bounds \eqref{uv-bi-unbal} and \eqref{vv-bi-unbal}
are still valid if the equation \eqref{para} for $v_\lambda$ 
is replaced by the linear Schr\"odinger equation   
\[
(i \partial_t+\Delta_{g(0)}) v_\lambda = f_\lambda.
\]
\end{corollary}

This is more  a corollary of the proof of the theorem, rather than the theorem itself. 
Indeed, the unbalanced case of the proof of the theorem applies in this case without any change.  The above corollary applies in particular if $v_\lambda$ is an $L^2$
solution to the homogeneous Schro\"dinger equation. That leads to a follow-up
corollary, as in \cite{IT-qnls2}:

\begin{corollary}\label{c:para1}
The unbalanced bounds \eqref{uv-bi-unbal} and \eqref{vv-bi-unbal}
are still valid if $v_\lambda \in U^2_{UH}$, with $d_\lambda = \|v_\lambda \|_{U^2_{UH}}$.
\end{corollary}

We move to the proof of our theorem.

\begin{proof}[Proof of Theorem~\ref{t:para}]
 In order to treat the
dyadic bilinear $(u,v)$ and $(v,v)$ bounds at the same time,
we write the equation for $u_\lambda$ in the form 
\eqref{para-full} where we estimate favourably the source term $N_\lambda$, and prove 
the following estimate:

\begin{proposition}\label{p:N-lambda}
 a) Let $n \geq 3$ and $s > \frac{n+1}{2}$. Assume that the function $u$ satisfies   the bounds \eqref{uk-ee}, \eqref{uab-bi-unbal}, \eqref{uab-bi-bal} and \eqref{uk-Str}
 in a time interval $[0,T]$. Then for $\epsilon$ small enough, the functions $N_\lambda(u)$ in \eqref{para-full} satisfy 
 \begin{equation}\label{good-nl}
\| N_\lambda(\bfu) \bu_\lambda^{x_0}\|_{L^1_{t,x}} \lesssim \epsilon^4 c_\lambda^2 \lambda^{-2s}.
 \end{equation}

 b) Let $n = 2$ and $s > \frac{n+1}{2}=\frac{3}{2}$. Then \eqref{good-nl}
 holds under the additional assumption $T \lesssim\epsilon^{-6}$.
 Furthermore, for any $T$ we have the partial bound
 \begin{equation}\label{good-fl}
\| F_\lambda(\bfu) \bu_\lambda^{x_0}\|_{L^1_{t,x}} \lesssim \epsilon^4 c_\lambda^2 \lambda^{-2s}.
 \end{equation}
\end{proposition}
\begin{proof}  The proof of \eqref{good-nl} (with the restriction on $T$ when $n=2$) is virtually identical to Proposition 7.2 in \cite{IT-qnls2}, except for some very minor adjustments. One simply applies the bounds \eqref{uk-ee}-\eqref{uk-l4-2d} in place of their analogues in \cite{IT-qnls2}. We remark that the $L^4_{t,x}$ bound now comes with a small dyadic loss when $n= 3$ (unlike in the definite case), but the strict inequality $s>\frac{n+1}{2}$ ensures that this loss is harmless. We leave the verification of the details to the interested reader. 
\medskip

When $n=2$, in the estimate \eqref{good-fl}, most terms are handled in an analogous way to the corresponding estimate in \cite{IT-qnls2}, except in the ultrahyperbolic setting, there is one additional term in $F_{\lambda}(u)\overline{u}_{\lambda}^{x_0}$ corresponding to so-called weakly transversal interactions (which is precisely why we need the new unbalanced bilinear bound in this section). Such terms can be expressed generically as 4-linear expressions of the form
\begin{equation*}
\lambda^2 L(\mathbf{u}_{\lambda},\mathbf{u}_{\lambda},\mathbf{u}_{\lambda},\mathbf{u}_{\lambda})
\end{equation*}
where two factors are balanced and two factors are unbalanced in the sense of Definition~\ref{transversalpair}. We can use the $L^4$ Strichartz bound \eqref{uk-Str2} (with a $\frac{1}{2}$ derivative loss) to estimate the balanced pair and the unbalanced $L^2$ bound \eqref{uab-bi-unbal} to estimate the remaining two factors. This yields
\begin{equation*}
\lambda^2 \|L(\mathbf{u}_{\lambda},\mathbf{u}_{\lambda},\mathbf{u}_{\lambda},\mathbf{u}_{\lambda})\|_{L_{t,x}^1}\lesssim \epsilon^4c_{\lambda}^4\lambda^{3-4s}\lesssim \epsilon^4c_{\lambda}^2\lambda^{-2s}
\end{equation*}
where we used the hypothesis $s>\frac{3}{2}$. 
\end{proof}
Given  the estimate \eqref{good-nl}, the $(u,v)$ bilinear bounds and the $(v,v)$ bilinear bounds are absolutely identical.
Hence in what follows we will just consider
the $(v,v)$ bounds. To prove the theorem it is convenient to make the following bootstrap assumptions: 

\begin{enumerate}
    \item Uniform energy bounds:
\begin{equation}\label{v-ee-boot}
\| v_\lambda \|_{L^\infty_t L^2_x} \leq C d_\lambda 
\end{equation}

 \item Unbalanced bilinear $(u,v)$-$L^2$ bound: Let $v_{\lambda}$ and $u_{\mu}^{x_0}$ be unbalanced. Then there holds
\begin{equation} \label{uv-bi-unbal-boot}
\| v_\lambda \bu_\mu^{x_0}  \|_{L^2_{t,x}} \leq C^2 \epsilon d_{\lambda} c_\mu \mu^{-s+\frac{n-1}2} \lambda^{-\frac12} 
 \qquad \mu \leq \lambda 
 \end{equation}

\item Balanced bilinear $(v,v)$-$L^2$ bound:
\begin{equation} \label{vv-bi-bal-boot}
\sup_{j\in\mathbb{Z}}\|P_j|D|^{\frac{3-n}2} (v_\lambda \bv_\mu^{x_0})  \|_{L^2_{t,x}} \leq   C^2 d_{\lambda} d_\mu  \lambda^{\frac12} (1+ \lambda |x_0|),
\qquad \mu \approx \lambda
\end{equation}

 \item Unbalanced bilinear $(v,v)$-$L^2$ bound: Let $v_{\lambda}$ and $v_{\mu}^{x_0}$ be unbalanced. Then there holds
\begin{equation} \label{vv-bi-unbal-boot}
\| v_\lambda \bv_\mu^{x_0}  \|_{L^2_{t,x}} \leq C^2 \epsilon d_{\lambda} d_\mu \mu^{\frac{n-1}2} \lambda^{-\frac12} 
 \qquad \mu \leq \lambda .
\end{equation}
\end{enumerate}
These are assumed to hold with a large universal constant $C$, and then will be proved to hold without it. The way $C$ will be handled is by always pairing it with $\epsilon$ factors in the estimates,
so that it can be eliminated simply by 
assuming that $\epsilon$ is sufficiently 
small.

Here we note that to prove \eqref{v-ee},
\eqref{uv-bi-bal} and \eqref{uv-bi-unbal} it suffices 
to work with a fixed $v_\lambda$, while for the bilinear

$v$ bounds we need to work with exactly two $v_\lambda$'s. Since the functions $v_\lambda$ are completely independent, without any restriction in generality we could assume that $d_\lambda = 1$ in what follows.

We first observe that the bound \eqref{v-ee} follows 
directly via a energy estimate. Next we use the bootstrap assumptions to estimate the sources $F^{para}$ in the density flux energy identities for $v_\lambda$ in $L^1$,
\begin{lemma}\label{l:Fmp}
Assume that \eqref{uv-bi-unbal-boot} holds. Then we have 
\begin{equation}\label{Fmp}
\| F^{para}_{\lambda,m}\|_{L^1_{t,x}} \lesssim  d_\lambda^2,
\qquad 
\|  F^{j,para}_{\lambda,p}\|_{L^1_{t,x}} \lesssim  \lambda  d_\lambda^2.
\end{equation}
\end{lemma}

\begin{proof} 
This is virtually identical to Lemma 7.3 in the previous work \cite{IT-qnls2}.
\end{proof}
\subsection{The balanced bilinear estimate}
We begin our analysis by establishing the balanced $(v,v)$ estimate \eqref{vv-bi-bal}. For this purpose, we take our cue from Section~\ref{s:im-cc} and define the one-parameter family of Morawetz functionals
\begin{equation}\label{Ia-sharp-def-vv}
\bI_r^{x_0}(v_\lambda,v_\mu) :=
 \iint a_{r,j}(x-y) (M(v_\lambda)(x) P^j(v^{x_0}_\mu) (y) -  
P^j(v_\lambda)(x) M(v^{x_0}_\mu) (y)) \, dx dy,\hspace{5mm}r>0
\end{equation}
where the weight $a_r$ is defined as in Section 5, by the formula
\begin{equation}\label{rweightdef}
a_{r}(x)=|x|+\frac{r}{2}e^{-r^{-1}|x|}
\end{equation}
Above and in the sequel we will use the convention
\begin{equation*}
a_{r,j}:=\partial_ja_r,\hspace{5mm}a_{r,jk}:=\partial_j\partial_ka_r.
\end{equation*}
The time derivative of $\bI_r^{x_0}(v_\lambda,v_\mu)$ is
\begin{equation}\label{I-vv}
\frac{d}{dt} \bI_r^{x_0}(v_\lambda,v_\mu) =  \bJ_r^4(v_\lambda,v_\mu^{x_0}) +  \bK_r(v_\lambda,v_\mu^{x_0}), 
\end{equation}
where
\begin{equation}\label{J4}
\begin{aligned}
\bJ_{r}^4(v_\lambda,v_\mu^{x_0}) := \iint & a_{r,jk}(x-y) 
\left(M(v_\lambda) E^{jk}(v_\mu^{x_0}) - P^j (v_\lambda)  P^k(v_\mu^{x_0}) \right.
\\ & \left. + 
M(v_\lambda) E^{jk}(v_\mu^{x_0})
-  P^j(v_\lambda)  P^k(v_\mu^{x_0})\right)\, dxdy,
\end{aligned}
\end{equation}
respectively
\begin{equation}\label{K8-def-ab-vv}
\begin{aligned}
\bK_{r} (v_\lambda,v_\mu^{x_0}):= \iint a_{r,j}(x-y) & \ M(v_\lambda)(x)   F^{j,para,x_0}_{\lambda,p}(y) + P_j(v_\lambda)(x)  F^{para,x_0}_{\lambda,m}(y) 
\\ & \!\!\!\!\!\! - 
M(v_\mu^{x_0})(y)  F^{j,para}_{\lambda,p}
(x)  - P_j(v_\mu^{x_0})(y)  F^{para}_{\lambda,m}(x)\,
dx dy.
\end{aligned}
\end{equation}

Below, we will use the (time-integrated) interaction Morawetz identity \eqref{I-vv} to prove the balanced bilinear bounds as follows:
\begin{itemize}
    \item The spacetime term $\bJ_{r}^4$ will control (at leading order), the quantity
    \begin{equation*}
\|P_{r^{-1}}|D|^{\frac{3-n}{2}}|v_{\lambda}|^2\|_{L_{t,x}^2},
    \end{equation*}
uniformly in $r>0$, which yields control of the relevant $L_{t,x}^2$ quantity in \eqref{vv-bi-bal}.
    \item The fixed time expression $\bI_{r}^{x_0}$ gives the primary upper bound for 
    $\bJ_r^4$ (after a time-integration).
    \item The space-time term $\bK_{r}$ will contribute a perturbative error. 
\end{itemize}
We summarize the straightforward estimates for both $\bI_r^{x_0}$ and $\bK_r$ below and then afterwards will turn our attention to $\bJ_r^4$.

\bigskip

\textbf{I. The bound for $\bI_{r}^{x_0}(v_\lambda,v_\mu)$.}
We obtain the (uniform in $r$) bound
\begin{equation}\label{I-bound}
 |    \bI_{r}^{x_0}(v_\lambda,v_\mu) | \lesssim (\lambda+\mu) \, d_\lambda^2 d_\mu^2,
\end{equation}
which is an immediate consequence of the uniform bound
\begin{equation}\label{unifarbound}
|Da_r|\lesssim 1,    
\end{equation}
and the fixed time density 
bounds
\begin{equation}\label{mpe-l1}
\| M(v_\lambda)\|_{L^1_{t,x}} \lesssim d_\lambda^2, \qquad 
\| P^j(v_\lambda)\|_{L^1_{t,x}}\lesssim \lambda d_\lambda^2, \qquad 
\| E^{jk}(v_\lambda)\|_{L^1_{t,x}}\lesssim \lambda^2 d_\lambda^2.
\end{equation}
These bounds simply follow from the uniform $L^2$ bound
from $v_\lambda$ and its frequency localization; the source terms $f_\lambda$
play no role.

\bigskip

\textbf{II. The bound for $\bK$.} We obtain
\begin{equation}\label{K-bound}
 |    \int_0^T \bK_r(v_\lambda,v_\mu)\, dt | \lesssim (\lambda+\mu) \,d_\lambda^2 d_\mu^2,
\end{equation}
which follows immediately from the bounds \eqref{Fmp} and \eqref{mpe-l1} and \eqref{unifarbound}.
\bigskip

\textbf{III. The contribution of $\bJ_{r}^4$.}
Now, we turn to the estimate for $\bJ_r^4$. We begin by observing the following simple algebraic computation 
\begin{equation}
\begin{aligned}    
\bJ_{r}^4(v_\lambda,v_\mu^{x_0}) = & \iint  a_{r,jk}(x-y) 
\left(|v_\lambda|^2  E^{jk}(v_\mu^{x_0}) 
+ |u_\mu^{x_0}|^2 E^{jk}(v_\lambda)
- 8 \Im(\bar v_\lambda  \partial^j v_\lambda) \Im(\bar v_\mu^{x_0}  \partial^j v_\mu^{x_0}) 
\right)\, dx dy,
\end{aligned}
\end{equation}
where
\[
\begin{aligned}
E^{jk}(v_\lambda) = & \   \partial^j v_\lambda 
\partial^k \bv_\lambda + \partial^k v_\lambda 
\partial^j \bv_\lambda -
v \partial^j   \partial^k   \bv -
\bv \partial^j   \partial^k   v 
\\
= & \   2\partial^j v_\lambda 
\partial^k \bv_\lambda + 2\partial^k v_\lambda 
\partial^j \bv_\lambda -
 \partial^j   (v \partial^k   \bv) -
 \partial^j   (\bv \partial^k   v) .
\end{aligned}
\]
Next, we integrate by parts the contributions of the last two terms in $E^{jk}$. As in \cite{IT-qnls2}, this generates error terms of two types, which arise from
\begin{itemize}
    \item the fact that $\partial^j$ is not skew-adjoint, and thus we have derivatives applied to the metric

    \item the fact that the metric $g$ in $\partial^j$ may be $g_{[<\lambda]}(x)$ or $g_{[<\mu]}^{x_0}$, depending on whether this operator is applied to $v_\lambda$ 
    or to $v_\mu^{x_0}$.
\end{itemize}
We therefore obtain the identity
\begin{equation}
\label{J4-expansion}
\begin{aligned}    
\bJ_r^4(v_\lambda,v_\mu^{x_0}) = & \ 4\iint  a_{r,jk}(x-y) 
\left(|v_\lambda|^2  \Re(\partial^j v_\mu^{x_0} 
\partial^k \bv_\mu^{x_0})
+ |u_\mu^{x_0}|^2  \Re(\partial^j v_\lambda 
\partial^k \bv_\lambda) \right.
\\
& \qquad \qquad \qquad \qquad \left.- 2 \Im(\bar v_\lambda  \partial^j v_\lambda) \Im(\bar v_\mu^{x_0}  \partial^j v_\mu^{x_0}) 
\right)\, dx dy
\\
& \ +  4\iint  a_{r,jk}(x-y) 
\left(\partial^j |v_\lambda|^2 
\partial^k |\bv_\mu^{x_0}|^2
+ \partial^k |u_\mu^{x_0}|^2  \partial^j |v_\lambda|^2 
\right)\, dx dy
\\
 & \ - \iint a_{r,jk}(x-y) |v_\lambda|^2 (g^{lj}_{[<\mu]} (y+x_0) - g^{lj}_{[<\lambda]}(x))
\partial_l \partial^j |v_\mu^{x_0}|^2 \, dx dy
\\
& \ + 2 \iint a_{r,jk}(x-y)(\partial_l g^{lj}_{[<\lambda]}(x)) \left(
|v_\lambda|^2 \partial^k |v_\mu^{x_0}|^2 
 + |u_\mu^{x_0}|^2  \partial^k |v_\lambda|^2  \right) \, dx dy
\\
:= & \ \bJ^4_{r,main} + \bR^4_{r,1} + \bR^4_{r,2}.
\end{aligned}
\end{equation}
Here the leading term $\bJ^4_{r,main}$ 
contains the contribution of the first two lines, and the last two lines represent 
the error terms $\bR^4_{b,1}$, respectively $\bR^4_{b,2}$. Exactly as in the constant coefficient case, $\bJ^4_{r,main}$ can be rewritten as
\begin{equation}\label{J4main}
\bJ^4_{r,main} = 2 \iint a_{r,jk}(x-y) F^j \bar F^k \, dxdy=2 \iint a_{r,jk}(x-y) G_j \bar G_k \, dxdy,
\end{equation}
where
\begin{equation}
G_j:=F^j = \partial^j v_\lambda \bar v_\mu^{x_0} + v_\lambda \partial^j \bar v_{\mu}^{x_0}    .
\end{equation}
That is, $G$ is defined by raising the indices of $F$ with respect to the metric $g$. Our next objective will be to extract the leading lower bound for $\bJ_{r}^4$ in the case when $\lambda=\mu$ and $x_0=0$.
\medskip

\emph{III A. The balanced symmetric case, 
$\lambda = \mu$, $x_0= 0$.} The key estimate in this setting is summarized by the following lemma.
\begin{lemma}\label{l:J4-bal}
Under our bootstrap assumptions on $u$ and the bilinear bootstrap assumptions \eqref{v-ee-boot}-\eqref{uv-bi-unbal-boot} on $v_\lambda$ we have the uniform in $r$ bound
\begin{equation}\label{eq:l2A}
\int_0^T \bJ_{r}^4(v_\lambda,v_\lambda)\, dt \gtrsim \| P_{r^{-1}}|D_x|^{\frac{3-n}2} |v_\lambda|^2\|_{L^2_{t,x}}^2+O(C^2\epsilon^2\lambda d_{\lambda}^4).
\end{equation}
\end{lemma}
Combined with \eqref{I-bound} and \eqref{K-bound}, this implies 
the bound \eqref{vv-bi-bal} (when $\lambda=\mu$ and $x_0=0$).

\begin{proof}
We begin by studying the leading term $\bJ^4_{r,main}$. By taking into account the uniform non-degeneracy of the metric $g_{<\lambda}$ (when $\lambda$ is large enough) as well as  Remark~\ref{uniformnondegremark}, the estimates from the case of the flat metric given by \eqref{balancedflatbound} and \eqref{lowerboundultra} apply verbatim and we obtain
\begin{equation}\label{J4-cov}
\bJ^4_{r,main}(v_\lambda,v_\lambda) \gtrsim \| P_{r^{-1}}|D_x|^{\frac{3-n}2} |v_\lambda|^2\|_{L^2_x}^2.
\end{equation}
\begin{remark}
We remark once again that the bound \eqref{J4-cov} appears to be quite robust in that it essentially only requires uniform non-degeneracy of the metric $g$. We note that this bound does not even require that $g$ be a small perturbation of a constant metric, which could be of interest in a large data regime. In the present setting, this observation can also be used to simplify the corresponding proof in \cite{IT-qnls2} (where the authors used the small data assumption). Of course, here, one pays the added price of a small dyadic loss.
\end{remark}
It remains to estimate the 
expressions $\bR^4_{r,1}$ and $\bR^4_{r,2}$. The analysis for such terms was done in the case of a positive-definite metric $g$ and with the weight $a(x)=|x|$ in \cite{IT-qnls2}. By observing the $r$-uniform bound on the Hessian of $a_r$, 
\begin{equation*}
|D^2a_r|\lesssim |x|^{-1}
\end{equation*}
and using Young's inequality and making use of the unbalanced bilinear $(u,v)$ bounds, we also obtain the bound
\begin{equation*}
\int_0^T \bR^4_{r,i}\, dt=O(C^2\epsilon^2\lambda d_{\lambda}^4),\hspace{5mm}i=1,2.
\end{equation*}
We omit the proof as it is entirely analogous to the analysis in \cite{IT-qnls2}. This concludes the proof of the lemma
\end{proof}

\bigskip

\emph{III B. The balanced non-symmetric case, 
$\mu \approx \lambda$, $x_0 \in \R$.} The non-symmetric bound follows exactly as in \cite{IT-qnls2} by applying the symmetric bound (established above) to the functions 
\[
w_\lambda = a v_\lambda + b v_\mu^{x_0}, \qquad |a|, |b| \leq 1
\]
for a suitable choice of complex numbers $a$ and $b$. We refer to \cite{IT-qnls2} for the details.

\bigskip

\subsection{ The unbalanced bilinear estimate}

Here we consider two scenarios:
\begin{itemize}
    \item the unbalanced case where we estimate the interaction of $v_\lambda$ and $v_\mu$, with $\mu \ll \lambda$.

    \item the semi-balanced case where we estimate the interaction of $v_\lambda$ and $v_\mu$, with $\mu \approx \lambda$, but with $O(\lambda)$ frequency separation. 
\end{itemize}

One common step in both cases is to further localize 
the two functions in frequency, so that 
\begin{itemize}
    \item $v_{\mu}$ is frequency localized in a ball $B_\mu$
    of radius $O(\delta \mu)$.
\item $v_{\lambda}$ is frequency localized in a ball $B_\lambda$
    of radius $O(\delta \lambda)$.
   \item The two balls $B_\lambda$ and $B_\mu$ have $O(\lambda)$ separation. 
\end{itemize}
Here $\delta \ll 1$ is a fixed small universal parameter.
This localization is achieved exactly as in \cite{2024arXiv240206278J}, using the following lemma which we recall for convenience:

\begin{lemma} 
Let $A(D)$ be a smooth, bounded multiplier at frequency $\lambda$.
If $v_\lambda$ satisfies \eqref{dl}, then so does $A(D) v_\lambda$.
\end{lemma} 

The reason we assume the two functions are frequency localized in separated balls is to achieve a similar separation of their corresponding group velocities.
Precisely, the group velocities of waves in associated 
to frequencies in $B_\mu$, respectively $B_\lambda$
are also localized in two balls $V_\mu$, respectively $V_\lambda$ with $O(\lambda)$ separation. Then we can choose 
a direction $e$ so that the two group velocity balls are 
$O(\lambda)$ separated in the $e$ direction. Without any restriction in generality suppose $e = e_1$.

Then we define our interaction Morawetz functional as
\begin{equation}\label{Ia-sharp-def-vv-tr}
\bI^{x_0}(v_\lambda,v_\mu) :=
 \iint_{x_1< y_1}  M(v_\lambda)(x) M(v^{x_0}_\mu) (y) \, dx dy.
\end{equation}
The time derivative of $\bI^{x_0}(v_\lambda,v_\mu)$ is
\begin{equation}\label{I-vv-tr}
\frac{d}{dt} \bI^{x_0}(v_\lambda,v_\mu) =  \bJ^4(v_\lambda,v_\mu^{x_0}) +  \bK(v_\lambda,v_\mu^{x_0}), 
\end{equation}
where
\begin{equation}\label{J4-tr}
\begin{aligned}
\bJ^4(v_\lambda,v_\mu^{x_0}) := \iint_{x_1 = y_1} & 
M(v_\lambda) P^1(v_\mu^{x_0}) - P^1 (v_\lambda)  M(v_\mu^{x_0}) \, dxdy,
\end{aligned}
\end{equation}
respectively
\begin{equation}\label{K8-def-ab-vv-tr}
\begin{aligned}
\bK (v_\lambda,v_\mu^{x_0}):= \iint_{x_1 < y_1} & \ M(v_\lambda)(x)   F^{para,x_0}_{\mu,m}(y) +
M(v_\mu^{x_0})(y)  F^{para}_{\lambda,m}
(x)\,
dx dy.
\end{aligned}
\end{equation}

The bound for $\bI^{x_0}(v_\lambda,v_\mu)$ is straightforward,
\begin{equation}\label{hl-I}
| \bI^{x_0}(v_\lambda,v_\mu)| \lesssim  \|v_\lambda\|_{L^2}^2 \|v_\mu\|_{L^2}^2 \lesssim d_\lambda^2  d_\mu^2  
\end{equation}
Using \eqref{Fmp} the bound for $\bK(v_\lambda,v_\mu^{x_0})$ is also good,
\begin{equation}\label{hl-K}
\left| \int_0^T  \bK(v_\lambda,v_\mu^{x_0}) dt \right|
\lesssim d_\lambda^2  d_\mu^2  
\end{equation}

It remains to consider $\bJ^4$. 
A first observation is that we can use the uniform bound $|g(u_{\lambda}) - g(0)| \lesssim C^2 \epsilon^2$ to  
compare $\bJ^4$ with the expression $\bJ^{4,0}$ obtained 
by replacing $g$ with $g(0)$.
\begin{equation}\label{hl-dJ4}
\left| \int_0^T  (\bJ^4 - \bJ^{4,0})(v_\lambda,v_\mu^{x_0}) dt \right|
\lesssim C^2 \epsilon^2 \lambda \sup_{x_0} \int_0^T \int_{x_1=y_1}  |u_\lambda(x)|^2 |v^{x_0}_\mu(y)|^2 dx dy
\end{equation}

Finally we examine  $\bJ^{4,0}$, which corresponds to the 
constant coefficient problem. Here it no longer suffices 
to show that $\bJ^{4,0}$ controls the $L^2$ norm of $v_\lambda v_\mu^{x_0}$, we also need to be able to control 
the error in the last estimate above. We will accomplish this in two steps:
\begin{equation}\label{good-J4}
\lambda \sup_{x_0} \int_0^T \int_{x_1=y_1}  |u_\lambda(x)|^2 |v^{x_0}_\mu(y)|^2 dx dy
\lesssim \sup_{x_0}
\int_0^T   \bJ^{4,0}(v_\lambda,v_\mu^{x_0}) dt 
\end{equation}
respectively 
\begin{equation}\label{J4-Bernstein}
\|   v_\lambda v_\mu^{x_0} \|_{L^2}^2 \lesssim 
\mu^{n-1}  \int_0^T \int_{x_1=y_1}  |u_\lambda(x)|^2 |v^{x_0}_\mu(y)|^2 dx dy
\end{equation}
Once we have these two estimates, combining them with 
\eqref{hl-I}, \eqref{hl-K} and \eqref{hl-dJ4} yields the desired bilinear bound.

It remains to prove \eqref{good-J4} and \eqref{J4-Bernstein}. We begin with the proof of \eqref{good-J4}.
Here we begin by observing that $J^{4,0}$ is a quadrilinear
form with symbol 
\[
j^{4,0}(\xi^1,\xi^2,\xi^3,\xi^4) = g^{1j}(0)
(\xi^1_j+\xi^2_j -\xi^3_j - \xi^4_j)
\]
where each entry represents the $x_1$ component of the group velocity associated to the corresponding frequency.
Denoting by $\xi^\mu$ and $\xi^\lambda$ the centers of the 
frequency balls $B_\lambda$ and $B_\mu$, within the supports we have
\[
|\xi^1 - \xi^\lambda| + |\xi^2 - \xi^\lambda| \lesssim \delta
\lambda,
\qquad 
|\xi^3 - \xi^\mu| + |\xi^4 - \xi^\mu| \lesssim \delta
\mu.
\]
Denoting the symbol at the ball centers by 
\[
j^{4,0}_{center} =  g^{1j}(0)
(\xi^\lambda_j-\xi^\mu_j), 
\]
we can thus estimate the contribution of the difference by
\[
\left|\int_0^T   (\bJ^{4,0}- \bJ^{4,0}_{center})(v_\lambda,v_\mu^{x_0}) dt \right| \lesssim 
\delta \lambda \sup_{x_0} \int_0^T \int_{x_1=y_1}  |v_\lambda(x)|^2 |v^{x_0}_\mu(y)|^2 dx dy
\]
which is an acceptable error if $\delta \ll 1$. 
On the other hand the center contribution is positive, of size $\approx \lambda$, due to the ball group velocity 
separation and our choice of the $e_1$ direction. Hence we have 
\[
\left|\int_0^T   \bJ^{4,0}_{center}(v_\lambda,v_\mu^{x_0}) dt\right| 
\approx \lambda\int_0^T \int_{x_1=y_1}  |v_\lambda(x)|^2 |v^{x_0}_\mu(y)|^2 dx dy.
\]
Combining the last two bounds we obtain \eqref{good-J4}.

For \eqref{J4-Bernstein}. we apply Bernstein's inequality in $(x_2,...,x_n)$ to obtain
\begin{equation}\label{J4-Bernstein+}
\int_{0}^{T}\int |v_\lambda(x)|^2 |v_\mu^{x_0}(x)|^2 \, dxdt \lesssim 
\mu^{n-1}  \int_0^T \int  |v_\lambda(x)|^2 |v^{x_0}_\mu(x_1,y')|^2 \, dx dy'dt
\end{equation}
which yields the required bound.

\end{proof}


\section{Low regularity Strichartz bounds for the paradifferential flow}
\label{s:Str}

Broadly speaking, the nonlinear interactions are controlled in this paper via bilinear estimates, which we divide into balanced and unbalanced. The bilinear $L^2$ bounds in the previous section
suffice in order to estimate unbalanced interactions. However,
the balanced bilinear $L^2$ bounds exhibit a loss at low frequency output in low dimension. For this reason, we will also need to rely on Strichartz estimates, of which the $L^4$ bound is the critical one in the case of cubic nonlinearities. Our approach for the $L^4$ bounds depends on the dimension, where we distinguish three cases:

\begin{enumerate}
    \item high dimension, $n \geq 4$. There, on one hand, $L^4$  is not a sharp Strichartz norm, and on the other hand the $L^4$ bound is obtained directly from the interaction Morawetz analysis, so this section is not playing any role in obtaining it.
    
    \item intermediate dimension, $n = 3$. This is a borderline case, where we nearly obtain the $L^4$ bound, except for a dyadic logarithmic summation loss. To rectify this, we will
    need to interpolate with the Strichartz estimates in this section. These have a substantial loss, but after interpolation 
    we obtain an $L^4$ bound with an arbitrarily small loss.

 \item low dimension, $n = 2$. Here $L^4$  is a sharp Strichartz norm, which cannot be obtained directly from the interaction Morawetz analysis. This is the most essential reason for the analysis in this section.
\end{enumerate}

In the case $n= 2$ we split further the analysis in two cases, corresponding to local or global in time estimates: 

\begin{enumerate}[label=(3\roman*)]
    \item $n=2$, local in time bounds (needed for local well-posedness). Here interpolating the interaction Morawetz bilinear bound with the $L^\infty_t L^2_x$ energy estimate yields a lossless $L^6_t L^4_x$ bound which suffices for the local result.

    \item $n=2$, global in time bounds (generically needed for global well-posedness). This is the one case where the analysis here is of the essence, providing an $L^4_{t,x}$ bound with a $1/2$ derivative loss.
\end{enumerate}

To a large extent the above discussion is similar with the prior work of the first and the last authors \cite{IT-qnls2}, with the notable differences generated by the logarithmic loss in the 
balanced bilinear $L^2$ bounds.  Again no defocusing assumption is necessary, unlike in the 1D case studied in  \cite{IT-qnls}. The conservative assumption only plays a role in 2D, in reducing the Sobolev regularity at which the global result holds;
this would have otherwise required an $L^4_{t,x}$ bound with a smaller $1/4$ derivative loss.

The result in this section provides a full range of Strichartz estimates in all dimensions, globally in time, but with a loss of derivatives: precisely one full derivative at the Pecher endpoint.
For the proof of our global results we need it only in two and three space dimensions, but the result is dimension independent.

For context, we note that Strichartz estimates for linear Schr\"odinger evolutions with variable coefficients have been the focus of research for over two decades. Initially, local in time full Strichartz estimates for compactly supported perturbations of the flat metric in $\R^n$ were considered in \cite{ST}, followed by \cite{RZ, TW, BT} for asymptotically flat metrics. Subsequent to these pioneering works, global in time results were established in \cite{T:Str} and \cite{MMT-lin}, also within the framework of asymptotically flat metrics. In the alternative setting of compact manifolds, \cite{BGT} proved local in time Strichartz estimates with derivative losses, obtained by adding up sharp Strichartz estimates on shorter, semiclassical time scales. For the sake of comparison, it should be noted that all of the metrics considered in these works possess at least $C^2$ regularity, while the metrics in the present article are just above $C^{\frac{1}{2}}$. In contrast, the estimates proven in this section exhibit losses of derivatives but are global in time, precluding their derivation from lossless bounds in shorter time scales.

\medskip
To review the set-up here, we seek $L^p_t L^q_x$ Strichartz bounds for the paradifferential problem
\begin{equation}\label{para-re}
i \partial_t v_\lambda + \partial_j g^{jk}_{[<\lambda]} \partial_k v_\lambda 
= f_\lambda, \qquad v_\lambda(0) = v_{0,\lambda},
\end{equation}
where the pair of indices $(p,q)$ lies on the sharp Strichartz line
\begin{equation}
\frac{2}{p} + \frac{n}q = \frac{n}2, \qquad p \geq 2.    
\end{equation}
We will refer to the case $p=2$ as the Pecher endpoint, which is forbidden in dimension $n = 2$.

From a classical perspective it is natural to place the source term $f_\lambda$ in dual Strichartz spaces.
While we certainly want to allow such sources, in our analysis 
we would like to allow for a larger class of functions which is
consistent with our needs for the quasilinear problem. One such class will be the space $DV^2_{\UH}$ associated to the flat ultrahyperbolic Schr\"odinger flow
\begin{equation} \label{flat-uh}
(i \partial_t + \Delta_{g(0)}) u_\lambda = f_\lambda, \qquad u_\lambda(0) = u_{\lambda 0},
\end{equation}
where the operator $\Delta_{g(0)}$ associated to the constant metric $g(0)$ is nondegenerate but not elliptic.

Following \cite{IT-qnls2}, a second class $N^\sharp$ is defined  by testing against solutions to the adjoint constant coefficient problem 
\begin{equation} \label{adjoint}
(i \partial_t + \Delta_{g(0)}) w_\lambda = h_\lambda, \qquad w(T) = w_T,
\end{equation}
defining a norm as follows:
\begin{equation}
\| f_\lambda \|_{\Ns} := \sup_{\| h_\lambda \|_{S'} + \|w_T\|_{L^2_{x}} \leq 1} \left|\int f_\lambda w_\lambda \,dx dt\right| .
\end{equation}
For technical reasons we will use both $DV^2_{\UH}$ and $N^\sharp$;
the first choice turns out to be better in the low dimension $n = 2$, and the second is better in the higher dimension $n \geq 3$. Their common property is provided by the embeddings 
\begin{equation}\label{N-embed}
 N \subset \Ns, \qquad N \subset DV_{\UH}^2,
\end{equation}
which hold regardless of the dimension, where $N$ represents any of the admissible sharp dual Strichartz norms.

We will also need to measure 
$f_\lambda$ relative to $v_\lambda$, via the quantity 
\[
I(v_\lambda, f_\lambda) = \sup_{x_0} \| v_\lambda f_\lambda^{x_0}\|_{L^1_{t,x}} .
\]
One consequence of allowing such terms is that we can commute
the $x$ derivatives as we like in the equation,
placing errors in the perturbative 
source term $f_\lambda$.

\begin{theorem}\label{t:para-se}
Assume that $u$ solves \eqref{qnls} in a time interval $[0,T]$, and satisfies the bootstrap bounds \eqref{uk-ee-boot},\eqref{uab-bi-unbal-boot},  \eqref{uab-bi-bal-boot}, \eqref{uk-Str2-boot} and \eqref{uk-Str3-boot}  for some $s > \frac{n+1}2$. Then the following Strichartz estimates hold for the linear paradifferential equation \eqref{para-lin} for any pair of sharp Strichartz exponents $(p,q)$ and $(p_1,q_1)$:
\begin{equation}\label{Str-2}
\|v_\lambda\|_{L^\infty_t L^2_x} + \lambda^{-1} \| v_\lambda \|_{V^2_{UH}} \lesssim  \| v_{\lambda,0}\|_{L^2_x} 
 + \lambda^{-1} \|f_\lambda\|_{DV^2_{UH}}
 + I(v_{\lambda},f_\lambda)^\frac12, \qquad n = 2,
\end{equation}
\begin{equation}\label{Str-3}
\|v_\lambda\|_{L^\infty_t L^2_x} + \lambda^{-1} \| v_\lambda \|_{L^2 L^\frac{2n}{n-2}} \lesssim  \| v_{\lambda,0}\|_{L^2_{x}} + 
\lambda^{-1} \|f_\lambda\|_{\Ns}
 + I(v_{\lambda},f_\lambda)^\frac12, \qquad n \geq 3.
\end{equation}
In particular we have
\begin{equation}\label{Str-all}
\lambda^{-\frac{2}p} \| v_\lambda \|_{L^p_t L^q_x} \lesssim  \| v_{\lambda,0}\|_{L^2} 
 + \lambda^{\frac{2}{p_1}} \|f_\lambda\|_{L^{p'_1}_tL^{q'_1}_x}.
\end{equation}
\end{theorem}

The proof of the theorem is identical to the elliptic case, as  only the unbalanced bilinear bounds are used. Hence we omit it and instead refer the reader to \cite{IT-qnls2}.

\section{The local/global well-posedness result}
\label{s:rough}

In this section our main goals are as follows: First, we aim to establish our primary local well-posedness result in Theorem~\ref{t:local3+}, as well as the corresponding global well-posedness result in Theorem~\ref{t:global3} for dimensions $n\geq 3$. Moreover, with minor adjustments to the proof, we will also obtain the unconditional global well-posedness result in Theorem~\ref{t:global2} for the two-dimensional problem (which corresponds to Sobolev regularities $s\geq 2$). For simplicity of exposition, we will restrict our attention to the problem \eqref{qnls}, as the proof for \eqref{dqnls} is virtually identical.

The starting point in our analysis is to use the high regularity local well-posedness result in Theorem~\ref{t:regular1}, which allows us to construct (from smooth data), a regular solution to \eqref{qnls}.  For these solutions, our goal will be to establish the frequency envelope bounds in Theorems~\ref{t:local-fe3}, \ref{t:local-fe2}[(i)-(ii)], which will allow us to continue the regular solutions as follows:

\begin{enumerate}
    \item globally in time if $n \geq 3$,
    \item globally in time if $n=2$ for restricted $s$  i.e.,  $s \geq 2$.
    \item up to time $\epsilon^{-6}$ if $n=2$ at low regularity $s > \frac32$.
\end{enumerate}

We will then conclude the proofs of Theorems~\ref{t:local3+}, \ref{t:global3}, and \ref{t:global2}
by obtaining rough solutions as unique limits of these regular solutions. Given the bounds established in the previous section, much of the proof proceeds almost exactly as in \cite{IT-qnls2}. Therefore, we only outline the important changes and refer to the mentioned article for the more detailed calculations.

\bigskip

\subsection{ A-priori bounds for smooth solutions: Proof of Theorems~\ref{t:local-fe3},~\ref{t:local-fe2}[(i)-(ii)].}
Our starting point is to first impose the bootstrap assumption that 
the bounds \eqref{uk-ee-boot} and \eqref{uab-bi-unbal-boot} hold for some large universal constant $C$. To this, we add the estimate \eqref{uk-Str2-boot} in dimension two, respectively \eqref{uk-Str3-boot} in dimension three. We now consider the estimates in the theorems:

\medskip
\emph{ A. The energy bound \eqref{uk-ee} and the bilinear estimates \eqref{uab-bi-unbal}, \eqref{uab-bi-bal}.} These are obtained by applying Theorem~\ref{t:para} with $v_\lambda=u_\lambda$. Toward this end, one can simply write the equation for $u_\lambda$ in the paradifferential form \eqref{para-full}, and use Proposition~\ref{p:N-lambda} for controlling the source term. This suffices for Theorem~\ref{t:local-fe3} and Theorem~\ref{t:local-fe2}[(ii)]. In the case of Theorem~\ref{t:local-fe2}(i), Proposition~\ref{p:N-lambda} provides the desired bound for the $F_\lambda$ part of $N_\lambda$, but it remains to estimate directly the contribution of the balanced term $C_\lambda$. This uses the $L^4$ bootstrap bound \eqref{uk-Str2-boot} as follows:
\[
\| C_\lambda(\bfu) \cdot \bu^{x_0}\|_{L^1} \lesssim 
\lambda^2 \|u_\lambda\|_{L^4}^4 \lesssim 
C^4 \epsilon^4 c_\lambda^4 \lambda^{-4s+4} \lesssim \epsilon^2 c_\lambda^2 \lambda^{-2s},
\]
where the last step requires $s \geq 2$.

\medskip

\emph{ B. The Strichartz estimates \eqref{uk-Str}.} These are obtained by applying Theorem~\ref{t:para-se} to $u_\lambda$. For this we need to estimate the contribution of the source terms on the right, with $f_\lambda = N_\lambda(u)$. The bound for the expression $I(v_\lambda,N_\lambda)$  
comes from Proposition~\ref{p:N-lambda}, as discussed above.
It remains to bound the middle term on the right in \eqref{Str-2}, respectively \eqref{Str-3}. 

\medskip

In dimension $n\geq 3$ this requires 
the estimate
\[
\| N_\lambda(u) \bar w_\lambda\|_{L^1} \lesssim \epsilon c_\lambda \lambda^{-s} ( \|w_{\lambda,T}\|_{L^2}+ \|h_\lambda\|_{S'})
\]
for $w_\lambda$ solving \eqref{adjoint}. This is the same as Proposition~\ref{p:N-lambda} but with $u_\lambda^{x_0}$ replaced by 
$w_\lambda$. This repeats the proof of Proposition~\ref{p:N-lambda}, but with one instance of \eqref{uab-bi-unbal-boot} replaced by Corollary~\ref{c:para}.

\medskip

In dimension $n = 2$, by duality  this requires 
the estimate
\[
\| N_\lambda(u) \bar w_\lambda\|_{L^1} \lesssim \epsilon c_\lambda \lambda^{-s}  \|w_{\lambda}\|_{U^2_{UH}}.
\]
This is again the same as Proposition~\ref{p:N-lambda} with $u_\lambda^{x_0}$ replaced by 
$w_\lambda$, which requires the corresponding counterpart of Corollary~\ref{c:para} for such a $w_\lambda$. It suffices to verify this for a $U^2_{UH}$ atom, which is an $\ell^2$ concatenation of solutions to the homogeneous Schr\"odinger equation. But then 
the problem reduces to the case of solutions to the homogeneous Schr\"odinger equation,
which is directly included in  Corollary~\ref{c:para}.

This concludes the proof of the above Theorems. For later use, we also observe that in particular, the estimates for the paradifferential source terms in Proposition~\ref{p:N-lambda} are also valid for these solutions. 

\subsection{Higher regularity}
Here we consider regular solutions, which are generated from regular initial data $u_0\in H^{\sigma}\subset H^s$, but with smallness only in the rough topology $H^s$. Precisely, we assume for some $\sigma>s$,
\[
\|u_0\|_{H^s} \leq \epsilon, \qquad \| u_0\|_{H^{\sigma}} \leq M,
\]
where $M$ can potentially be large relative to $\epsilon$. We then claim that $u$ retains its quantitative regularity in the stronger topology $H^{\sigma}$ (on the relevant time-scales for which $u$ remains small in the rougher topology). That is, we aim to show,
\begin{equation}
\| u\|_{L^\infty_t H^{\sigma}_x} \lesssim M.    
\end{equation}

To establish this,  for initial data $u_0$  as above, we let $\epsilon c_\lambda$ be a minimal $H^s$ frequency envelope for $u_0$ which satisfies the unbalanced slowly varying condition as in Remark~\ref{r:unbal-fe}. Then by definition, we have
\[
\sum_\lambda (\epsilon c_\lambda \lambda^{\sigma-s})^2 \lesssim M^2.
\]
 Theorems~\ref{t:local-fe3}, \ref{t:local-fe2} ensure that the frequency envelope bound is propagated, and we obtain (on the relevant timescale)
\[
\| u(t) \|_{H^{\sigma}}^2 \lesssim \sum_\lambda (\epsilon c_\lambda \lambda^{\sigma-s})^2 \lesssim M^2, \qquad t \in [0,T].
\]
In addition to the $L^2$ based bound above, the corresponding bilinear $L^2$ and Strichartz bounds are also propagated.

\subsection{ Continuation of regular solutions}  
Fix any number $s>\frac{n}{2}+\frac{1}{2}$ and let $\sigma\geq s+1$. Next, we show that the regular solution $u$ generated from data $u_0 \in H^{\sigma}$ can be continued as long as $u$ remains sufficiently small in the rougher topology $L_t^{\infty}H_x^s$. More precisely, we have the following proposition.
\begin{proposition} \label{p:continuation}
For every initial data  $u_0 \in H^{\sigma}$ satisfying the smallness condition
\begin{equation}\label{small-data}
\| u_0 \|_{H^s} \leq \epsilon \ll 1,   
\end{equation}
there exists a local solution $u \in C_t H^{\sigma}_x$. Moreover, this local solution can be continued for as long as 
\[
\| u \|_{L^\infty_t H^s_x} \lesssim \epsilon.
\]
\end{proposition}

Combining this with  Theorems~\ref{t:local-fe3},~\ref{t:local-fe2}[(i)(ii)], we have the following estimates on the lifespan of regular solutions:

\begin{corollary}\label{c:continuation}
Given initial data as in Proposition~\ref{p:continuation}, 
there exists a regular solution $u\in C_tH_x^{\sigma}$ which persists 

(i) globally in time in dimension $n \geq 3$,

(ii) up to time $O(\epsilon^{-6})$ in dimension $n =2$.    
\end{corollary}

\begin{proof}[Proof of Proposition~\ref{p:continuation}]
To construct local solutions we would like to apply Theorem~\ref{t:regular1}. However, as in the definite case \cite{IT-qnls2}, we run into the following technical issue: even if we have small data in $H^s$ as in \eqref{small-data}, the $H^\sigma$ norm is potentially large. Consequently, we will need to use the large data version 
of the local well-posedness from Theorem 1.3 and Remark 1.5 in \cite{pineau2024low}. To apply this result, we will need to verify the  nontrapping assumption for the initial data metric. For this we have the following lemma.
\begin{lemma}\label{l:nontrap}
Let $0<\epsilon \ll 1$. Then the initial data $u_0$ satisfying the smallness condition \eqref{small-data} are uniformly nontrapping.
\end{lemma}

Here, uniformly nontrapping means that any bicharacteristic with initially (approximately) unit speed for the metric $g$ intersects the ball $B(x_0,r)$ for time at most $\lesssim r$ (for large $r$),
with a uniform constant. 
\begin{remark}
One might wonder why the unit speed condition is only imposed at the initial time for the bicharacteristic. This is because unlike in the elliptic case, it does not make sense to normalize the speed by $|\xi|_g$ (which is a conserved quantity of the Hamilton flow). Nevertheless, under suitable regularity and decay conditions on a non-degenerate metric (which will be satisfied in our setting), one can show that the resulting bicharacteristics exist globally and approximately retain their unit speed. See \cite{pineau2024low} for more discussion on this matter.
\end{remark}

Assuming that the above lemma holds, the large data local well-posedness theorem in \cite{pineau2024low} yields a local $H^\sigma$ solution. We then let $[0,T)$ denote the maximal interval on which the $H^\sigma$ solution exists.
The apriori bounds above ensure that the $H^\sigma$ size of the solution remains bounded in terms of the $H^{\sigma}$ data, 
\[
\|u(t) \|_{H^\sigma} \lesssim \|u_0 \|_{H^\sigma}, \qquad t \in [0,T),
\]
and moreover, that we have a  corresponding $H^\sigma$ frequency envelope bound, which holds uniformly in time. This information does not immediately ensure that the limit 
\[
u(T) = \lim_{t \to T} u(T)
\]
exists in $H^\sigma$.
However, by directly examining the equation \eqref{qnls}, one easily establishes convergence in $H^{\sigma -2}$ (say). Combining this with the frequency envelope bound shows that the limit exists in $H^s$.
This in turn contradicts the maximality of $T$, which concludes the proof of the proposition.
\medskip

It remains to prove the  nontrapping lemma:
\begin{proof}[Proof of Lemma~\ref{l:nontrap}]
The strategy here will be similar to that of the definite case \cite{IT-qnls2}, although with some small adjustments. For convenience, we outline most of the details below. Out starting point is the equations for the Hamilton flow, which have the form
\[
\dot x^j  = 2 g^{jk} \xi_k, \qquad \dot \xi_j = \partial_{x_j} g^{kl}
\xi_k \xi_l,
\]
where $g = g(u_0)$. We have to prove that for $\xi(0)\neq 0$, that $t\to x(t)$ escapes to $\infty$ as $t\to\pm\infty$. From the scaling symmetry of the Hamilton flow, we may assume without loss of generality that $|\xi(0)|=1$. From the equation for $x(t)$, this ensures (by non-degeneracy of $g$) that $x(t)$ is initially approximately unit speed. As in \cite{IT-qnls2}, our strategy will be to show that the curve parameterized by $t\mapsto x(t)$ is close to a straight line. For this purpose, it suffices to show 
\begin{equation}\label{H-flow}
|\xi(t) - \xi(0)| \lesssim \epsilon, \qquad |\dot x(t) - \dot x(0)|
\lesssim \epsilon.
\end{equation}
Given the assumption on the data, there exists a constant non-degenerate matrix $g_{\infty}$ such that $\|g-g_\infty\|_{H^s}\lesssim \epsilon$ (and thus $\|g-g_{\infty}\|_{L^{\infty}}\lesssim\epsilon$). Therefore, the first bound in \eqref{H-flow} implies the second. To establish the first bound, we make a bootstrap assumption
\begin{equation}\label{H-flow-boot}
 |\xi(t) - \xi(0)|
\leq C\epsilon,
\end{equation}
for some large universal constant $C$. Hence, we also have 
\[
 |\dot x(t) - \dot x(0)| \lesssim C\epsilon \ll 1,
\]
which in turn indicates that the   bicharacteristic curve $\gamma$ paramaterized by $t\mapsto x(t)$, is approximately straight and that $x(t)$ is approximately unit speed. Since $g-g_\infty$ is at least quadratic in $u$ (by assumption), the equation for $\xi$ and the bootstrap assumption on $\xi$ yield
\[
|\xi(t) - \xi(0)| \lesssim \int_{0}^t |\nabla g(x(s))|\, ds
\lesssim \int_{0}^t |u_0(x(s))| |\nabla u_0(x(s))|\, ds.
\]
Since $\gamma$ is approximately straight and $x$ is approximately unit speed, and $s > \frac{n+1}{2}$, the trace theorem and Cauchy-Schwarz yields
\[
|\xi(t) - \xi(0)|\lesssim \| u\|_{L^2(\gamma)}\|\nabla u\|_{L^2(\gamma)} \lesssim \| u_0\|_{H^s}^2 \lesssim \epsilon^2.
\]
which improves the constant in the bootstrap if $\epsilon$ is small enough. This completes the proof
\end{proof}

This in turn completes the proof of the continuation proposition.
\end{proof}

\subsection{$L^2$ bounds for the linearized equation and Proof of Theorem~\ref{t:linearize-fe}.}

Similarly to above, we begin by applying Theorem~\ref{t:para} to the linearized equation by writing the it first in the paradifferential form \eqref{para-lin}. For the linearized evolution, we take $d_\lambda$ to be an $L^2$ frequency envelope for the initial data.

As in the definite case in \cite{IT-qnls2}, it suffices to make a bootstrap assumption for $v_\lambda$, as in the proof of Theorem~\ref{t:para},
and to establish the following estimate for the linearized source term $N_\lambda^{lin} v$. This serves as the linearized counterpart of Proposition~\ref{p:N-lambda}:

\begin{proposition}
 Let $s > \frac{n+1}2$. Assume that the function $u$ satisfies  the bounds \eqref{uk-ee},\eqref{uab-bi-unbal} and \eqref{uab-bi-bal} in a time interval $[0,T]$, where $T \lesssim \epsilon^{-6}$
 in dimension $n=2$. Assume also that $v$ satisfies the bootstrap bounds \eqref{v-ee-boot}-\eqref{vv-bi-unbal-boot}. Then for $0<\epsilon\ll 1$, the source term $N^{lin}_\lambda v$ satisfies the frequency envelope bound
 \begin{equation}\label{good-nl-lin}
\| N^{lin}_\lambda v  \cdot \bv_\lambda^{x_0}\|_{L^1_{t,x}} \lesssim \epsilon^2 d_\lambda^2 .
 \end{equation}
 One obtains also a similar bound if we replace $\bv_\lambda^{x_0}$ by $\bar w_\lambda$, where $w_\lambda$ is a solution for the constant coefficient Schr\"odinger equation \eqref{adjoint}.
 
\end{proposition}

\begin{proof}
This proof is almost identical to the proof of Proposition~\ref{p:N-lambda} (which is similar to the corresponding proof in the definite case in \cite{IT-qnls2}). We omit the details.

\end{proof}

\bigskip

\subsection{Rough solutions as limits of regular solutions}

Here we combine the $H^s$ energy estimates for the nonlinear equation with the $L^2$ estimates for the linearized flow to obtain the local well-posedness results (including existence, uniqueness, continuous dependence) in  
Theorem~\ref{t:local3+}, and also the global results in Theorems \ref{t:global3} and \ref{t:global2}. Given these estimates, an entirely standard argument applies here. In fact, the proof is  virtually identical to the corresponding proof in the definite case. Therefore, we refer to the corresponding section of \cite{IT-qnls2} which outlines the main steps. See also the argument in the prior 1D paper \cite{IT-qnls}, and also the expository paper~\cite{IT-primer} for a detailed overview of this type of scheme; a more abstract approach  is provided in \cite{ABITZ}.
\subsection{Scattering }
In this subsection, our goal is to establish the scattering part Theorem ~\ref{t:global3} and Theorem ~\ref{t:global2}. The analysis will differ slightly depending on if we consider the case $n\geq 3$ or $n=2$. We will focus on the higher dimensional case first, and then outline the minor differences in two dimensions.

Our starting point is to write the equation for each $u_\lambda$
as as constant coefficient equation, directly the metric as a perturbation of the flat metric $g(0)$,
\begin{equation}\label{constcoeffscattering}
(i\partial_t + \Delta_{g(0)}) u_\lambda = f_{\lambda}:= N_\lambda(u) - \partial_j(g^{jk}_{[<\lambda]}- g^{jk}(0)) \partial_k u_\lambda.
\end{equation}
By following exactly the reasoning in the corresponding section of \cite{IT-qnls2}, we obtain the following bound on the source term $f_{\lambda}$ with a one derivative loss,
\begin{equation}\label{Flambda}
\| f_{\lambda} \|_{\Ns} \lesssim \epsilon^3 c_\lambda^3 \lambda^{-s+1}.
\end{equation}
This bound holds globally in time, but arguing as in \cite{IT-qnls2}, if we can restrict the time interval then we also obtain the decay
\begin{equation}\label{fdecay}
\lim_{T \to \infty}    \| f_{\lambda} \|_{\Ns[T,\infty)} = 0.
\end{equation}
The bounds \eqref{Flambda}-\eqref{fdecay} give the estimate
\begin{equation}
\| u_\lambda (t) - e^{i(t-\tau) \Delta_{g(0)}} u_\lambda(\tau)\|_{L^2_x} \lesssim   \epsilon^3 c_\lambda^3 \lambda^{-s+1},
\end{equation}
and also the continuity property
\begin{equation}
   \lim_{t,\tau \to \infty}  \| u_\lambda(t) - e^{i(t-\tau) \Delta_{g(0)}} u_{\lambda}(\tau)\|_{L^2_x} = 0.
\end{equation}
It follows that the limit
\[
u_{\lambda}^\infty = \lim_{t \to \infty} e^{-it\Delta_{g(0)}} u_\lambda(t) 
\]
exists in $L^2$, and moreover, that we have the bound 
\begin{equation}
  \|  u_{\lambda}^\infty - u_{0\lambda}\|_{L^2}  \lesssim   \epsilon^3 c_\lambda^3 \lambda^{-s+1}.
\end{equation}
In addition, by passing to the limit in the $L_x^2$ energy bounds for $u_\lambda$ we obtain
\begin{equation}
   \|  u_{\lambda}^\infty\|_{L^2_x}  \lesssim   \epsilon c_\lambda \lambda^{-s}  .
\end{equation}
Therefore, if we define
\[
u^\infty = \sum_\lambda u_\lambda^\infty \in H^s,
\]
then by combining the above properties, we obtain the convergence
\begin{equation}
\lim_{t \to \infty}     e^{-it\Delta_{g(0)}} u(t) = u^\infty \qquad \text{ in } H^s, 
\end{equation}
in the strong topology. This handles the proof of scattering in three and higher dimensions.

The scattering result in two dimensions follows mostly the same line of reasoning, but with the main change being that we use the space $DV^2_{UH}$ in place of $\Ns$, as in the Strichartz estimates in Section~\ref{s:Str}. We also remark that the analysis for obtaining the requisite bound $\eqref{Flambda}$ is a bit more involved and makes use of the $L_{t,x}^4$ Strichartz bound with a $\frac{1}{2}$ derivative loss. This is the reason for the restriction $s\geq2$ for the result in Theorem ~\ref{t:global2}. For more details, see the corresponding discussion in the definite case in \cite{IT-qnls2}, which is entirely analogous at this point (after replacing $DV_{\Delta}^2$ with $DV_{UH}^2$ in the analysis).

\section{
The 2D global result at low regularity}
Our main goal in this last section will be to establish the 2D global well-posedness result in Theorem~\ref{t:global2c}. Thanks to the local result in Theorem~\ref{t:local3+} and the corresponding frequency envelope  estimates in Theorem~\ref{t:local-fe2}, it remains only to establish that the frequency envelope bounds in Theorem~\ref{t:local-fe2} hold globally. Moreover, it is sufficient to do this under suitable bootstrap assumptions as in Proposition~\ref{p:boot}.

For the local result in Theorem~\ref{t:local3+}, (and global result when $s\geq 2$),
it was sufficient to rewrite \eqref{qnls} in a paradifferential form and treat the resulting source terms as perturbations by using a suitable array of linear and bilinear Strichartz estimates. In the low regularity regime (i.e. ,$s>\frac{3}{2}$), we have to further isolate the balanced cubic term as a non perturbative error. This is in contrast to the higher regularity regime $s \geq 2$, where $s$ is large enough to allow us to use $L^4$ Strichartz estimates \eqref{uk-Str2}
with a half derivative loss to control these terms. In the low regularity regime, our strategy for dealing with this term is mostly similar to the definite case in \cite{IT-qnls2}. We summarize it as follows:  

\begin{enumerate}[label=(\roman*)]
    \item Low frequency terms: We estimate these directly using the $L^4$ Strichartz bound (the derivative loss being negligible in this case).
    \item High frequency terms without phase rotation symmetry. These terms can be divided into two components
\begin{itemize}
 \item transversal or weakly transversal interactions, where we can split the three input frequencies and output frequency into two pairs, where at least one is unbalanced (in the sense of Definition~\ref{transversalpair}). These terms are then treated perturbatively using the balanced and unbalanced bilinear $L^2$ bounds.

\item nonresonant, which morally can be treated by a normal form analysis. Precisely, in our setting, this will be implemented via an appropriate quartic density and flux correction for the mass and the momentum. 
\end{itemize}    
\item High frequency terms with phase rotation symmetry: Here, the conservative assumption comes into play, and enables us to construct a flux correction modulo a perturbative error.
\end{enumerate}

The main goal of this section is now to establish the bounds \eqref{uk-ee}, \eqref{uab-bi-unbal}, \eqref{uab-bi-bal} and \eqref{uk-Str2}. As usual, we may harmless impose the bootstrap hypotheses \eqref{uk-ee-boot}, \eqref{uab-bi-unbal-boot}, \eqref{uab-bi-bal-boot}
and \eqref{uk-Str2-boot}.

Our starting point is to observe that the mass and momentum densities
\[
M_\lambda = M_\lambda(u,\bar u), 
\]
\[
P_{\lambda}^j = P_{\lambda}^j(u,\bar u) 
\]
satisfy the local conservation laws
\[
\partial_t M_\lambda(u) = \partial_j  P^j_{\lambda}(u)
+ C^{4,res}_{\lambda,m}(u) + C^{4,nr}_{\lambda,m}(u) +F^4_{\lambda,m}(u),
\]
respectively
\[
\partial_t P^j_{\lambda}(u) = \partial_k E^{jk}_{\lambda}(u)
+ C^{4,res}_{\lambda,p}(u) + C^{4,nr}_{\lambda,p}(u)+F^4_{\lambda, p}(u).
\]
In the above, the term $C^{4,res}$ denotes the contribution of the balanced resonant cubic terms satisfying phase rotation symmetry. On the other hand, the term $C^{4,nr}$ contains the  
contribution of the balanced cubic nonresonant terms. As usual, the $F^4_{\lambda}$ terms represent the remaining terms in the above expansions, which we will directly treat as perturbations. We remark that this term contains (among other things) the weakly transversal cubic part of the nonlinearity, which did not exist in the elliptic case. To simplify the analysis, we further remark that this refined decomposition is only necessary at high frequency $\lambda \gg 1$, as in the low-frequency regime $\lambda \lesssim 1$, one can directly treat the balanced quartic terms directly using the $L^4$ bootstrap bound \eqref{uk-Str2-boot}.

The contributions of the $F^4_{\lambda}$ terms can be estimated directly from Lemma~\ref{l:Fmp}, where we have the perturbative bounds
\begin{equation}\label{pert-flux-source}
 \|F^4_{\lambda,m}(u) \|_{L^1_{t,x}} \lesssim \epsilon^4 C^4 c_\lambda^4 \lambda^{-2s},
 \qquad
 \|F^4_{\lambda,p}(u) \|_{L^1_{t,x}} \lesssim \epsilon^4 C^4 c_\lambda^4
\lambda^{-2s+1}.
\end{equation}

In sharp contrast, we cannot directly treat the balanced terms $C^{4,res}_{\lambda}$ and $C^{4,nr}_{\lambda}$ in $L^1$. Instead, we will need the more refined analysis below.
\subsection{The resonant balanced term} In this subsection, we analyze the finer structure of the term $C^{4, res}$. For simplicity, we will mainly study $C^{4,res}_{\lambda,m}$, as $C^{4,res}_{\lambda,p}$ will be treated by similar analysis. Here, we will crucially rely on the conservative assumption, which enforces the vanishing condition on the diagonal for the corresponding symbols
\begin{equation}\label{c4-conserv}
c^{4,res}_{\lambda,m}(\xi,\xi,\xi,\xi) = 0, 
\qquad 
c^{4,res}_{\lambda,p}(\xi,\xi,\xi,\xi) = 0, \quad \xi \in \R^2.
\end{equation}
This yields the following symbol expansion for $c_{\lambda,m}^{4,res}$,
\[
\lambda^{-1} c^{4,res}_{\lambda,m}(\xi_1,\xi_2,\xi_3,\xi_4)
= i(\xi_{odd} -\xi_{even}) r^{4,res}_{\lambda,m}(\xi_1,\xi_2,\xi_3,\xi_4) + i\Delta^4 \xi q^{4,res}_{\lambda,m}
(\xi_1,\xi_2,\xi_3,\xi_4),
\]
where all symbols on the right-hand side above are smooth and bounded.
Separating variables in $r^{4,res}_{\lambda, m}$, we arrive at the following decomposition 

\begin{lemma}
The quartic form $C^{4, res}_{\lambda,m}$ admits a representation
of the form  
\begin{equation}
\lambda^{-1} C^{4,res}_{\lambda,m}(u,\bu,u,\bu) = \partial_x
 Q^{4,res}_m(u,\bu,u,\bu) + \sum \partial_x R^2_{j,a}(u,\bu)
 R^2_{j,b} (u,\bu),
\end{equation}
where the sum is rapidly convergent in $j$.
\end{lemma}
A simple consequence of this bound is the estimate 
\begin{equation} \label{en-bal}
\left| \int  C^{4,res}_{\lambda,m}(u,\bu,u,\bu)\, dx dt \right|
\lesssim \sum_j \lambda \| |D|^\frac12 R^2_{j,a}(u,\bu)\|_{L^2_{t,x}}
\| |D|^\frac12 R^2_{j,b}(u,\bu)\|_{L^2_{t,x}},
\end{equation}
To estimate each summand on the right-hand side above, we can split $R_{j,a}^2$ (and identically $R_{j,b}^2$) into low and high-frequency parts. Precisely, for the first term, we have
\begin{equation*}
R_{j,a}^2(u,u)=P_{<1}R^2_{j,a}(u,u)+P_{\geq 1}R^2_{j,a}(u,u).
\end{equation*}
For the low-frequency part, we can interpolate between the bilinear balanced bound \eqref{uab-bi-bal-boot} and an arbitrarily small factor of the $L^4_{t,x}$ Strichartz bound \eqref{uk-Str2} to obtain for instance,
\begin{equation*}
\|P_{<1}R^2_{j,a}(u,u)\|_{L^2_{t,x}}\lesssim_{\delta} \sup_{j\in\mathbb{Z}}\|P_j|D|^{\frac{1}{2}}R^2_{j,a}(u,u)\|_{L_{t,x}^2}^{1-\delta}\|u_{\lambda}\|^{2\delta}_{L_{t,x}^4}\leq C^2\epsilon^2c_{\lambda}^2\lambda^{-2s+\frac{1+\delta}{2}}
\end{equation*}
For the high-frequency part, using $\eqref{uab-bi-bal-boot}$ we have instead
\begin{equation*}
\|P_{\geq 1}|D|^{\frac{1}{2}}R_{j,a}^2(u,u)\|_{L_{t,x}^2}\lesssim_{\delta} \lambda^{\frac{\delta}{2}}\sup_{j\in\mathbb{Z}}\|P_j|D|^{\frac{1}{2}}R^2_{j,a}(u,u)\|_{L_{t,x}^2}\leq C^2\epsilon^2c_{\lambda}^2\lambda^{-2s+\frac{1+\delta}{2}}
\end{equation*}
which (coupled with the identical estimates for $R_{j,b}^2(u,u)$) is more than good enough (in view of the restriction on $s$) to control the right-hand side of \eqref{en-bal}. This entirely suffices for energy estimates, but not for interaction Morawetz estimates. To deal with this, we will instead expand the terms in the above series using a standard paraproduct decomposition 
\[
\begin{aligned}
\partial_x R^2_{j,a}(u,\bu) R^2_{j,b} (u,\bu)
= & \ 
T_{\partial_x R^2_{j,a}(u,\bu)} R^2_{j,b} (u,\bu)
+ \Pi(\partial_x R^2_{j,a}(u,\bu), R^2_{j,b} (u,\bu))
\\ &\ - T_{\partial_x R^2_{j,b}(u,\bu)} R^2_{j,a} (u,\bu)
 + \partial_x (T_{ R^2_{j,b} (u,\bu)} R^2_{j,a}(u,\bu)).
\end{aligned}
\]
The first three terms can be treated directly in $L^1$ and then further estimated as with the right-hand side of \eqref{en-bal} above. On the other hand, the last term in the decomposition, along with $Q_m^{4, res}$ will go into a suitable flux correction. A similar analysis yields an analogous decomposition for $C_{\lambda,p}^{4,res}$. We summarize our findings below,
\begin{proposition}\label{p:bal-sources}
Assume that $c^{4,res}_{\lambda, m}$ satisfies \eqref{c4-conserv}. 
Let $u$ satisfy \eqref{uk-ee-boot} and \eqref{uab-bi-bal-boot}.
Then there exists a decomposition
\begin{equation}
 C^{4,res}_{\lambda,m}(u) = \partial_x Q^{4,res}_{\lambda,m}(u) +  F^{4,res}_{\lambda,m}(u) ,\hspace{5mm}  C^{4,res}_{\lambda,p}(u)=\partial_x Q_{\lambda,p}^{4,res}(u)+F_{\lambda,p}^{4,res}
\end{equation}
so that 
\begin{equation}\label{flux-bal}
\|Q^{4,res}_{\lambda,m}(u)\|_{L^\frac65_t L^\frac43_x} \lesssim C^4 \epsilon^4 c_\lambda^4 \lambda^{1+\frac56-4s},\hspace{5mm}\|Q^{4,res}_{\lambda,p}(u)\|_{L^\frac65_t L^\frac43_x}\lesssim C^4 \epsilon^4 c_\lambda^4 \lambda^{2+\frac56-4s}.
\end{equation}
while
\begin{equation}\label{error-cor-res}
\|  F^{4,res}_{\lambda,m}(u) \|_{L^1_{t,x}} \lesssim C^4 \epsilon^4 c_\lambda^4 \lambda^{-2s},\hspace{5mm} \|  F^{4,res}_{\lambda,p}(u) \|_{L^1_{t,x}} \lesssim C^4 \epsilon^4 c_\lambda^4 \lambda^{-2s+1}.  
\end{equation} 
\end{proposition}
The first two bounds above in \eqref{flux-bal} follow from interpolating between the energy bound \eqref{uk-ee-boot} and the balanced bilinear $L^2$ bound \eqref{uab-bi-bal-boot}. The last two bounds in \eqref{error-cor-res}, on the other hand,  are obtained using an analysis similar to the right-hand side of \eqref{en-bal}, as well as the restriction $s>\frac{3}{2}$ (in fact, $s>1$ is enough here).
\subsection{The balanced cubic nonresonant term} To treat $C^{4,nr}$, our strategy will be to implement a normal form type correction to both the densities and the fluxes:

\begin{proposition}\label{p:nr-sources}
 
Let $u$ satisfy \eqref{uk-ee-boot} and \eqref{uab-bi-bal-boot}.
Then there exist decompositions
\begin{equation}
 C^{4,nr}_{\lambda,m}(u) = - \partial_t B^{4,nr}_{\lambda,m}(u)
 + \partial_x Q^{4,nr}_{\lambda,m}(u) +  F^{6,nr}_{\lambda,m}(u)  ,
\end{equation}
\begin{equation}
C^{4,nr}_{\lambda,p}(u) = - \partial_t B^{4,nr}_{\lambda,p}(u)
 + \partial_x Q^{4,nr}_{\lambda,p}(u) +  F^{6,nr}_{\lambda,p}(u)
\end{equation}
so that we have the fixed time bounds
\begin{equation}\label{energy-cor-nr}
\|  B^{4,nr}_{\lambda,m}(u) \|_{L^1_x} \lesssim C^4 \epsilon^4 c_\lambda^4 \lambda^{2-4s},\hspace{5mm}\|  B^{4,nr}_{\lambda,p}(u) \|_{L^1_x} \lesssim C^4 \epsilon^4 c_\lambda^4 \lambda^{3-4s}
\end{equation}
and the space-time bounds
\begin{equation}\label{flux-nr}
\|  Q^{4,nr}_{\lambda,m}(u) \|_{L^\frac65_t L^\frac43_x} \lesssim C^4 \epsilon^4 c_\lambda^4 \lambda^{1+\frac56-4s},\hspace{5mm} \|  Q^{4,nr}_{\lambda,p}(u) \|_{L^\frac65_t L^\frac43_x} \lesssim C^4 \epsilon^4 c_\lambda^4 \lambda^{2+\frac56-4s} ,
\end{equation}
respectively
\begin{equation}\label{error-cor-nr}
\|  F^{6,nr}_{\lambda,m}(u) \|_{L^1_{t,x}} \lesssim C^4 \epsilon^4 c_\lambda^4 \lambda^{-2s},\hspace{5mm}\|  F^{6,nr}_{\lambda,p}(u) \|_{L^1_{t,x}} \lesssim C^4 \epsilon^4 c_\lambda^4 \lambda^{1-2s}    
\end{equation}
\end{proposition}
\begin{proof} The proof is completely analogous to the computation in the definite case in \cite{IT-qnls2}. Therefore, we omit the details.
\end{proof}

To summarize our findings in the previous two subsections, we observe that the modified densities 
\begin{equation}\label{ms}
\ms_\lambda(u) = M_\lambda(u) + B^{4,nr}_{\lambda,m}(u),    
\end{equation}
respectively
\begin{equation}\label{ps}
\ps_\lambda(u) = P_\lambda(u) + B^{4,nr}_{\lambda,p}(u),    
\end{equation}
satisfy the associated density-flux identities
\begin{equation}\label{modifiedmassflux}
\partial_t \ms_\lambda(u) = \partial_j  (P^j_{\lambda}(u)+ Q^{4,j}_{\lambda,m}+Q^{4,nr,j}_{\lambda,m}) +F^{4,res}_{\lambda,m}(u)+F^{6,nr}_{\lambda,m}(u)
+F^4_{\lambda,m}(u),
\end{equation}
respectively
\begin{equation}\label{modifiedmomentumflux}
\partial_t {\ps}^j_{\lambda}(u) = \partial_k (E^{jk}_{\lambda}(u)
+ Q^{4,res,jk}_{\lambda,p}+Q^{4,nr,jk}_{\lambda,p}) +F^{4,res,j}_{\lambda,m}(u)
+F^{6,nr}_{\lambda,p}(u) +F^{4,j}_{\lambda,p}(u),
\end{equation}
where the flux corrections $Q^{4, res}$,  $Q^{4,nr}$  are as in \eqref{flux-bal}, respectively \eqref{flux-nr},
and all the $F^4$ and $F^6$ terms are estimated in $L^1$ as in \eqref{pert-flux-source}, \eqref{error-cor-res}, \eqref{energy-cor-nr} and \eqref{error-cor-nr}.

\subsection{The uniform energy bounds}
Here we integrate the density-flux identities \eqref{modifiedmassflux} and \eqref{modifiedmomentumflux} using the bounds in Proposition~\ref{p:bal-sources}  and Proposition~\ref{p:nr-sources}
in order to prove the dyadic energy estimates \eqref{uk-ee} in our 
main result in Theorem~\ref{t:global2c}, under the appropriate bootstrap assumptions as stated in Proposition~\ref{p:boot}. As in \cite{IT-qnls2} (and observed above), the mass and momentum corrections 
play a perturbative role at fixed time in this computation,
\begin{equation}\label{B4-pert}
\begin{aligned}
\|   B^{4,nr}_{\lambda,m}(u) \|_{L^1} \lesssim C^4 \epsilon^4 c_\lambda^4 \lambda^{-4s+2}
\lesssim \epsilon^2 c_\lambda^2 \lambda^{-2s},
\\
\|   B^{4,nr}_{\lambda,p}(u) \|_{L^1} \lesssim C^4 \epsilon^4 c_\lambda^4 \lambda^{-4s+3}
\lesssim \epsilon^2 c_\lambda^2 \lambda^{-2s+1}.
\end{aligned}
\end{equation}

\subsection{The interaction Morawetz identities}

Here we prove the balanced and unbalanced bilinear $L^2$ bounds \eqref{uab-bi-bal}
and \eqref{uab-bi-unbal}, respectively. In doing so, we will finally conclude the proof of Theorem~\ref{t:global2c}.

As in \cite{IT-qnls2}, we will aim to pair a frequency $\lambda$ 
portion of one solution $u$ with a frequency $\mu$ portion 
of another solution $v$, which will ultimately be taken to be a translate of $u$. That is, 
$v = u^{x_0}$. The leading computation here is very similar to the analysis in Section~\ref{s:para}, with the difference now being that we have to treat the additional correction terms above.

In the sequel, we will write the requisite identities in a general form, and then later, to establish Theorem~\ref{t:global2c} we will specialize to three cases:
\begin{enumerate}
    \item The diagonal case $\lambda = \mu$, $u=v$ (and thus $x_0=0$).
    \item The balanced shifted case $\lambda \approx \mu$, 
    with  $x_0$ arbitrary.
\item The unbalanced case $ \mu \leq \lambda$, 
    with  $x_0$ arbitrary.
\end{enumerate}
As in Section~\ref{s:para}, the choice of Morawetz functional in our analysis will differ depending on whether we consider balanced interactions (corresponding to the first two cases above) or unbalanced interactions (which correspond to the third case). Below, we carry out the analysis for the balanced cases, and later outline the minor changes needed to deal with the unbalanced case.
\subsection{The balanced case}
Following the analysis of Section~\ref{s:para}, we define the one-parameter family of interaction Morawetz functionals associated with the functions
$(u,v)$ and frequencies $(\lambda,\mu)$ as
\begin{equation}\label{IM}
\bI_{r,\lambda\mu}(u,v) :=   \iint a_{r,j}(x-y) (\ms_\lambda(u)(x) {\ps}^j_{\mu}(v) (y) -  
{\ps}^j_{\lambda}(u)(x) \ms_{\mu}(v) (y) \, dx dy,\hspace{5mm}r>0,
\end{equation}
where $a_r$ is defined by \eqref{rweightdef}. To compute the time derivative of $\bI_{r,\lambda\mu}$, we use the modified density-flux identities \eqref{modifiedmassflux} and \eqref{modifiedmomentumflux},
\begin{equation}\label{interaction-xilm}
\frac{d}{dt} \bI_{r,\lambda\mu} =  \bJ^4_{r,\lambda\mu} + \bJ^{6}_{r,\lambda\mu} + \bK_{r,\lambda\mu}. 
\end{equation}
Most of the resulting terms are analyzed similarly to the corresponding expression in Section~\ref{s:para} for the case of the paradifferential equation. However, as in \cite{IT-qnls2}, there are two exceptional terms that require further explanation:
\medskip

(i) the mass and momentum corrections in $\bI_{r, \lambda\mu}$, whose contributions are easily handled by the fixed-time estimates \eqref{B4-pert}.

\medskip
(ii) the expression  $\bJ^{6}_{r,\lambda \mu}$,  which contains the contributions 
of the density and flux corrections,
{\small
\begin{equation}\label{J6cor}
\begin{split}
\bJ^{6,bal}_{r,\lambda \mu} := & \ \iint a_{r,jk}(x-y) \left((Q^{4,k}_{\lambda,m}(u))(x){\ps}^j_{\mu}(v) (y) -  
(Q^{4,jk}_{\lambda,p}(u))(x) \ms_{\mu}(v) (y) \, \right. 
\\ & \left. - \ms_\lambda(u)(x) (Q^{4,jk}_{\mu,p}(v)) (y) +  
{\ps}^j_{\lambda}(u)(x) (Q^{4,k}_{\mu,m}(v))(y)\right) \, dx dy.
\end{split}
\end{equation}
}
where we have written for simplicity of notation, $Q^{4}:=Q^{4,res}+Q^{4,nr}$. To obtain the same estimate as in the paradifferential case, it suffices to show that we have the (uniform in $r$) bound
\begin{equation}\label{J6-bal-want}
 \left| \int_0^T    \bJ^{6,bal}_{r,\lambda \mu} \, dt \right|
 \lesssim \epsilon^4 c_\lambda^2 c_\mu^2 \lambda^{1-2s}\mu^{-2s}.
\end{equation}
At this point, the analysis is similar to \cite{IT-qnls2}. As a model example, we consider the second term in the above expression. Here, using the uniform in $r$ bound $|D^2a_r(x-y)| \lesssim \dfrac{1}{|x-y|} $, we can apply Young's inequality to control
\[
 \left| \int_0^T  \iint a_{r,jk}(x-y)Q^{4,jk}_{\lambda,p}(u)(x) \ms_{\mu}(v) (y)  \, dxdydt \right| \lesssim \| Q^{4,jk}_{\lambda,p}(u)\|_{L^\frac65_t L^\frac43_x} \| M_{\mu}(v)\|_{L^6_t L^\frac43_x},
\]
To estimate $\ms_\mu(v)$, we can interpolate the energy bound in $L^\infty_t L^2_x$ bound provided by \eqref{uk-ee-boot}
and the $L^6_t L^4_x$ bound \eqref{uk-l4-2d} to 
obtain
\[
\| \ms_{\mu}(v)\|_{L^6_t L^\frac43_x} \lesssim C^2 \epsilon^2 c_\mu^2 \mu^{\frac16-2s},
\]
where the quartic correction has a better estimate than the leading quadratic term.

Combining this with \eqref{flux-bal} or \eqref{flux-nr} we then obtain
\begin{equation}
\lesssim C^6 \epsilon^6 c_\lambda^4 c_\mu^2 \lambda^{2+\frac56-4s}
\mu^{\frac16-2s}
\end{equation}
which is better than the required bound \eqref{J6-bal-want} provided that $s > \frac32$ (in fact $s\geq 1$ suffices). Similar estimates hold for the other components of $\bJ_{r,\lambda\mu}^{6,bal}$. Combining the estimates above establishes the required balanced bilinear estimates in Theorem~\ref{t:global2c}.
\subsection{The unbalanced case} For the unbalanced case, we again take our cue from Section ~\ref{s:para}, where we instead define our interaction Morawetz functional as
\begin{equation}\label{unbIMcor}
\bI_{\lambda\mu}(u,v) :=
 \iint_{x_1< y_1}  \ms_{\lambda}(u)(x)\ms_{\mu}(v) (y) \, dx dy
\end{equation}
Here, as in Section~\ref{s:para}, we have distinguished (without loss of generality) $e_1$ as the direction in which we have $O(\lambda)$ frequency separation for the corresponding group velocities. This in turn yields an analogous decomposition of the time derivative,
\begin{equation}\label{tdunbIMcor}
\frac{d}{dt} \bI_{\lambda\mu}(u,v) =  \bJ_{\lambda\mu}^4(u,v) +\bJ_{\lambda\mu}^6(u,v)  +\bK_{\lambda\mu}(u,v), 
\end{equation}
As in the balanced case, the leading part in the above decomposition can be estimated exactly as in Section ~\ref{s:para}. Moreover, we can again estimate the contributions of the density-flux corrections in $\bI_{\lambda\mu}$ at a fixed time using \eqref{B4-pert}. The analogue of \eqref{J6cor} which corresponds to the part of $\bJ_{\lambda\mu}^6(u,v)$ containing the contributions of the density-flux corrections has the form
\begin{equation}\label{J6bal2}
\bJ_{\lambda\mu}^{6,bal}(u,v):=\iint_{x_1=y_1}(Q^{4,1}_{\lambda,m}(u))(x)\ms_{\mu}(v) (y)-\ms_{\lambda}(u)(x)(Q^{4,1}_{\mu,m}(v))(y)dxdy
\end{equation}
The required analogue of the bound \eqref{J6-bal-want} that we want in this case is
\begin{equation}\label{J6-bal-want2}
 \left| \int_0^T    \bJ^{6,bal}_{r,\lambda \mu} \, dt \right|
 \lesssim \epsilon^4 c_\lambda^2 c_\mu^2 \lambda^{-2s}\mu^{-2s}.
\end{equation}
To estimate the first term in \eqref{J6bal2} by the right-hand side of \eqref{J6-bal-want2}, we first estimate
\begin{equation*}
\begin{split}
\int_{0}^{T}\iint_{x_1=y_1}|(Q_{\lambda,m}^{4,1}(u))(x)\ms_{\mu}(v) (y)|dxdydt&\lesssim \|Q_{\lambda,m}^4\|_{L_t^1L_{x_2}^1L_{x_1}^{\infty}}\|\ms_{\mu}(v)\|_{L_t^{\infty}L_x^1}
\\
&\lesssim C\epsilon^2c_{\mu}^2\mu^{-2s}\|Q_{\lambda,m}^4\|_{L_t^1L_{x_2}^1L_{x_1}^{\infty}}.
\end{split}
\end{equation*}
To deal with the quartic term $Q_{\lambda,m}^4$, we use the representation $Q_{\lambda,m}^4=\lambda L(u_{\lambda},u_{\lambda},u_{\lambda},u_{\lambda})$, which by Bernstein in $x_1$ and the analysis following \eqref{en-bal} enables us to estimate
\begin{equation*}
\|Q_{\lambda,m}^4\|_{L_t^1L_{x_2}^1L_{x_1}^{\infty}}\lesssim \lambda\||u_{\lambda}|^2\|_{L_t^2L_{x_2}^2L_{x_1}^{\infty}}^2\lesssim \lambda\||D|^{\frac{1}{2}}|u_{\lambda}|^2\|_{L_t^2L_x^2}^2\lesssim C^4\epsilon^4c_{\lambda}^4\lambda^{-4s+2+\delta},
\end{equation*}
which is more than sufficient. The other term in \eqref{J6bal2} is handled similarly (by simply swapping the roles of $\mu$ and $\lambda$ in the above estimates).
\subsection{ The $L^4$ Strichartz bounds}
Here we prove \eqref{uk-Str2}, thereby concluding the proof of Theorem~\ref{t:global2c}.
We follow the proof of Theorem~\ref{t:para-se} with the difference  that here we have already established the uniform energy bounds for $u_\lambda$. By interpolation, it remains to show the $V^2_{UH}$ bound with a derivative loss, 
\begin{equation}\label{ul-V2}
\|u_\lambda \|_{V^2_{UH}} \lesssim \epsilon c_\lambda \lambda^{-s+1}    
\end{equation}
For this we would like to apply \eqref{Str-2}, except that we do not have a good estimate for the last term on the right.  However, repeating the argument there we can 
write
\begin{equation}\label{para-re-cc+}
(i \partial_t + \Delta_{g(0)}) u_\lambda 
= N_\lambda + f_\lambda^1, \qquad u_\lambda(0) = u_{0,\lambda},
\end{equation}
where
\begin{equation}
   f_\lambda^1 :=  \partial_j (g^{jk}_{[<\lambda]}-g^{jk}(0)) \partial_k u_\lambda,
\end{equation}
and estimate directly 
\[
\| u_\lambda \|_{V^2_{UH}} \lesssim \| v_\lambda(0)\|_{L^2_x}
+ \| N_\lambda\|_{DV^2_{UH}} 
+ \| f_\lambda^1\|_{DV^2_{UH}},
\]
Here the second term on the right is estimated in $L^\frac43 \subset DV^2_{UH}$, separating $F_\lambda$ and $C_\lambda$. For $C_\lambda$ we use 
three $L^4$ bounds \eqref{uk-Str2-boot},
\[
\| C_\lambda(u)\|_{L^\frac43} \lesssim C^3 \epsilon^3 c_\lambda^3 \lambda^{-3s+\frac72}
\]
which suffices. For $F_\lambda$ we have the two highest frequencies comparable and at least $\lambda$, while a third one is smaller. Hence we use an unbalanced bilinear $L^2$ bound \eqref{uab-bi-unbal-boot} and an $L^4$ bound for the second high frequency, to obtain
\[
\|F_\lambda(u)\|_{L^\frac43} \lesssim C^3 \epsilon^3 c_\lambda^2 \lambda^{-2s+2}
\]
which again suffices.

Finally the last term on the right is estimated by duality, testing against a function $w \in U^2_{UH}$,
\[
\left| \int f_\lambda^1 \bar w_\lambda dx dt \right|
\lesssim \| f_\lambda^1 \bar w_\lambda \|_{L^1}  \lesssim \lambda^2 \| L(\bfu_{\ll\lambda}, u_\lambda)\|_{L^2} \| L(\bfu_{\ll \lambda}, w_\lambda)\|_{L^1}
\lesssim C^3 \epsilon^3 \lambda c_\lambda \lambda^{-s+1},
\]
where at the last step we have used two bilinear unbalanced $L^2$ bounds, namely 
one instance of \eqref{uab-bi-unbal-boot} and one application of Corollary~\ref{c:para}
after using the atomic structure of $U^2_{UH}$ to reduce to the case when $w$ is a solution. This has to be done in the context of this section, using the modified density flux identities for $u_\lambda$.

This concludes the proof of \eqref{ul-V2} and thus the proof of the theorem.

\bibliography{1d-global}

\bibliographystyle{plain}

\end{document}